\documentclass{amsart}

\usepackage{amssymb,amsfonts, amsmath, amsthm}
\usepackage{xypic}
\usepackage[active]{srcltx}

\newcommand{\ds}{\,ds}

\newcommand{\comment}[1]{}

\newcommand{\C}{\mathbb{C}}

\newcommand{\rt}{\rtimes}
\newcommand{\eps}{\varepsilon}
\newcommand{\supp}{\text{supp}}
\newcommand{\norm}[1]{\left\Vert #1 \right\Vert}
\newcommand{\normr}[1]{\left\Vert #1 \right\Vert^{\mathcal R}}
\newcommand{\brackets}[1]{\left( #1 \right)}
\newcommand{\integrate}{{\mathcal I}^\mr}
\newcommand{\separate}{{\mathcal S}^\mr}
\newcommand{\intform}{\pi \rtimes U}
\newcommand{\aintform}{U \rtimes \pi}
\newcommand{\intformr}{(\pi \rtimes U)^\mathcal R}

\newcommand{\Cr}{{C^\mr}}
\newcommand{\mr}{\mathcal R}

\newcommand{\ma}{\mathcal A}
\newcommand{\qr}{q^{\mathcal R}}
\newcommand{\nur}{\nu^\mr}
\newcommand{\Nr}{N^\mr}
\newcommand{\sigmar}{\sigma^{\mathcal R}}
\newcommand{\crosprod}{(A \rt_\alpha G)^\mathcal{R}}
\newcommand{\Aut}{\textup{Aut}}
\newcommand{\Inv}{\textup{Inv}}
\newcommand{\dynsys}{(A,G,\alpha)}
\newcommand{\covrep}{(\pi,U)}
\newcommand{\Covrep}{\textup{Covrep}}

\newcommand{\Repndb}{\ensuremath{\textup{Rep}_{\textup{nd,b}}}}
\newcommand{\bclass}{\mathcal X}

\newcommand{\Covreprndc}{\ensuremath{\Covrep_{\textup{nd,c}}^\mr}}
\newcommand{\LCA}{{\mathcal M}_l(\mathcal A)}
\newcommand{\RCA}{{\mathcal M}_r(\mathcal A)}
\newcommand{\DCA}{{\mathcal M}(\mathcal A)}

\newcommand{\id}{\textup{id}}
\newcommand{\triv}{\textup{triv}}

\newcommand{\End}{\textup{End\,}}

\theoremstyle{plain}

\newtheorem{theorem}{Theorem}[section]
\newtheorem{prop}[theorem]{Proposition}
\newtheorem{lemma}[theorem]{Lemma}
\newtheorem{corol}[theorem]{Corollary}

\theoremstyle{definition}

\newtheorem{defn}[theorem]{Definition}

\newtheorem{remark}[theorem]{Remark}

\numberwithin{equation}{section}

\begin{document}

\title{Crossed products of Banach algebras. I.}

\author{Sjoerd Dirksen}
\address{Sjoerd Dirksen, Delft Institute of Applied Mathematics, Delft University of Technology, P.O.\ Box 5031, 2600 GA Delft, The Netherlands}
\email{s.dirksen@tudelft.nl}

\author{Marcel de Jeu}
\address{Marcel de Jeu, Mathematical Institute, Leiden University, P.O.\ Box 9512, 2300 RA Leiden, The Netherlands}
\email{mdejeu@math.leidenuniv.nl}

\author{Marten Wortel}
\address{Marten Wortel, Mathematical Institute, Leiden University, P.O.\ Box 9512, 2300 RA Leiden, The Netherlands}
\email{wortel@math.leidenuniv.nl}




\begin{abstract}
We construct a crossed product Banach algebra from a Banach algebra dynamical system $\dynsys$ and a given uniformly bounded class $\mr$ of continuous covariant Banach space representations of that system. If $A$ has a bounded left approximate identity, and $\mr$ consists of non-degenerate continuous covariant representations only, then the non-degenerate bounded representations of the crossed product are in bijection with the non-degenerate $\mr$-continuous covariant representations of the system. This bijection, which is the main result of the paper, is also established for involutive Banach algebra dynamical systems and then yields the well-known representation theoretical correspondence for the crossed product $C^*$-algebra as commonly associated with a $C^*$-algebra dynamical system as a special case. Taking the algebra $A$ to be the base field, the crossed product construction provides, for a given non-empty class of Banach spaces, a Banach algebra with a relatively simple structure and with the property that its non-degenerate contractive representations in the spaces from that class are in bijection with the isometric strongly continuous representations of $G$ in those spaces. This generalizes the notion of a group $C^*$-algebra, and may likewise be used to translate issues concerning group representations in a class of Banach spaces to the context of a Banach algebra, simpler than $L^1(G)$, where more functional analytic structure is present.
\end{abstract}

\maketitle

\section{Introduction}\label{sec:intro}

The theory of crossed products of $C^*$-algebras started with the papers by Turumaru \cite{turumaru} from 1958 and Zeller-Meier from 1968 \cite{zeller-meier}. Since then the theory has been extended extensively, as is attested by the material in Pedersen's classic \cite{pedersen} and, more recently, in Williams' monograph \cite{williams}.
Starting with a $C^*$-dynamical system $\dynsys$, where $A$ is a $C^*$-algebra, $G$ is a locally compact group, and $\alpha$ a strongly continuous action of $G$ on $A$ as involutive automorphisms, the crossed product construction yields a $C^*$-algebra $A\rt_\alpha G$ which is built from these data. Thus the crossed product construction provides a means to construct examples of $C^*$-algebras from, in a sense, more elementary ingredients.
One of the basic facts for a crossed product $C^*$-algebra $A\rt_\alpha G$ is that the non-degenerate involutive representations of this algebra on  Hilbert spaces are in one-to-one correspondence with the non-degenerate involutive continuous covariant representations of $\dynsys$, i.e., with the pairs $\covrep$, where $\pi$ is a non-degenerate involutive representation of $A$ on a Hilbert space, and $U$ is a unitary strongly continuous representations of $G$ on the same space, such that the covariance condition $\pi(\alpha_s(a))=U_s\pi(a)U_s^{-1}$ is satisfied, for $a\in A$, and $s\in G$.

This paper contains the basics for the natural generalization of this construction to the general Banach algebra setting. Starting with a Banach algebra dynamical system $\dynsys$, where $A$ is a Banach algebra, $G$ is a locally compact group, and $\alpha$ a strongly continuous action of $G$ on $A$ as not necessarily isometric automorphisms, we want to build a Banach algebra of crossed product type from these data. Moreover, we want the outcome to be such that (suitable) non-degenerate continuous covariant representations of $\dynsys$ are in bijection with (suitable) non-degenerate bounded representations of this crossed product Banach algebra. Later in this introduction, more will be said about how to construct such an algebra, and how the construction can be tuned to accommodate a class $\mr$ of non-degenerate continuous covariant representations relevant for the case at hand. It will then also become clear what being ``suitable" means in this context. For the moment, let us oversimplify a bit and, neglecting the precise hypotheses, state that such an algebra can indeed be constructed. The precise statement is Theorem~\ref{t:bijection}, which we will discuss below.

Before continuing the discussion of crossed products of Banach algebra as such, however, let us mention our motivation to start investigating these objects, and sketch perspectives for possible future applications of our results. Firstly, just as in the case of a crossed product $C^*$-algebra, it simply seems natural as such to have a means available to construct Banach algebras from more elementary building blocks. Secondly, there are possible applications of these algebras in the theory of Banach representations of locally compact groups. We presently see two of these, which we will now discuss.

Starting with the first one, we recall that, as a special case of the correspondence for crossed product $C^*$-algebras mentioned above, the unitary strongly continuous representations of a locally compact group $G$ in Hilbert spaces are in bijection with the non-degenerate involutive representations of the group $C^*$-algebra $C^*(G)=\mathbb C\rt_\triv G$ in Hilbert spaces. It is by this fact that questions concerning, e.g.,  the existence of sufficiently many irreducible unitary strongly continuous representations of $G$ to separate its points, and, notably, the decomposition of an arbitrary unitary strongly continuous representation of $G$ into irreducible ones, can be translated to $C^*$-algebras and solved in that context \cite{dixmier}. For Banach space representations of $G$, the theory is considerably less well developed. To our knowledge, the only available general decomposition theorem, comparable to those in a unitary context, is Shiga's \cite{shiga}, stating that a strongly continuous representation of a compact group in an arbitrary Banach space decomposes in a Peter-Weyl--fashion, analogous to that for a unitary strongly continuous representation in a Hilbert space. With the results of the present paper, it is possible to construct Banach algebras  which, just as the group $C^*$-algebra, are ``tuned'' to the situation. Our main results in this direction are Theorem~\ref{t:general_bijection_trivial_algebra} and Theorem~\ref{t:group_algebra_isometric_case}. The latter, for example, yields, for any non-empty class $\bclass$ of Banach spaces, a Banach algebra $\mathcal B_\bclass(G)$, such that the non-degenerate contractive representations of $\mathcal B_\bclass(G)$ in spaces from $\bclass$ are in bijection with the isometric strongly continuous representations of $G$ in these spaces. This algebra $\mathcal B_\bclass(G)$ could be called the group Banach algebra of $G$ associated with $\bclass$, and, as will become clear in Section~\ref{subsec:trivial_algebra}, only the isometric strongly continuous representations of $G$ in the spaces from $\bclass$ are used in its construction. The analogy with the group $C^*$-algebra $C^*(G)$, which is in fact a special case, is clear. Just as is known to be the case with $C^*(G)$, one may hope that, for certain classes $\bclass$ of sufficiently well-behaved spaces, the study of $\mathcal B_\bclass(G)$ will shed light on the theory of isometric strongly continuous representations of $G$ in the spaces from $\bclass$. For comparison, we recall  the well-known fact \cite[Assertion~VI.1.32]{helemskii}, \cite[Proposition~2.1]{johnson} that the non-degenerate bounded representations of $L^1(G)$ in a Banach space are in bijection with the uniformly bounded strongly continuous representations of $G$ in that Banach space. So, certainly, there is already a Banach algebra available to translate questions concerning group representations to, but the point is that it is very complicated, simply because $L^1(G)$ apparently carries the information of all such representations of $G$ in \emph{all} Banach spaces. One may hope that, for certain choices of $\bclass$, an algebra such as $\mathcal B_\bclass(G)$, the construction of which uses no more data than evidently necessary, has a sufficiently simpler structure than $L^1(G)$ to admit the development of a reasonable representation theory, and hence for the isometric strongly continuous representations of $G$ in these spaces, thus paralleling the case where $\bclass$ consists of all Hilbert spaces and $B_\bclass(G)=C^*(G)$. Aside, let us mention that $L^1(G)$ is, in fact, isometrically isomorphic to a crossed product $(\mathbb F\rt_\triv G)^\mr$ as in the present paper, if one chooses the class $\mr$---to be discussed below---appropriately. In that case, it is possible to infer the aforementioned bijection between the non-degenerate bounded representations of $L^1(G)$ and the uniformly bounded strongly continuous representations of $G$ from Theorem~\ref{t:bijection}, due to the fact that these representations of $G$ can then be seen to correspond to the $\mr$-continuous---also to be discussed below---representations of $(\mathbb F,G,\triv)$. In view of the further increase in length of the present paper that would be a consequence of the inclusion of these and further related results, we have decided to postpone these to the sequel \cite{crossedtwo}, including only some preparations for this at the end of Section~\ref{subsec:general_algebra_and_group}.

The second possible application in group representation theory lies in the relation between imprimitivity theorems and Morita equivalence. Whereas the construction of the group Banach algebras $B_\bclass(G)$ and establishing their basic properties could be done in a paper quite a bit shorter than the present one, the general crossed product construction and ensuing results are indeed needed for this second perspective. Starting with the involutive context, we recall that Mackey's now classical result \cite{mackeyimprimitivity} asserts that a unitary strongly continuous representation $U$ of a separable locally compact group $G$ is unitarily equivalent to an induced unitary strongly continuous representation of a closed subgroup $H$, precisely when there exists a system of imprimitivity $(G/H,U,P)$ based on the $G$-space $G/H$. The separability condition of $G$ is actually not necessary, as shown by Loomis \cite{loomis} and Blattner \cite{blattner}, and for general $G$ Mackey's imprimitivity theorem
can be reformulated as (\cite[Theorem~7.18]{rieffelinducedadvances}): $U$ is unitarily equivalent to such an induced representation precisely when there exists a non-degenerate involutive representation $\pi$ of $C_0(G/H)$ in the same Hilbert space, such that $(\pi,U)$ is a covariant representation of the $C^*$-dynamical system $(C_0(G/H),G,\textup{lt})$, where $G$ acts as left translations on $C_0(G/H)$. Using the standard correspondence for crossed products of $C^*$-algebras, one thus sees that, up to unitary equivalence, such $U$ are precisely the unitary parts of the non-degenerate involutive continuous covariant representations of the $C^*$-dynamical system $(C_0(G/H),G,\textup{lt})$ corresponding to the non-degenerate involutive representations of the crossed product $C^*$-algebra $C_0(G/H)\rt_{\textup{lt}}G$. Rieffel's theory of induction for $C^*$-algebras \cite{rieffelinducedbulletin}, \cite{rieffelinducedadvances} and Morita-equivalence \cite{rieffelmorita}, \cite{rieffelkingston} allows us to follow another approach to Mackey's theorem, as was in fact done in \cite{rieffelinducedadvances}, by proving that $C_0(G/H)\rt_{\textup{lt}}G$ and $C^*(H)$ are (strongly) Morita equivalent as a starting point. This implies that these $C^*$-algebras have equivalent categories of non-degenerate involutive representations, and working out this correspondence then yields Mackey's imprimitivity theorem. For more detailed information we refer to \cite{rieffelinducedadvances}, \cite{rieffelmorita}, and \cite{rieffelkingston}, as well as (also including significant further developments) to \cite{green}, \cite{mansfield}, \cite{raeburnwilliams}, \cite{echterhoffetal}, \cite{williams} and \cite{fellanddoran}, the latter also for Banach $^*$-algebras and Banach $^*$-algebraic bundles.

The Morita theorems in a purely algebraic context are actually more symmetric than the analogous ones in Rieffel's work. We formulate part of the results for algebras over a field $k$ (cf.\ \cite[Theorem~12.12]{faith}): If $A$ and $B$ are unital $k$-algebras, then the categories of left $A$-modules and left $B$-modules are $k$-linearly equivalent precisely when there exist bimodules ${_A}P_B$ and ${_B}Q_A$, such that $P\otimes_B Q\simeq A$ as $A$-$A$-bimodules, and $Q\otimes_A P\simeq B$ as $B$-$B$-bimodules. From the existence of such bimodules it follows easily that the categories are equivalent, since equivalence are manifestly given by $M\mapsto Q\otimes_A M$, for a left $A$-module $M$, and by $N\mapsto P\otimes_B N$, or a left $B$-module $N$. The non-trivial statement is that the converse is equally true. In Rieffel's analytical context, the role of the bimodules $P$ and $Q$ for $C^*$-algebras $A$ and $B$ is taken over by so-called imprimitivity bimodules, sometimes also called equivalence bimodules. These are $A$-$B$-Hilbert $C^*$-modules (\cite[Definition~3.1]{raeburnwilliams}), and the existence of such imprimitivity bimodules (actually, exploiting duality, only one is needed, see \cite[p.~49]{raeburnwilliams}) implies that the categories of non-degenerate involutive representations of these $C^*$-algebras are equivalent \cite[Theorem~3.29]{raeburnwilliams}. In contrast with the algebraic context, the converse is not generally true (see \cite[Remark~3.15 and Hooptedoodle~3.30]{raeburnwilliams}). This has led to the distinction between strong Morita equivalence (in the sense of existing imprimitivity bimodules) and weak Morita equivalence (in the sense of equivalent categories of non-degenerate involutive representations) of $C^*$-algebras. The work of Blecher \cite{blecher}, generalizing earlier results of Beer \cite{beer}, shows how to remedy this: if one enlarges the categories, taking them to consist of all left $A$-operator modules as objects and completely bounded $A$-linear maps as morphisms, and similarly for $B$, then symmetry is restored as in the algebraic case: the equivalence of these larger categories is then equivalent with the existence of an imprimitivity bimodule, i.e., with strong Morita equivalence of the $C^*$-algebras in the sense of Rieffel. As a further step, strong Morita equivalence was developed for operator algebras (i.e., norm-closed subalgebras of $B(H)$, for some Hilbert space $H$) by Blecher, Muhly and Paulsen in \cite{blechermuhlypaulsen}. Restoring symmetry again, Blecher proved in \cite{blecheroperatoralgebras} that, for operator algebras with a contractive approximate identity, strong Morita equivalence is equivalent to their categories of operator modules being equivalent via completely contractive functors.

A part of the well-developed theory in a Hilbert space context as mentioned above has a parallel for Banach algebras and representations in Banach spaces, but, as far as we are aware, the body of knowledge is much smaller than for Hilbert spaces.\footnote{As an illustration: as far as we know, for groups, \cite{lyubich} is currently the only available book on Banach space representations.} Induction of representations of locally compact groups and Banach algebras in Banach spaces has been investigated by Rieffel in \cite{rieffelinducedbanach}, from the categorical viewpoint that, as a functor, induction is, or ought to be, an adjoint of the restriction functor. In \cite{gronbaekimprimitivity}, Gr{\o}nb{\ae}k studies Morita equivalence for Banach algebras in a context of Banach space representations, and a Morita-type theorem \cite[Corollary~3.4]{gronbaekmorita} is established for Banach algebras with bounded two-sided approximate identities: such Banach algebras $A$ and $B$ have equivalent categories of non-degenerate left Banach modules precisely when there exist non-degenerate Banach bimodules ${_A}P_B$ and ${_B}Q_A$, such that $P{\widehat{\otimes}}_B Q\simeq A$ as $A$-$A$-bimodules, and $Q{\widehat{\otimes}}_A P\simeq B$ as $B$-$B$-bimodules. In subsequent work \cite{gronbaekmoritaselfinduced}, this result is generalized to self-induced Banach algebras, and this generalization yields an imprimitivity theorem \cite[Theorem~IV.9]{gronbaekimprimitivity} in a form quite similar to Mackey's theorem as formulated by Rieffel \cite[Theorem~7.18]{rieffelinducedadvances} (i.e., with a $C_0(G/H)$-action instead of a projection valued measure), with a continuity condition on the action of $C_0(G/H)$. The approach of this imprimitivity theorem, via Morita equivalence of Banach algebras, is therefore analogous to Rieffel's work, and here again algebras which are called crossed products make their appearance \cite[Definition~IV.1]{gronbaekimprimitivity}. Given the results in the present paper, it is natural to ask whether this imprimitivity theorem (or a variation of it) can also be derived from a surmised Morita equivalence of the crossed product Banach algebra $(C_0(G/H)\rt_{\textup{lt}} G)^\mr$ and a group Banach algebra $B_\bclass(H)$ as in the present paper (for suitable $\mr$ and $\bclass$), and what the relation is between the algebras in  \cite[Definition~IV.1]{gronbaekimprimitivity}, also called crossed products, and the crossed product Banach algebras in the present paper. We expect to investigate this in the future, also taking the work of De Pagter and Ricker \cite{depagterricker} into account. In that paper, it is shown that, for certain bounded Banach space representations (including all bounded representations in reflexive spaces\footnote{In fact: including all bounded representations in spaces not containing a copy of $c_0$, see \cite[Corollary~2.16]{depagterricker}.}) of $C(K)$, where $K$ is a compact Hausdorff space, there is always an underlying projection valued measure. In such cases, if $G/H$ is compact (and it is perhaps not overly optimistic to expect that the results in \cite{depagterricker} can be generalized to the locally compact case, so as to include representations of $C_0(G/H)$ for non-compact $G/H$), an imprimitivity theorem for Banach space representations of groups can be derived in Mackey's original form in terms of systems of imprimitivity. If all this comes to be, then this would be a satisfactory parallel---for suitable Banach spaces---with the Hilbert space context, both in the spirit of Rieffel's strong Morita equivalence of $C_0(G/H)\rt_{\textup{lt}}G$ and $C^*(H)$ as a means to obtain an imprimitivity theorem, and of Mackey's systems of imprimitivity as a means to formulate such a theorem. We hope to be able to report on this in due time.

\medskip
We will now outline the mathematical structure of the paper. Although the crossed product of a general Banach algebra is more involved than its $C^*$-algebra counterpart, the reader may still notice the evident influence of \cite{williams} on the present paper. We start by explaining how to construct the crossed product. Given a Banach algebra dynamical system $\dynsys$ (Definition~\ref{d:banach_algebra_dynamical_system}), and a non-empty class $\mr$ of continuous covariant representations (Definition~\ref{d:covariant_representation}), we want to introduce an algebra seminorm $\sigmar$ on the twisted convolution algebra $C_c(G,A)$ by defining
\[
\sigmar(f)=\sup_{\covrep\in\mr}\left\Vert\int_G \pi(f(s))U_s\ds\right\Vert\quad(f\in C_c(G,A)).
\]
For a $C^*$-dynamical system, if one lets $\mr$ consist of all pairs $\covrep$ where $\pi$ is involutive and non-degenerate, and $U$ is unitary and strongly continuous, this supremum is evidently finite, and $\sigmar$ is even a norm. For a general Banach algebra dynamical system, neither need be the case. This therefore leads us, first of all, to introduce the notion of a uniformly bounded (Definition~\ref{d:uniformly_bounded_class}) class of covariant representations, in order to ensure the finiteness of $\sigmar$. The resulting crossed product Banach algebra $\crosprod$ is then, by definition, the completion of $C_c(G,A)/\textup{ker\,}(\sigmar)$ in the algebra norm induced by $\sigmar$ on this quotient. Thus, as a second difference with the construction of the crossed product $C^*$-algebra associated with a $C^*$-dynamical system, a non-trivial quotient map is inherent in the construction. 

While the construction is thus easily enough explained, the representation theory, to which we now turn, is more involved. Suppose that $\covrep$ is a continuous covariant representation of $\dynsys$, and that there exists $C\geq 0$, such that $\norm{\int_G \pi(f(s))U_s\ds}\leq C\sigmar(f)$, for all $f\in C_c(G,A)$. In that case, we say that $\covrep$ is $\mr$-continuous, and it is clear that there is an associated bounded representation of $\crosprod$, denoted by $\intformr$. Certainly all elements of $\mr$ are $\mr$-continuous, yielding even contractive representations of $\crosprod$, but, as it turns out, there may be more. Likewise, $\crosprod$ may have non-contractive bounded representations. This contrasts the analogous involutive context for the crossed product $C^*$-algebra associated with a $C^*$-dynamical system. The natural question is, then, what the precise relation is between the $\mr$-continuous covariant representations of $\dynsys$ and the bounded representations of $\crosprod$. The answer turns out to be quite simple: if $A$ has a bounded left approximate identity, and if $\mr$ consists of non-degenerate (Definition~\ref{d:covariant_representation}) continuous covariant representations only, then the map $\covrep\mapsto\intformr$ is a bijection between the non-degenerate $\mr$-continuous covariant representations of $\dynsys$ and the non-degenerate bounded representations of $\crosprod$. This is the main content of Theorem~\ref{t:bijection}. 

Establishing this, however, is less simple. The first main step to be taken is to construct any representations of the group and the algebra at all from a given (non-degenerate) bounded representation of $\crosprod$. In case of a crossed product $C^*$-algebra and involutive representations in Hilbert spaces, there is a convenient way to proceed \cite{williams}. One starts by viewing this crossed product as an ideal of its double centralizer algebra. If the involutive representation $T$ of the crossed product $C^*$-algebra is non-degenerate, then it can be extended to an involutive representation of the double centralizer algebra. Subsequently, it can be composed with existing homomorphisms of group and algebra into this double centralizer algebra, thus yielding a pair $\covrep$ of representations. These can then be shown to have the desired continuity, involutive and covariance properties and, moreover, the corresponding non-degenerate involutive representation of the crossed product $C^*$-algebra turns out to be $T$ again. For Banach algebra dynamical systems we want to use a similar circle of ideas, but here the situation is more involved. To start with, it is not necessarily true that a Banach algebra $\mathcal A$ can be mapped injectively into its double centralizer algebra $\DCA$, or that a non-degenerate representation of $\mathcal A$ necessarily comes with an associated representation of the double centralizer algebra, compatible with the natural homomorphism from $\mathcal A$ into $\DCA$. This question motivated the research leading to \cite{extendart} as a preparation for the present paper, and, as it turns out, such results can be obtained. For example, if the algebra $\mathcal A$ has a bounded left approximate identity, and a non-degenerate bounded representation of $\mathcal A$ is given, then there is an associated bounded representation of the left centralizer algebra $\LCA$ which is compatible with the natural homomorphism from $\mathcal A$ into $\LCA$, with similar results for right and double centralizer algebras.\footnote{Theorem~\ref{t:summary_for_centralizers} contains a summary of what is needed in the present paper.} If we want to apply this in our situation, then we need to show that $\crosprod$ has a bounded left approximate identity. For crossed product $C^*$-algebras, this is of course automatic, but in the present case it is not. Thus it becomes necessary to establish this independently, and indeed $\crosprod$ has a bounded approximate left identity if $A$ has one, with similar right and two-sided results. As an extra complication, since the representations of $A$ under consideration are now not necessarily contractive anymore, and the group need not act isometrically, it becomes necessary, with the future applications in Section~\ref{sec:correspondences} in mind, to keep track of the available upper bounds for the various maps as they are constructed during the process. For this, in turn, one needs an explicit upper bound for bounded left and right approximate identities in $\crosprod$. It is for these reasons that Section~\ref{sec:approximate_identities} on approximate identities in $\crosprod$ and their bounds, which is superfluous for crossed product $C^*$-algebras, is a key technical interlude in the present paper. 

After that, once we know that $\crosprod$ has a left bounded approximate identity, we can let the left centralizer algebra $\mathcal M_l(\crosprod)$ take over the role that the double centralizer algebra has for crossed product $C^*$-algebras. Given a non-degenerate bounded representation $T$ of $\crosprod$, we can now find a compatible non-degenerate bounded representation of $\mathcal M_l(\crosprod)$, and on composing this with existing homomorphisms of the algebra and the group into $\mathcal M_l(\crosprod)$, we obtain a pair $\covrep$ of representations. The continuity and covariance properties are easily established, as is the non-degeneracy of $\pi$, but as compared to the situation for crossed product $C^*$-algebras, a complication arises again. Indeed, since in that case $\mr$ consists of all non-degenerate involutive covariant representations of $\dynsys$ in Hilbert spaces, and an involutive $T$ yields and involutive $\pi$ and unitary $U$, the pair $\covrep$ is automatically in $\mr$, and is therefore certainly $\mr$-continuous. For Banach algebra dynamical systems this need not be the case, and the norm estimates in our bookkeeping, although useful in Section~\ref{sec:correspondences}, provide no rescue: one needs an independent proof to show that $\covrep$ as obtained from $T$ is $\mr$-continuous. Once this has been done, it is not overly complicated anymore to show that the associated bounded representation $\intformr$ of $\crosprod$ (which can then be defined) is $T$ again. By keeping track of invariant closed subspaces and bounded intertwining operators during the process, and also considering the involutive context at little extra cost, the basic correspondence in Theorem~\ref{t:bijection} has then finally been established.

With this in place, and also the norm estimates from our bookkeeping available, it is easy so give applications in special situations. This is done in the final section, where we formulate, amongst others, the results for group Banach algebras $B_\bclass(G)$ already mentioned above. We then also see that the basic representation theoretical correspondence for ``the'' $C^*$-crossed product as commonly associated with a $C^*$-dynamical system is an instance of a more general correspondence (Theorem~\ref{t:involutive_correspondence}), valid for $C^*$-algebras of crossed product type associated with an involutive (Definition~\ref{d:banach_algebra_dynamical_system}) Banach algebra dynamical systems $\dynsys$, provided that, for all $\eps>0$, $A$ has a $(1+\eps)$-bounded approximate left identity.

\medskip

This paper is organized as follows.

In Section~\ref{sec:preliminaries} we establish the necessary basic terminology and collect some preparatory technical results for the sequel. Some of these can perhaps be considered to be folklore, but we have attempted to make the paper reasonably self-contained, especially since the basics for a general Banach algebra and Banach space situation are akin, but not identical, to those for $C^*$-algebras and Hilbert spaces, and less well-known. At the expense of a little extra verbosity, we have also attempted to be as precise as possible, throughout the paper, by including the usual conditions, such as (strong) continuity or (in the case of algebras) non-degeneracy of representations, only when they are needed and then always formulating them explicitly, thus eliminating the need to browse back and try to find which (if any) convention applies to the result at hand. There are no such conventions in the paper. It would have been convenient to assume from the very start that, e.g., all representations are (strongly) continuous and (in case of algebras) non-degenerate, but it seemed counterproductive to do so.

Section~\ref{sec:construction_and_basic_properties} contains the construction of the crossed product and its basic properties. The ingredients are a given Banach algebra dynamical system $\dynsys$ and a uniformly bounded class $\mr$ of continuous covariant representations thereof.

Section~\ref{sec:approximate_identities} contains the existence results and bounds for approximate identities in $\crosprod$. As explained above, this is a key issue which need not be addressed in the case of crossed product $C^*$-algebras.

Section~\ref{sec:from_dynsys_to_crosprod} is concerned with the easiest part of the representation theory as considered in this paper: the passage from $\mr$-continuous covariant representations of $\dynsys$ to bounded representations of $\crosprod$. We have included results about preservation of invariant closed subspaces, bounded intertwining operators and non-degeneracy. In this section, two homomorphisms $i_A$ and $i_G$ of, respectively, $A$ and $G$ into $\End (C_c(G,A))$ make their appearance, which will later yield homomorphisms $i_A^\mr$ and $i_G^\mr$ into the left centralizer algebra $\mathcal M_l(\crosprod)$, as needed to construct a covariant representation of $\dynsys$ from a non-degenerate bounded representation of $\crosprod$. With the involutive case in mind, anti-homomorphism $j_A$ and $j_G$ into $\End (C_c(G,A))$ are also considered.

Section~\ref{sec:centralizer_algebras} on centralizer algebras starts with a review of part of the results from \cite{extendart}, and then, after establishing a separation property to be used later (Proposition~\ref{p:representations_separate_left_centralizers}), continues with the study of more or less canonical (anti-)homomorphisms of $A$ and $G$ into the left, right or double centralizer algebra of $\crosprod$. These (anti-)homomorphisms, such as $i_A^\mr$ and $i_G^\mr$ already alluded to above are based on the (anti-)homomorphisms from Section~\ref{sec:from_dynsys_to_crosprod}.

Section~\ref{sec:from_crosprod_to_dynsys} contains the most involved part of the representation theory: the passage from non-degenerate bounded representations of $\crosprod$ to non-degenerate $\mr$-continuous covariant representations of $\dynsys$. At this point, if $A$ has a bounded left approximate identity, then Sections~\ref{sec:approximate_identities} and~\ref{sec:centralizer_algebras} provide the necessary ingredients. If $T$ is a non-degenerate bounded representation of $\crosprod$, then there is a compatible non-degenerate bounded representation $\overline{T}$ of $\mathcal{M}_l(\crosprod)$, and one thus obtains a representation $\overline{T} \circ i_A^\mr$ of $A$ and a representation $\overline{T} \circ i_G^\mr$ of $G$. The main hurdle, namely to construct any representations of $A$ and $G$ at all from $T$, has thus been taken, but still some work needs to be done to take care of the remaining details.

Section~\ref{sec:general_correspondence} contains, finally, the bijection between the non-degenerate $\mr$-continuous covariant representations of $\dynsys$ and the non-degenerate bounded representations of $\crosprod$, valid if $A$ has a bounded left approximate identity and $\mr$ consists of non-degenerate continuous covariant representations only. Obtaining this Theorem~\ref{t:bijection} is simply a matter of putting the pieces together. Results about preservation of invariant closed subspaces and bounded intertwining operators are also included, as is a specialization to the involutive case. For convenience, we have also included in this section some relevant explicit expressions and norm estimates as they follow from the previous material.

In Section~\ref{sec:correspondences} the basic correspondence from Theorem~\ref{t:bijection} is applied to various situations, including the cases of a trivial algebra and of a trivial group. Whereas an application of this theorem to the case of a trivial algebra does lead to non-trivial results about group Banach algebras, as discussed earlier in this Introduction, it does not give optimal results for a trivial group. In that case, the machinery of the present paper is, in fact, largely superfluous, but for the sake of completeness we have nevertheless included a brief discussion of that case and a formulation of the (elementary) optimal results.

\medskip
\emph{Reading guide.} In the discussion above it may have become evident that, whereas the construction of a Banach algebra crossed product requires modifications of the crossed product $C^*$-algebra construction which are fairly natural and easily implemented, establishing the desired correspondence at the level of (covariant) representations is more involved than for crossed product $C^*$-algebras. As evidence of this may serve the fact that Theorem~\ref{t:bijection} can, without too much exaggeration, be regarded as the summary of most material preceding it, including some results from \cite{extendart}. To facilitate the reader who is mainly interested in this correspondence as such, and in its applications in Section~\ref{sec:correspondences}, we have included (references to) the relevant definitions in Sections~\ref{sec:general_correspondence} and~\ref{sec:correspondences}. We hope that, with some browsing back, these two sections, together with this Introduction, thus suffice to convey how $\crosprod$ is constructed and what its main properties and special cases are.

\medskip
\emph{Perspectives.} According to its preface, \cite{williams} can only cover part of what is currently known about crossed products of $C^*$-algebras in one volume. Although the theory of crossed products of Banach algebras is, naturally, not nearly as well developed as for $C^*$-algebras, it is still true that more can be said than we felt could reasonably be included in one research paper. Therefore, in \cite{crossedtwo} we will continue the study of these algebras. We plan to include (at least) a characterization of $\crosprod$ by a universal property in the spirit of \cite[Theorem~2.61]{williams}, as well as a detailed discussion of $L^1$-algebras. As mentioned above, $L^1(G)$ is isometrically isomorphic to a crossed product as constructed in the present paper, and the well-known link between its representation theory and that for $G$ follows from our present results. We will include this, as a special case of similar results for $L^1(G,A)$ with twisted convolution. Also, we will then consider natural variations on the bijection theme: suppose that one has, say, a uniformly bounded class $\mr$ of pairs $\covrep$, where $\pi$ is a non-degenerate continuous anti-representation of $A$, $U$ is a strongly continuous anti-representation of $G$, and the pair $\covrep$ is anti-covariant, is it then possible to find an algebra of crossed product type, the non-degenerate bounded anti-representations of which correspond bijectively to the $\mr$-continuous pairs $\covrep$ with the properties as just mentioned? It is not too difficult to relate these questions to the results in the present paper, albeit sometimes for a closely related alternative Banach algebra dynamical system, and it seems quite natural to consider this matter, since examples of such $\mr$ are easy to provide. Once this has been done, we will also be able to infer the basic relation \cite[Proposition~2.1]{johnson} between $L^1(G)$-\emph{bi}modules and $G$-\emph{bi}modules from the results in the present paper. As mentioned above, we also plan to consider Morita equivalence and imprimitivity theorems, but that may have to wait until another time. The same holds for crossed products of Banach algebras in the context of positive representations on Banach lattices.\footnote{As a preparation, positivity issues have already been taken into account in \cite{extendart}.}

\section{Preliminaries}\label{sec:preliminaries}

In this section we introduce the basic definitions and notations, and establish some preliminary results. We start with a few general notions.

If $G$ is a group, then $e$ will be its identity element. If $G$ is a locally compact group, then we fix a left Haar measure $\mu$ on $G$, and denote integration of a function $\psi$ with respect to this Haar measure by $\int_G \psi(s) \ds$. We let $\Delta: G \to (0,\infty)$ denote the modular function, so for $f \in C_c(G)$ and $r \in G$ we have (\cite[Lemma~1.61,~Lemma~1.67]{williams})
\[ \Delta(r) \int_G f(sr) \ds = \int_G f(s) \ds, \quad \int_G \Delta(s^{-1}) f(s^{-1}) \ds = \int_G f(s) \ds. \]

If $X$ is a normed space, we denote by $B(X)$ the normed algebra of bounded operators on $X$. We let $\Inv(X)$ denote the group of invertible elements of $B(X)$. If $A$ is a normed algebra, we write $\Aut(A)$ for its group of bounded automorphisms.

A neighbourhood of a point in a topological space is a set with that point as interior point. It is not necessarily open.

Throughout this paper, the scalar field can be either the real or the complex numbers.

\subsection{Group representations}\label{subsec:group_representations}

\begin{defn}
A \emph{representation} $U$ of a group $G$ on a normed space $X$ is a group homomorphism $U: G \to \Inv(X)$.
\end{defn}

Note that there is no continuity assumption, which is actually quite convenient during proofs. For typographical reasons, we will write $U_s$ rather than $U(s)$, for $s\in G$.

\begin{lemma}\label{l:str_cont_compactly_bounded}
 Let $X$ be a non-zero Banach space and $U$ be a strongly continuous representation of a topological group $G$ on $X$. Then for every compact set $K \subset G$ there exist a constant $M_K > 0$ such that, for all $r \in K$,
\[ \frac{1}{M_K} \leq \norm{U_r} \leq M_K.\]
\end{lemma}

\begin{proof}
 For fixed $x \in X$, the map $r \mapsto U_r x$ is continuous, so the set $\{U_r x: r \in K\}$ is compact and hence bounded. By the Banach-Steinhaus Theorem, there exists $M_K^\prime>0$ such that $\norm{U_r} \leq M_K^\prime$ for all $r\in K$. Since, for $r\in K$, $1=\norm{\id_X}\leq \norm{U_{r^{-1}}}\norm{U_r}\leq M_{K^{-1}}^\prime \norm{U_r}$, $M_K=\max (M_K^\prime, M_{K^{-1}}^\prime)$ is as required.
\end{proof}

If $U$ is a strongly continuous representation of a topological group $G$ on a Banach space $X$, then the natural map from $G\times X$ to $X$ is separately continuous. Actually, it is automatically jointly continuous, according to the next result.

\begin{prop}\label{p:jointly_continuous}
Let $U$ be a strongly continuous representation of the locally compact group $G$ on the Banach space $X$. Then the map $(r,x)\to U_r x$ from $G\times X$ to $X$ is continuous.
\end{prop}

\begin{proof}
We may assume that $X$ is non-zero.
Fix $(r_0,x_0)\in G\times X$ and let $\eps>0$. There exists a neighbourhood $V$ of $r_0$ such that, for all $r\in V$, $\|U_r x_0-U_{r_0} x_0\|< \eps/2$. We may assume that $V$ is compact, and then Lemma~\ref{l:str_cont_compactly_bounded} yields an $M_V>0$ such that $\norm{U_r}\leq M_V$ for all $r\in V$. Therefore, if $r\in V$ and $\norm{x-x_0}< \eps/(2M_V)$, we have
\begin{align*}
\norm{U_r x-U_{r_0} x_0}&\leq \norm{U_r x - U_r x_0}+\norm{U_r x_0-U_{r_0} x_0}\\
&<M_V\cdot\frac{\eps}{2M_V} + \frac{\eps}{2}\\
&=\eps.
\end{align*}
\end{proof}

Corollary~\ref{c:strongly_continuous_at_e} below, and notably its second statement, will be used repeatedly when showing that a representation of a locally compact group is strongly continuous. The following lemma is a preparation.

\begin{lemma}\label{l:continuity_between_two_groups}
 Let $G$ and $H$ be two groups with a topology such that right multiplication is continuous in both groups, or such that left multiplication is continuous in both groups. Let $U: G \to H$ be a homomorphism. Then $U$ is continuous if and only if it is continuous at $e$.
\end{lemma}

\begin{proof}
Assume that right multiplication is continuous in both groups. Let $U$ be a homomorphism which is continuous at $e$ and let $(r_i) \in G$ be a net converging to $r\in G$. Then $r_i r^{-1} \to e$ by the continuity of right multiplication by $r^{-1}$ in $G$, and so
\[ U_{r_i} = U_{r_i r^{-1}} U_r \to U_r, \]
where the continuity of right multiplication by $U_r$ in $H$ is used in the last step. The case of continuous left multiplication is proved similarly, writing $U_{r_i}=U_rU_{r^{-1}r_i}$.
\end{proof}

\begin{corol}\label{c:strongly_continuous_at_e}
Let $G$ be a group with a topology such that right or left multiplication is continuous. Let $X$ be a Banach space and suppose $U:G\to\Inv(X)$ is a representation of $G$ on $X$. Then $U$ is a strongly continuous representation if and only $U$ is strongly continuous at $e$. If $U$ is uniformly bounded on some neighbourhood of $e$, and $Y\subset X$ is a dense subset of $X$, then $U$ is a strongly continuous representation if and only if $r\mapsto U_r y$ is continuous at $e$ for all $y\in Y$.
\end{corol}

\begin{proof}
 The first part follows from Lemma~\ref{l:continuity_between_two_groups} and the fact that multiplication in $B(X)$ equipped with the strong operator topology is separately continuous. The second part is an easy consequence of the first.
\end{proof}

If $X$ is a Hilbert space, then the \emph{$*$-strong operator topology} is the topology on $B(X)$ generated by the seminorms $T \mapsto \norm{Tx} + \norm{T^* x}$, with $x \in X$. A net $(T_i)$ converges $*$-strongly to $T$ if and only if both $T_i \to T$ strongly and $T_i^* \to T^*$ strongly. This topology is stronger than the strong operator topology and weaker than the norm topology, and multiplication is continuous in this topology on uniformly bounded subsets of $B(X)$.

\begin{remark}\label{r:unitary_is_*-strong-cont}
If $U$ is a unitary representation, then the decomposition of $r \mapsto U_r^*$ as $r \mapsto r^{-1} \mapsto U_r^{-1} = U_r^*$ shows that $U$ is strongly continuous if and only if $U$ is $*$-strongly continuous.
\end{remark}

\subsection{Algebra representations}

\begin{defn}
A \emph{representation} $\pi$ of an algebra $A$ on a normed space $X$ is an algebra homomorphism $\pi: A \to B(X)$. The representation $\pi$ is non-degenerate if $\pi(A)\cdot X:=\textup{span\,}\{\pi(a)x : a\in A,x\in X\}$ is dense in $X$.
\end{defn}

Note that it is not required that $\pi$ is unital if $A$ has a unit element, nor that $\pi$ is (norm) bounded if $A$ is a normed algebra.

\begin{remark}\label{r:strong_continuity_approx_identity}
 If $A$ is a normed algebra with a bounded left approximate identity $(u_i)$, and $\pi$ is a bounded representation of $A$ on the Banach space $X$, then it is easy to verify that $\pi$ is non-degenerate if and only if $\pi(u_i) \to \id_X$ in the strong operator topology.
\end{remark}

The following result, which will be used in the context of covariant representations, follows readily using Remark~\ref{r:strong_continuity_approx_identity}.

\begin{lemma}\label{l:alg_restriction_also_non_degenerate}
Let $A$ be a normed algebra with a bounded approximate left identity, and let $\pi$ be a bounded representation of $A$ on a Banach space $X$.
\begin{enumerate}
\item If $\pi$ is non-degenerate and $Z\subset X$ is an invariant subspace, then the restricted representation of $A$ to $Z$ is non-degenerate.
\item There is a largest invariant subspace such that the restricted representation of $A$ to it is non-degenerate. This subspace is closed. In fact, it is $\overline{\pi(A)\cdot X}$.
\end{enumerate}
\end{lemma}

\subsection{Banach algebra dynamical systems and covariant representations}\label{subsec:banach_algebra_dynamical_systems}

We continue by defining the notion of a dynamical system in our setting.

\begin{defn}\label{d:banach_algebra_dynamical_system}
A \emph{normed (resp.\ Banach) algebra dynamical system} is a triple $\dynsys$, where $A$ is a normed (resp.\ Banach) algebra\footnote{If $A$ is an algebra, then we do not assume that it is unital, nor that, if it is a unital normed algebra, the identity element has norm 1.}, $G$ is a locally compact Hausdorff group, and $\alpha: G \to \Aut(A)$ is a strongly continuous representation of $G$ on $A$. The system is called involutive when the scalar field is $\C$, $A$ has a bounded involution and $\alpha_s$ is involutive for all $s\in G$.
\end{defn}

From Proposition~\ref{p:jointly_continuous} we see that, for a Banach algebra dynamical system $\dynsys$, the canonical map $(s,a)\to \alpha_s(a)$ is continuous from $G\times A$ to $A$. This fact has as important consequence that a number of integrands in the sequel are \emph{continuous} vector valued functions on $G$, and we mention one of these explicitly for future reference.

\begin{lemma}\label{l:map_continuous}
 Let $\dynsys$ be a Banach algebra dynamical system, let $f \in C(G,A)$ and $s\in G$. Then the map $r\mapsto\alpha_{r}\left(f(r^{-1}s)\right)$ from $G$ to $A$ is continuous.
\end{lemma}

Indeed, this maps is the composition of the maps $r\mapsto (r,f(r^{-1}s))$ from $G$ to $G\times A$ and the canonical map from $G\times A$ to $A$.

Next we define our main objects of interest, the covariant representations.

\begin{defn}\label{d:covariant_representation}
Let $\dynsys$ be a normed algebra dynamical system, and let $X$ be a normed space. Then a \emph{covariant representation} of $\dynsys$ on $X$ is a pair $(\pi,U)$, where $\pi$ is a representation of $A$ on $X$ and $U$ is a representation of $G$ on $X$, such that for all $a\in A$ and $s\in G$,
\[
\pi(\alpha_s(a)) = U_s \pi(a) U_s^{-1} .
\]
The covariant representation $(\pi,U)$ is called continuous if $\pi$ is norm bounded and $U$ is strongly continuous, and it is called non-degenerate if $\pi$ is a non-degenerate representation of $A$.

If $\dynsys$ is a normed algebra dynamical system, then the covariant representation $\covrep$ of $\dynsys$ on $X$ is called involutive if the representation space $X$ is a Hilbert space, $\pi$ is an involutive representation of $A$ and $U$ is a unitary representation of $G$.
\end{defn}

We can use Lemma~\ref{l:alg_restriction_also_non_degenerate} to obtain a similar general result for normed dynamical systems which, for $G=\{e\}$, specializes to Lemma~\ref{l:alg_restriction_also_non_degenerate} again.

\begin{lemma}\label{l:dyn_sys_restriction_also_non_degenerate}
 Let $\dynsys$ be a normed algebra dynamical system, where $A$ has a bounded approximate left identity. Let $\covrep$ be a covariant representation of $\dynsys$ on the Banach space $X$, and assume that $\pi$ is bounded.
\begin{enumerate}
 \item If $(\pi,U)$ is non-degenerate and $Z$ is a subspace which is invariant under both $\pi(A)$ and $U(G)$, then the restricted covariant representation of $\dynsys$ to $Z$ is non-degenerate.
\item There is a largest subspace which is invariant under both $\pi(A)$ and $U(G)$ such that the restricted covariant representation of $\dynsys$ to it is non-degenerate. This subspace is closed. In fact, it is $\overline{\pi(A)\cdot X}$.
\end{enumerate}
\end{lemma}

\begin{proof}
The first part follows directly from the first part of Lemma~\ref{l:alg_restriction_also_non_degenerate}.
As for the second part, the second part of Lemma~\ref{l:alg_restriction_also_non_degenerate} shows that any subspace which is invariant under both $\pi(A)$ and $U(G)$ is contained in $\overline{\pi(A)\cdot X}$. It also yields that $\pi$ restricted to this space is a non-degenerate representation of $A$, hence we need only show that it is invariant under $U(G)$. As to this, if $y = \pi(a)x$, where $a \in A$ and $x \in X$, then, for $r \in G$, using the covariance,
\[ U_r y = U_r \pi(a)x = \pi(\alpha_r(a))U_r x \in \pi(A)\cdot X. \]
By continuity, this implies that $\overline{\pi(A)\cdot X}$ is invariant under $U_r$, for all $r \in G$.
\end{proof}

We conclude this subsection with some terminology about intertwining operators. Let $A$ be an algebra, and let $G$ be a group. Suppose that $X$ and $Y$ are two Banach spaces, and that $\pi:A\to B(X)$ and $\rho:A\to B(Y)$ are two representations of $A$. Then a bounded operator $\Phi:X\to Y$ is said to be a bounded intertwining operator between $\pi$ and $\rho$ if $\rho(a) \circ \Phi = \Phi \circ \pi(a)$, for all $a\in A$. A bounded intertwining operator between two group representations is defined similarly. If $\dynsys$ is a normed dynamical system, and $\covrep$ and $(\rho, V)$ are two covariant representations on Banach spaces $X$ and $Y$, respectively, then a bounded operator $\Phi: X \to Y$ is called an intertwining operator for these covariant representations, if $\Phi$ is an intertwining operator for $\pi$ and $\rho$, as well as for $U$ and $V$.

\subsection{$C_c(G,X)$}\label{subsec:C_c(G,X)}

We will frequently work with functions in $C_c(G,X)$, where $X$ is a Banach space. The next lemma, for the proof of which we refer to \cite[Lemma~1.88]{williams}, shows that these functions are uniformly continuous.

\begin{lemma}\label{l:uniformly_continuous}
 Let $G$ be a locally compact Hausdorff group and let $X$ be a Banach space. If $f \in C_c(G,X)$ and $\eps > 0$, then there exists a neighbourhood $V$ of $e \in G$ such that either one of $sr^{-1} \in V$ or $s^{-1}r \in V$ implies
\[ \norm{f(s) - f(r)} < \eps. \]
\end{lemma}

\begin{remark}
In this paper we will sometimes refer to the so-called inductive limit topology on $C_c(G,X)$. In these cases, we will be concerned with nets $(f_i) \in C_c(G,X)$ that converge to $f \in C_c(G,X)$, in the sense that $(f_i)$ is eventually supported in some fixed compact set $K \subset G$, and that $(f_i)$ converges uniformly to $f$ on $G$. As explained in \cite[Remark~1.86]{williams} and \cite[Appendix~D.2]{raeburnwilliams}, such a net \emph{is} convergent in the inductive limit topology, but the converse need not be true. However, it is true that a map from $C_c(G,X)$, supplied with the inductive limit topology, to a locally convex space is continuous precisely when it carries nets which converge in the above sense to convergent nets. We will use this fact in the sequel.
\end{remark}

The algebraic tensor product $C_c(G) \otimes A$ can be identified with a subspace of $C_c(G,A)$, and the following approximation result will be used on several occasions. We refer to \cite[Lemma~1.87]{williams} for the proof, from which a part of the formulation in the version below follows.

\begin{lemma}\label{l:density_lemma_inductive_limit_topology}
Let $G$ be a locally compact group, and let $X$ be Banach space. If $X_0$ is a dense subset of $X$, then $C_c(G)\otimes X_0$ is a dense subset of $C_c(G,X)$ in the inductive limit topology. In fact, it is even true that there exists a sequence $(f_n)$ in $C_c(G)\otimes X_0$, with all supports contained in a fixed compact subset of $G$ and which converges uniformly to $f$ on $G$, which implies that $f_n\to f$ in the inductive limit topology of $C_c(G,X)$.
\end{lemma}

\subsection{Vector valued integration}\label{subsec:vector_valued_integration}

For vector-valued integration in Banach spaces, we base ourselves on an integral defined by duality. The pertinent definition, as well as the existence, are contained in the next result, for the proof of which we refer to \cite[Theorem~3.27]{rudin} or \cite[Lemma~1.91]{williams}.

\begin{theorem}\label{t:vector_valued_integration}
 Let $G$ be a locally compact group, and let $X$ be a Banach space with dual space $X'$. Then there is a linear map $f \mapsto \int_G f(s) ds$ from $C_c(G,X)$ to $X$ which is characterized by
\begin{align}\label{d:characterization_of_integral}
\left\langle \int_G f(s) \ds, x' \right\rangle = \int_G \langle f(s), x' \rangle \ds, \quad \forall f \in C_c(G,X), \forall x' \in X'.
\end{align}
\end{theorem}

\begin{remark}\label{r:vector_valued_integration}
If $X$ and $Y$ are Banach spaces, it follows easily that bounded operators from $X$ to $Y$ can be pulled through the integral of the above theorem. If $X$ has a bounded involution this can also be pulled through the integral, since a bounded involution is a bounded conjugate linear map which can be viewed as a bounded operator from $X$ to the conjugate Banach space of $X$.

For $F \in C_c(G \times G, X)$, it is shown in \cite[Proposition~1.102]{williams} that, if one integrates out one variable, the resulting function is in $C_c(G,X)$. Applying continuous linear functionals, it then follows easily, analogously to the proof of \cite[Proposition~1.105]{williams}, that for such functions $F$ the vector-valued version of Fubini's theorem is valid.
\end{remark}

The integral from Theorem~\ref{t:vector_valued_integration} enables us to integrate compactly supported strongly (and $*$-strongly) continuous operator-valued functions (recall that unitary representations are $*$-strongly continuous by Remark \ref{r:unitary_is_*-strong-cont}).

\begin{prop}
Let $X$ be a Banach space, let $G$ be a locally compact group, and let $\psi: G \to B(X)$ be compactly supported and strongly continuous. Define
\begin{align}\label{d:pointwise_integral_for_sot_continuous_representation}
  \int_G \psi(s) \,ds := \left[ x \mapsto \int_G \psi(s)x \,ds \right] ,
\end{align}
where the integral on the right hand side is the integral from Theorem~\ref{t:vector_valued_integration}. Then $\int_G \psi(s)\,ds\in B(X)$, and
\begin{align}\label{d:norm_through_integral}
\norm{\int_G \psi(s)} \ds \leq \int_G \norm{\psi(s)} \ds.
\end{align}
If $T,R \in B(X)$, then
\begin{align}\label{d:operators_through_integral}
 T \int_G \psi(s) \ds \, R = \int_G T \psi(s) R \ds.
\end{align}
Furthermore, if $X$ is a Hilbert space and $\psi$ is $*$-strongly continuous, then
\begin{align}\label{d:involution_through_integral}
 \left( \int_G \psi(s) \ds \right)^* = \int_G \psi(s)^* \ds.
\end{align}
\end{prop}

\begin{proof}
 By applying elements of $X$ and functionals we obtain \eqref{d:operators_through_integral}, while \eqref{d:norm_through_integral} follows from applying elements of $X$ and taking norms. As for \eqref{d:involution_through_integral}, let $x,y \in X$ be arbitrary, then
\begin{align*}
 \left\langle x, \left( \int_G \psi(s) \ds \right)^* y \right\rangle &= \left\langle \left( \int_G \psi(s) \ds \right) x, y \right\rangle \\
&= \int_G \langle \psi(s)x, y \rangle \ds \\
&= \int_G \langle x, \psi(s)^* y \rangle \ds \\
&= \left\langle x, \left( \int_G \psi(s)^* \ds \right) y \right\rangle,
\end{align*}
where the $*$-strong continuity of $\psi$ ensures that the last line is well defined. Since this holds for all $x,y \in X$, \eqref{d:involution_through_integral} follows.
\end{proof}

\subsection{Quotients}\label{subsec:quotients}

The following standard type result, with a routine proof, will be used many times over, often implicitly. If $(D, \sigma)$ and $(E, \tau)$ are seminormed spaces, a linear map $T: D \to E$, is said to be bounded if there exists $C \geq 0$ such that $\tau(Tx) \leq C \sigma(x)$, for all $x \in D$. The seminorm (which is a norm if $\tau$ is a norm) of $T$, is then defined to be the minimal such $C$.

\begin{lemma}\label{l:completion_associated_with_seminorm}
Let $(D,\sigma)$ and $(E, \tau)$ be seminormed spaces, let $\overline{D / \ker(\sigma)}^\sigma$ be the completion of $D / \ker(\sigma)$ in the norm induced by $\sigma$, and let $\overline{E / \ker(\tau)}^\tau$ be defined similarly. Suppose that $T:D \to E$ is a bounded linear map. Then $T[\ker(\sigma)] \subset \ker(\tau)$ and, with $q^\sigma$ and $q^\tau$ denoting the canonical maps, there exists a unique bounded operator $\widetilde{T} :\overline{D / \ker(\sigma)}^\sigma \to \overline{E / \ker(\tau)}^\tau$, such that the diagram
\begin{equation}
\xymatrix
{
D \ar[d]_{q^\sigma} \ar[r]^{T} & E \ar[d]_{q^\tau} \\
\overline{D / \ker(\sigma)}^\sigma \ar[r]_{\widetilde{T}} & \overline{E / \ker(\tau)}^\tau
}
\end{equation}
is commutative. The norm of $\widetilde{T}$ then equals the seminorm of $T$. In particular, if $(E,\tau)$ is a Banach space, then $T \mapsto \widetilde{T}$ is a Banach space isometry between the bounded operators from $D$ into $E$ and the bounded operators from $\overline{D / \ker(\sigma)}^\sigma$ into $E$.

If, in addition, $(D,\sigma)$ is a seminormed algebra, $E$ is a Banach algebra, and $T$ is a bounded algebra homomorphism, then $\ker(\sigma)$ is a two-sided ideal, $\overline{D / \ker(\sigma)}^\sigma$ is a Banach algebra, and $\widetilde{T}$ is a bounded algebra homomorphism. In particular, if $X$ is a Banach space and $T: D \to B(X)$ is a bounded representation, then $\widetilde{T}$ is a bounded representation of $\overline{D / \ker(\sigma)}^\sigma$. In this case, $T$ is non-degenerate if and only if $\widetilde{T}$ is non-degenerate, a closed subspace of $X$ is invariant for $T$ if and only if it is invariant for $\widetilde{T}$, and if $Y$ is a Banach space, $\Phi \in B(X,Y)$ and $S: D \to B(Y)$ is a bounded representation, then $\Phi$ intertwines $T$ and $S$ if and only if $\Phi$ intertwines $\widetilde{T}$ and $\widetilde{S}$.

Alternatively, if, in addition, $D$ is an algebra with an involution, $\sigma$ is a $C^*$-seminorm, $E$ is a Banach algebra with a \textup{(}possibly unbounded\textup{)} involution, and $T$ is a bounded involutive algebra homomorphism, then $\ker(\sigma)$ is a self-adjoint two-sided ideal, $\overline{D / \ker(\sigma)}^\sigma$ is a $C^*$-algebra, and $\widetilde{T}$ is a bounded involutive algebra homomorphism.
\end{lemma}

\section{Crossed product: construction and basic properties}\label{sec:construction_and_basic_properties}

Let $\dynsys$ be a Banach algebra dynamical system, and let $\mr$ be a class of (possibly degenerate) continuous covariant representations of $\dynsys$ on Banach spaces. We will construct a Banach algebra, denoted $\crosprod$, which deserves to be called the crossed product associated with $\dynsys$ and $\mr$, and establish some basic properties. As will become clear from the discussion below, an additional condition on $\mr$ is needed (see Definition~\ref{d:uniformly_bounded_class}), which is automatic, hence ``not visible'', in the case of crossed product $C^*$-algebras. In later sections, we will also require the elements of $\mr$ to be non-degenerate, but for the moment this is not necessary.

We now start with the construction. Analogously to the $C^*$-algebra case, this construction is based on the vector space $C_c(G,A)$, as follows.

Let $f,g\in C_c(G,A)$. As a consequence of Lemma~\ref{l:map_continuous}, the function $(s,r) \mapsto f(r) \alpha_r(g(r^{-1}s))$ is in $C_c(G \times G, A)$. Therefore, if $s\in G$, Remark~\ref{r:vector_valued_integration} implies that
\[ [f * g](s):= \int_G f(r)\alpha_r(g(r^{-1}s)) \,dr \]
is a well-defined element of $A$, and that the thus defined map $f*g:G\to A$, the twisted convolution product of $f$ and $g$, is in $C_c(G,A)$. The associativity of this product on $C_c(G,A)$ is easily shown using Fubini, and thus $C_c(G,A)$ has the structure of an associative algebra. An easy computation shows that $\supp(f * g) \subset \supp(f) \cdot \supp(g)$.

Furthermore, if $\dynsys$ is an involutive Banach algebra dynamical system, then the formula
\begin{equation}\label{e:involution_definition}
f^*(s) := \Delta(s^{-1}) \alpha_s (f(s^{-1})^*)
\end{equation}
defines an involution on $C_c(G,A)$, so that $C_c(G,A)$ becomes an involutive algebra. The proof of this fact relies on a computation as in \cite[page~48]{williams}, which, as the reader may verify, is valid again because the involution on $A$ is bounded and hence can be pulled through the integral by Remark~\ref{r:vector_valued_integration}.

Our next step is to find an algebra seminorm on $C_c(G,A)$. It is here that a substantial difference with the construction as in \cite{williams} for crossed products associated with $C^*$-algebras occurs, leading to Definition~\ref{d:uniformly_bounded_class}.

To start with, assume that $\covrep$ is a continuous covariant representation in a Banach space $X$. For $f \in C_c(G,A)$, the function $s \mapsto \pi(f(s))U_s$ is strongly continuous from $G$ into $B(X)$ by continuity of multiplication in the strong operator topology on uniformly bounded subsets.  Therefore we can define
\begin{align}\label{d:intformdef}
 \intform(f) := \int_G \pi(f(s))U_s \ds,
\end{align}
where the integral on the right-hand side is as in \eqref{d:pointwise_integral_for_sot_continuous_representation}. Note that if $\covrep$ is involutive, $U$ is $*$-strongly continuous by Remark \ref{r:unitary_is_*-strong-cont}, so $s \mapsto \pi(f(s))U_s$ is $*$-strongly continuous by continuity of multiplication in the $*$-strong operator topology on uniformly bounded subsets, hence the involution can be pulled through the integral by \eqref{d:involution_through_integral}. Therefore the computations in the proof of \cite[Proposition~2.23]{williams} are valid, and they show that $\intform$, called the integrated form of $\covrep$, is a representation of $C_c(G,A)$, and that it is involutive if $\dynsys$ is involutive and $\covrep$ is an involutive continuous covariant representation.

With the construction as in \cite{williams} as a model, the natural way to construct a normed algebra from the associative algebra $C_c(G,A)$, given a collection $\mr$ of (possibly degenerate) continuous covariant representations on Banach spaces of the Banach algebra dynamical system $\dynsys$, is then as follows. For a continuous covariant representation $\covrep$ and $f \in C_c(G,A)$, we define $\sigma_{\covrep}(f) := \norm{\intform(f)}$. Since $\intform$ is a representation of $C_c(G,A)$, the map $\sigma_{\covrep}:C_c(G,A) \to [0,\infty)$ is an algebra seminorm on $C_c(G,A)$. Moreover, if $\covrep$ is an involutive continuous covariant representation, then $\intform$ is involutive and hence $\sigma_{\covrep}$ is a $C^*$-seminorm on $C_c(G,A)$.

Lemma~\ref{l:str_cont_compactly_bounded} shows that $U$ is bounded on compact sets $K$ by constants $M_K(U)$, and so we obtain the estimate
\begin{equation}\label{e:sigmabound}
  \sigma_{\covrep}(f) = \norm{\int_G \pi(f(s))U_s \,ds} \leq \norm{\pi} M_{\supp(f)}(U) \norm{f}_{L^1(G,A)}.
\end{equation}
 Now in the case of $C^*$-algebra dynamical systems, one takes the supremum over all such $C^*$-seminorms corresponding to (non-degenerate) involutive continuous covariant representations $\covrep$, and one defines the corresponding crossed product as the completion of $C_c(G,A)$ with respect to this seminorm (which can then be shown to be a norm). This is meaningful: since the constants $\norm{\pi}$ and $M_{\supp(f)}(U)$ in \eqref{e:sigmabound} are then always equal to 1, regardless of the choice of $\covrep$, the supremum is, indeed, pointwise finite. In general situations, when one wants to construct, in a similar way, a crossed product associated with a given class $\mr$ of continuous covariant representations, this supremum over the class need no longer be pointwise finite. Furthermore, even if the supremum \emph{is} well-defined, this supremum seminorm need not be a norm on $C_c(G,A)$. The solution to these problems is, obviously, to consider only classes $\mr$ such that there is a uniform bound for $\pi$ and $M_{\supp(f)}(U)$ in \eqref{e:sigmabound}, as $\covrep$ ranges over $\mathcal R$, and to use a quotient in the construction, as in Lemma~\ref{l:completion_associated_with_seminorm}. This leads to the following definitions.

\begin{defn}\label{d:uniformly_bounded_class}
Let $\dynsys$ be a Banach algebra dynamical system, and suppose $\mr$ is a class of continuous covariant representations of $\dynsys$. Then $\mr$ is called \emph{uniformly bounded} if there exist a constant $C \geq 0$ and a function $\nu:G\to[0,\infty)$, which is bounded on compact subsets of $G$, such that, for all $\covrep$ in $\mr$, $\norm{\pi} \leq C$ and $\norm{U_r}\leq \nu(r)$, for all $r \in G$.

If $\mr$ is a non-empty uniformly bounded class of continuous covariant representations, we let $\Cr = \sup_{\covrep\in\mr}\norm{\pi}$ be the minimal such $C$, and we let
\begin{equation}\label{e:nur_definition}
\nur(r):=\sup_{\covrep \in \mr} \norm{U_r},                                                                                                                      \end{equation}
where $r\in G$, be the minimal such $\nu$.\footnote{Since $r\mapsto\norm{U_r}$ is the supremum of continuous functions, it is a lower semicontinuous function on $G$, and hence the same holds for $\nur$.}

If $\dynsys$ is involutive, then $\mr$ is said to be involutive if $\covrep$ is involutive for all $\covrep\in\mr$.
\end{defn}

\begin{defn}\label{d:crossed_product}
Let $\dynsys$ be a Banach algebra dynamical system, and suppose $\mr$ is a non-empty uniformly bounded class of continuous covariant representations. Then we define the algebra seminorm $\sigmar$ on $C_c(G,A)$ by
\begin{equation*}
\sigmar(f) := \sup_{\covrep \in \mr} \norm{\intform(f)},
\end{equation*}
for $f\in C_c(G,A)$, and we let the corresponding crossed product $\crosprod$ be the completion of $C_c(G,A)/\ker(\sigmar)$ in the norm $\normr{\,.\,}$ induced by $\sigmar$. Multiplication in $\crosprod$ will still be denoted by $*$.
\end{defn}

\begin{remark}\label{r:seminorm_remark}\quad
\begin{enumerate}
\item From now on, all representations are assumed to be on Banach spaces rather than on normed spaces, since this is needed when integrating.
\item
Obviously, $\sigmar$ in Definition~\ref{d:crossed_product} is indeed finite, since, as in \eqref{e:sigmabound},
\[
\sigmar(f) \leq \Cr \left( \sup_{r\in\supp(f)} \nur(r)\right)\norm{f}_1 < \infty,
\]
for $f\in C_c(G,A)$.
\item By construction, $\crosprod$ is a Banach algebra. If $\dynsys$ and $\mathcal R$ are both involutive, then the seminorms $\sigma_{\covrep}$, for $\covrep\in\mr$, are all $C^*$-seminorms on $C_c(G,A)$, and hence the same holds for their supremum $\sigmar$. As has already been observed in Lemma~\ref{l:completion_associated_with_seminorm}, this implies that $\crosprod$ is then a $C^*$-algebra.
\end{enumerate}
\end{remark}

If $\dynsys$ is a Banach algebra dynamical system, and $\mr$ is a non-empty uniformly bounded class of continuous covariant representations, then the corresponding quotient map from Lemma~\ref{l:completion_associated_with_seminorm} will be denoted by $\qr$, rather than $q^{\sigmar}$. Hence
\[
\qr: C_c(G,A)\to\crosprod
\]
is the quotient homomorphism. Likewise, if $E$ is Banach space, and the linear maps $T: C_c(G,A) \to E$ and $S: C_c(G,A) \to C_c(G,A)$ are $\sigmar$-bounded, then their norms will be denoted by $\Vert T \Vert^\mr$ and $\norm{S}^\mr$, and the corresponding bounded operators from Lemma~\ref{l:completion_associated_with_seminorm} will be denoted by $T^\mr$ and $S^\mr$, with norms $\norm{T^\mr} = \Vert T\Vert^\mr$ and $\norm{S^\mr} = \Vert S \Vert^\mr$. Hence $T^\mr:\crosprod\to E$ and $S^\mr: \crosprod \to \crosprod$ are determined by
\begin{equation}\label{e:Tr_definition}
T^\mr(\qr(f))=T(f), \quad S^\mr(\qr(f)) = \qr(S(f)),
\end{equation}
for all $f\in C_c(G,A)$.

If $\covrep\in\mathcal R$ is a continuous covariant representation in the Banach space $X$, then $\intform : C_c(G,A) \to B(X)$ is certainly $\sigmar$ bounded, with norm at most 1. Hence there is a corresponding contractive representation $(\intform)^\mr: \crosprod \to B(X)$ of the Banach algebra $\crosprod$ in $X$, determined by
\begin{equation}\label{e:intformr_definition}
(\intform)^\mr(\qr(f)) = \intform(f),
\end{equation}
for all $f\in C_c(G,A)$. If $\dynsys$ and $\mr$ are involutive, and $X$ is the Hilbert representation space for $\covrep\in\mr$, then $(\intform)^\mr : \crosprod \to B(X)$ is an involutive representation of the $C^*$-algebra $\crosprod$ in the Hilbert space $X$. It is contractive by construction, although this is of course also automatic.

Suppose that $\mr$ is a uniformly bounded class of continuous covariant representations. By construction, we have, for $f\in C_c(G,A)$,
\begin{align*}
\normr{\qr(f)} = \sigmar(f) &= \sup_{\covrep \in \mathcal{R}} \norm{\intform (f)} = \sup_{\covrep \in \mathcal{R}} \norm{(\intform)^\mr (\qr(f))}.
\end{align*}
For later use, we establish that this formula for the norm in $\crosprod$ extends from $\qr(C_c(G,A))$ to the whole crossed product. The separation property that is immediate from it, will later find a parallel for the left centralizer algebra $\mathcal M_l(\crosprod)$ of $\crosprod$ in Proposition~\ref{p:representations_separate_left_centralizers}, under the extra conditions that $A$ has a bounded left approximate identity and that $\mr$ consists of non-degenerate continuous covariant representations only.

\begin{prop}\label{p:norm_formula_in_crossed_product}
Let $\dynsys$ be a Banach algebra dynamical system and $\mathcal R$ a non-empty uniformly bounded class of continuous covariant representations. Then, for all $c \in \crosprod$,
\[
\normr{c}=\sup_{\covrep \in \mathcal{R}} \norm{(\intform)^\mr (c)}.
\]
In particular, the representations $(\intform)^\mr$, for $\covrep\in\mathcal R$, separate the points of $\crosprod$.
\end{prop}

\begin{proof}
Let $c\in\crosprod$. The contractivity of $(\intform)^\mr$, for $\covrep\in\mathcal R$, yields $\sup_{\covrep \in \mathcal{R}} \norm{(\intform)^\mr (c)}\leq \normr{c}$.

As for the other inequality, let $\eps>0$. Choose $f\in C_c(G,A)$ such that $\normr{c-\qr(f)}<\eps/3$, and next choose $\covrep\in\mathcal R$ such that $\norm{\intform (f)}>\normr{\qr(f)}-\eps/3$. Then
\begin{align*}
\norm{(\intform)^\mr(c)} &\geq \norm{(\intform)^\mr (\qr(f))} - \norm{(\intform)^\mr(c - \qr(f))} \\
&\geq \norm{\intform (f)} - \normr{c-\qr(f)} \\
&\geq \normr{\qr(f)}-\frac{\eps}{3}-\frac{\eps}{3} \\
&\geq \normr{c}-\frac{\eps}{3} - \frac{2\eps}{3}.
\end{align*}
Therefore, $\sup_{\covrep \in \mathcal{R}} \norm{(\intform)^\mr (c)}>\normr{c}-\eps$ for all $\eps>0$, as desired.
\end{proof}

We will now proceed to show that $\qr(C_c(G)\otimes A)$ is dense in $\crosprod$, which will obviously be convenient in later proofs. We start with a lemma which is of some interest in itself.

\begin{lemma}\label{l:quotient_map_continuous_in_inductive_limit_topology_on_C_c(G,A)}
 Let $\dynsys$ be a Banach algebra dynamical system.
\begin{enumerate}
 \item If $\mathcal{R}$ is a non-empty uniformly bounded class of continuous covariant representations, then $\qr:C_c(G,A)\to\crosprod$ is continuous in the inductive limit topology of $C_c(G,A)$. That is, if $(f_i) \in C_c(G,A)$ is a net, eventually supported in a compact set and converging uniformly to $f \in C_c(G,A)$ on $G$, then $\sigmar(f_i -f) \to 0$.
\item If $\covrep$ is a continuous covariant representation, then $\intform$ is continuous in the inductive limit topology. That is, if $(f_i) \in C_c(G,A)$ is a net, eventually supported in a compact set and converging uniformly to $f \in C_c(G,A)$ on $G$, then $\norm{\intform(f_i) - \intform(f)} \to 0$.
\end{enumerate}
\end{lemma}

\begin{proof}
\begin{enumerate}
\item It suffices to prove the case when $f = 0$. Let $K$ be a compact set and $i_0$ an index such that $f_i$ is supported in $K$ for all $i\geq i_0$. Let $M$ denote an upper bound for $\nur$ on $K$. Then, for $\covrep \in \mathcal{R}$ and $i\geq i_0$,
\begin{align*}
 \norm{\intform(f_i)} &= \norm{\int_G \pi(f_i(s))U_s \,ds} \\
&\leq\int_K \norm{\pi}\norm{f_i(s)}M\,ds\\
&\leq \Cr M \mu(K) \norm{f_i}_\infty.
\end{align*}
It follows that, for $i\geq i_0$,
\[ \sigmar(f_i) = \sup_{\covrep \in \mathcal{R}} \norm{\intform(f_i)} \leq \Cr M \mu(K) \norm{f_i}_\infty.\]
Hence $\sigmar(f_i) \to 0$.
\item This follows from $(i)$ by taking $\mr = \{ \covrep \}$, since then $\sigmar(f)$ equals $\norm{\intform(f)}$.
\end{enumerate}
\end{proof}

\begin{corol}\label{c:image_of_tensor_product_dense_in_crossed_product}
Let $\dynsys$ be a Banach algebra dynamical system and $\mathcal{R}$ a non-empty uniformly bounded class of continuous covariant representations. Then $\qr(C_c(G) \otimes A)$ is dense in $\crosprod$.
\end{corol}
\begin{proof}
By \cite[Lemma~1.87]{williams}, $C_c(G) \otimes A$ is dense in $C_c(G,A)$ in the inductive limit topology. The above lemma therefore implies that $\qr(C_c(G)\otimes A)$ is dense in $\qr(C_c(G,A))$. Since the latter is dense in $\crosprod$ by construction, the result follows.
\end{proof}

\section{Approximate identities}\label{sec:approximate_identities}

In this section we are concerned with the existence of a bounded approximate left identity in $\crosprod$. This is a key issue in the formalism, as the existence of a bounded left approximate identity will allow us to apply Theorem~\ref{t:summary_for_centralizers} later on to pass from (non-degenerate) representations of $\crosprod$ to representations of the left centralizer algebra of $\crosprod$, from which, in turn, we will be able to obtain a continuous covariant representation of $\dynsys$.

If $\crosprod$ happens to be a $C^*$-algebra, as is, e.g., the case in \cite{williams}, then the existence of a bounded two-sided approximate identity is of course automatic, but for the general case some effort is needed to show that the existence of a bounded left approximate identity in $A$ implies the similar property for $\crosprod$. For the present paper, we need only a bounded left approximate identity, but we also consider an approximate right identity for completeness and future use.

We need two preparatory results. The first one, Lemma~\ref{l:approximation_on_compacta_involving_alpha_and_approximate_right_identity}, is only relevant for the case of a bounded approximate right identity.

\begin{lemma}\label{l:approximation_on_compacta_involving_alpha_and_approximate_right_identity}
  Let $\dynsys$ be a Banach algebra dynamical system and suppose $A$ has a bounded approximate right identity $(u_i)$. Fix an element $a \in A$ and a compact set $K \subset G$. Then for all $\eps > 0$ we can find an index $i_0$ such that, for all $i \geq i_0$ and $s \in K$,
\[ \|a \alpha_s(u_i) - a\| < \eps. \]
\end{lemma}
\begin{proof}
Let $M \geq 1$ be an upper bound for $(u_i)$. By Lemma~\ref{l:str_cont_compactly_bounded} we can choose an upper bound $M_K>0$ for $\alpha$ on $K$.

By the continuity of $s \mapsto s^{-1} \mapsto \alpha_{s^{-1}}(a) = \alpha_s^{-1}(a)$, for every $s \in K$ there exists a neighbourhood $W_s$ such that, for all $r \in W_s$,
 \[ \|\alpha_r^{-1}(a) - \alpha_s^{-1}(a)\| < \frac{\eps}{3M_K M}. \]
Choose a finite subcover $W_{s_1}, \ldots, W_{s_n}$ of $K$. Then there exists an index $i_0$ such that $i\geq i_0$ implies
\[ \| \alpha_{s_k}^{-1}(a)u_i - \alpha_{s_k}^{-1}(a) \| < \frac{\eps}{3M_K}, \]
for every $1\leq k \leq n$.
For $s \in K$, choose $k$ such that $s \in W_{s_k}$. Then, for all $i \geq i_0$,
\begin{align*}
 \| a\alpha_s(u_i) - a\| &= \norm{\alpha_s(\alpha_s^{-1}(a) u_i - \alpha_s^{-1}(a))} \\
&\leq M_K \left( \|\alpha_s^{-1}(a)u_i - \alpha_{s_k}^{-1}(a)u_i\| + \| \alpha_{s_k}^{-1}(a)u_i - \alpha_{s_k}^{-1}(a) \|  \right. \\
&\phantom{=} \left. + \| \alpha_{s_k}^{-1}(a) - \alpha_s^{-1}(a) \| \right) \\
&< M_K \left( \frac{\eps }{3M_K M}\cdot M + \frac{\eps}{3M_K} + \frac{\eps}{3M_K M} \right) \leq \eps.
\end{align*}
\end{proof}

Actually, the existence of a bounded approximate left (resp.\ right) identity in $\crosprod$ is inferred from the existence of an suitable approximate left (resp.\ right) identity in $C_c(G,A)$ in the inductive limit topology. This is the subject of the following theorem.

\begin{theorem}\label{t:C_c(G,A)_has_approximate_identities_in_inductive_limit_topology}
 Let $\dynsys$ be a Banach algebra dynamical system and let $\mathcal{Z}$ be a neighbourhood basis of $e$ of which all elements are contained in a fixed compact set. For each $V \in \mathcal{Z}$, take a positive $z_V \in C_c(G)$ with support contained in $V$ and integral equal to one. Suppose $(u_i)$ is a bounded approximate left \textup{(}resp.\ right\textup{)} identity of $A$. Then the net $\left(f_{(V,i)}\right)$, where
\[ f_{(V,i)} := z_V \otimes u_i, \]
directed by $(V,i) \leq (W,j)$ if and only if $W \subset V$ and $i \leq j$, is an approximate left \textup{(}resp.\ right\textup{)} identity of $C_c(G,A)$ in the inductive limit topology. In fact, for all $f \in C_c(G,A)$ the net $\left(f_{(V,i)} * f\right)$ \textup{(}resp.\ $\left(f * f_{(V,i)}\right)$\textup{)} is supported in a fixed compact set and converges uniformly to $f$ on $G$.
\end{theorem}

\begin{proof}
Let $K$ be a compact set containing all $V$ in $\mathcal Z$, and assume that $(u_i)$ is bounded by $M>0$.

Since the $f_{(V,i)}$ are all supported in $K$, all $f_{(V,i)} * f$ (resp.\ $f * f_{(V,i)}$) are supported in $K \, \supp(f)$ (resp.\ $\supp(f) \, K$) for each $f\in C_c(G) \otimes A$.

For the approximating property we start with elements of $C_c(G)\otimes A$. For this it is sufficient to consider elementary tensors, so fix $0\neq f = z \otimes a$ with $z \in C_c(G)$ and $a \in A$.

First we consider the approximate left identity. Suppose that $\eps>0$ is given. Let $M_{K}>0$ be an upper bound for $\alpha$ on $K$. Then, for $s\in G$,
\begin{align*}
\norm{[f_{(V,i)} * f](s) - f(s)} &= \norm{\int_G f_{(V,i)}(r) \alpha_r(f(r^{-1}s)) \,dr - f(s)} \\
&= \norm{\int_G z_V(r) z(r^{-1}s)u_i \alpha_r(a) \,dr - z(s)a}  \\
&= \norm{ \int_G z_V(r) z(r^{-1}s)u_i \alpha_r(a) - z_V(r) z(s)a \, dr} \\
&\leq \int_G z_V(r) \norm{z(r^{-1}s)u_i \alpha_r(a) - z(s)a} \,dr  \\
&\leq \int_{\supp(z_V)} z_V(r) \norm{z(r^{-1}s)u_i\alpha_r(a) - z(s)u_i\alpha_r(a)}\,dr\\
& \quad + \int_{\supp(z_V)} z_V(r)\norm{z(s)u_i\alpha_r(a) - z(s)u_i a}\,dr  \\
& \quad + \int_{\supp(z_V)} z_V(r)\norm{z(s)u_i a - z(s)a} \,dr \\
&\leq MM_{K}\norm{a}\int_{\supp(z_V)} z_V(r) |z(r^{-1}s) - z(s)|\,dr \\
& \quad + \Vert z\Vert_\infty M\int_{\supp(z_V)} z_V(r)\norm{\alpha_r(a)-a}\,dr \\
& \quad + \Vert z\Vert_\infty \int_{\supp(z_V)} z_V(r) \norm{u_i a -a}\,dr.
\end{align*}
By the uniform continuity of $z$, there exists a neighbourhood $U_1$ of $e$ such that $|z(r^{-1}s) - z(s)|<\eps/(3MM_{K}\norm{a})$, for all $r\in U_1$ and $s\in G$. Hence, for all $s\in G$, the first term is less than $\eps/3$ as soon as $V\subset U_1$. By the strong continuity of $\alpha$, there exists a neighbourhood $U_2$ of $e$ such that $\norm{\alpha_r(a)-a}<\eps/(3M\Vert z\Vert_\infty)$, for all $r\in U_2$. Hence, for all $s\in G$, the second term is less than $\eps/3$ as soon as $V\subset U_2$. There exists an index $i_0$ such that the third term is less than $\eps / 3$ for all $i\geq i_0$ and all $V\in\mathcal Z$. Choose $V_0\in\mathcal Z$ such that $V_0\subset U_1\cap U_2$. Then, if $(V,i)\geq (V_0,i_0)$, we have $\norm{f_{(V,i)} * f(s) - f(s)}<\eps$ for all $s\in G$ as required.

The approximate right identity is somewhat more involved. Suppose that $\eps>0$ is given. Let $K_1$ be a compact set containing all $V$ in $\mathcal Z$, as well as the supports of all $f*f_{(V,i)}$ and $f$. Let $M_{K_1 K^{-1}}>0$ be an upper bound for $\alpha$ on $K_1 K^{-1}$. Since $f(s) * f_{(V,i)} (s)- f(s)=0$ for $s\notin K_1$, we assume that $s\in K_1$, and then
\begin{align*}
&\norm{ [f * f_{(V,i)}](s) - f(s)} \\
&= \norm{\int_G f(r) \alpha_r(f_{(V,i)}(r^{-1}s)) \,dr - f(s)} \\
&= \norm{\int_G z(r)a z_V(r^{-1}s)\alpha_r(u_i)\,dr - z(s)a} \\
&= \norm{\int_G z(sr) z_V(r^{-1}) a \alpha_{sr}(u_i) \,dr - z(s)a} \\
&= \norm{\int_G \Delta(r^{-1}) z(sr^{-1})z_V(r) a \alpha_{sr^{-1}}(u_i) - z_V(r) z(s)a \, dr} \\
&\leq \int_G z_V(r) \norm{\Delta(r^{-1})z(sr^{-1}) a \alpha_{sr^{-1}}(u_i) - z(s)a} \,dr \\
&\leq \int_{\supp(z_V)} z_V(r)  \norm{\Delta(r^{-1})z(sr^{-1}) a \alpha_{sr^{-1}}(u_i) - z(sr^{-1}) a \alpha_{sr^{-1}}(u_i)}\,dr \\
&\quad + \int_{\supp(z_V)} z_V(r) \norm{z(sr^{-1}) a \alpha_{sr^{-1}}(u_i) - z(s)a \alpha_{sr^{-1}}(u_i)} \,dr \\
&\quad + \int_{\supp(z_V)} z_V(r) \norm{z(s)a \alpha_{sr^{-1}}(u_i) - z(s)a} \,dr\\
&\leq \Vert z\Vert_\infty\norm{a} M_{K_1 K^{-1}} M \int_{\supp(z_V)} z_V(r)  |\Delta(r^{-1})-1|\,dr \\
&\quad + \norm{a} M_{K_1 K^{-1}} M \int_{\supp(z_V)} z_V(r) |z(sr^{-1})-z(s)|\,dr\\
&\quad + \Vert z\Vert_\infty \int_{\supp(z_V)} z_V(r) \norm{a \alpha_{sr^{-1}}(u_i) - a} \,dr.
 \end{align*}
By the continuity of $\Delta$ there exists a neighbourhood $U_1$ of $e$ such that $|\Delta(r^{-1})-1|<\eps/(3\Vert z\Vert_\infty\norm{a} M_{K_1 K^{-1}} M )$, for all $r\in U_1$. Hence, for all $s\in K_1$, the first term is less than $\eps/3$ as soon as $V\subset U_1$.  By the uniform continuity of $z$, there exists neighbourhood $U_2$ of $e$ such that $|z(sr^{-1}) - z(s)|<\eps/(3\norm{a} M_{K K_1^{-1}} M )$, for all $r\in U_2$ and $s\in G$. Hence, for all $s\in K_1$, the second term is less than $\eps/3$ as soon as $V\subset U_2$. An application of Lemma~\ref{l:approximation_on_compacta_involving_alpha_and_approximate_right_identity} to the compact set $K_1 K^{-1}$ shows that there exists an index $i_0$ such that $\norm{a \alpha_{sr^{-1}}(u_i) - a}<\eps/(3\Vert z\Vert_\infty)$, for all $i\geq i_0$, $s\in K_1$ and $r\in K$. Hence, for all $s\in K_1$, the third term is less than $\eps/3$ if $i\geq i_0$. Choose $V_0\in\mathcal Z$ such that $V_0\subset U_1\cap U_2$. Then, if $(V,i)\geq (V_0,i_0)$, we have $\norm{[f * f_{(V,i)}](s) - f(s)} < \eps$ for all $s\in K_1$, and hence for all $s \in G$, as required.

We now pass from $C_c(G)\otimes A$ to $C_c(G,A)$, using that $C_c(G)\otimes A$ is uniformly dense in $C_c(G,A)$ (a rather weak consequence of \cite[Lemma~1.87]{williams}).

We start with the left approximate identity. Let $f\in C_c(G,A)$ and $\eps>0$ be given. For arbitrary $g\in C_c(G,A)$ and $s\in G$ we have
\begin{align*}
&\norm{[f_{(V,i)} * f](s) - f(s)}\\
&\leq \norm{[f_{(V,i)}*(f-g)](s)} + \norm{[f_{(V,i)}*g](s)-g(s)} + \norm{g(s)-f(s)}\\
&=\norm{\int_{G} z_V(r) u_i \alpha_r((f-g)(r^{-1}s))\,dr} + \norm{[f_{(V,i)}*g](s)-g(s)} + \norm{g(s)-f(s)}\\
&\leq M  M_K \Vert f-g\Vert_\infty \int_{\supp(z_V)} z_V(r)\, dr + \norm{[f_{(V,i)}*g](s)-g(s)} + \norm{g - f}_\infty\\
&\leq (M M_K + 1) \Vert f-g\Vert_\infty + \norm{f_{(V,i)} * g -g}_\infty.
\end{align*}
The cited density yields $g\in C_c(G)\otimes A$ such that the first term is less than $\eps/2$. By the first part of the proof, there exists an index $(V_0,i_0)$ such that the second term is less than $\eps/2$ for all $(V,i)\geq (V_0,i_0)$. Therefore $\Vert f_{(V,i)}*f-f\Vert_\infty < \eps$ for all $(V,i)\geq (V_0,i_0)$.

As for the approximate right identity, let $f\in C_c(G,A)$ and $\eps>0$ be given. As above, we let $K_1$ be a compact set containing all $V$ in $\mathcal Z$, as well as the supports of all $f*f_{(V,i)}$ and $f$, and choose an upper bound $M_{K_1 K^{-1}}>0$ for $\alpha$ on $K_1 K^{-1}$. Let $N_{K^{-1}}$ be an upper bound for $\Delta$ on $K^{-1}$. Then, for arbitrary $g\in C_c(G,A)$ and $s\in K_1$, we have
\begin{align*}
&\norm{f * f_{(V,i)}(s)-f(s)}\\
&\leq \norm{[(f-g)*f_{(V,i)}](s)} + \norm{[g*f_{(V,i)}](s)-g(s)} + \norm{g(s)-f(s)}\\
&=\norm{\int_G  (f-g)(r) z_V(r^{-1}s)\alpha_r(u_i)\,dr} + \norm{[g*f_{(V,i)}](s)-g(s)} + \norm{g(s)-f(s)}\\
&=\norm{\int_{\supp(z_V)} \Delta(r^{-1})(f-g)(sr^{-1}) z_V(r)\alpha_{sr^{-1}}(u_i)\,dr} \\
&\quad + \norm{[g*f_{(V,i)}](s)-g(s)} + \norm{g(s)-f(s)}\\
&\leq N_{K^{-1}} \Vert f-g\Vert_\infty M_{K_1K^{-1}}M\int_{\supp(z_V)} z_V(r)\,dr  \\
&\quad+ \norm{[g*f_{(V,i)}](s)-g(s)} + \norm{g - f}_\infty \\
&\leq(N_{K^{-1}} M_{K_1K^{-1}}M + 1) \Vert f-g\Vert_\infty + \Vert g*f_{(V,i)}-g\Vert_\infty.
\end{align*}
As above, there exists an index $(V_0,i_0)$ such that, for all $(V,i)\geq (V_0,i_0)$, $\Vert f * f_{(V,i)}(s) - f(s)\Vert_\infty < \eps$ for all $s\in K_1$. Since this is trivially true for $s\notin K_1$, we are done.
\end{proof}

After these preparations we can now establish that $\crosprod$ has a bounded approximate left identity if $A$ has one. We keep track of the constants rather precisely, since the upper bound for the norms of a bounded approximate left identity of $\crosprod$ will enter the picture naturally when considering the relation between (non-degenerate) continuous covariant representations of $\dynsys$ and (non-degenerate) bounded representations of $\crosprod$ later on, cf.\ Remark~\ref{r:constantdefinitions_estimates} and Section~\ref{sec:correspondences}. Therefore, before we prove the result on the approximate identities, let us introduce the relevant constant.

\begin{defn}\label{d:Nr_definition}
Let $\dynsys$ be a Banach algebra dynamical system and $\mathcal{R}$ a non-empty uniformly bounded class of continuous covariant representations. Let $\nur:G\to [0,\infty)$ be defined as in \eqref{e:nur_definition}.
Let $\mathcal{Z}$ be a neighbourhood basis of $e$ of which all elements are contained in a fixed compact set, and define
\begin{equation}\label{e:Nr_definition}
\Nr= \inf_{V \in \mathcal{Z}} \sup_{r \in V} \nur(r)<\infty.
\end{equation}
\end{defn}
Note that, since all $V$ in $\mathcal Z$ are contained in a fixed compact set, and $\nur$ is bounded on compacta, $\Nr$ is indeed finite. Furthermore, this definition of $\Nr$ does not depend on the choice of $\mathcal Z$. To see this, let $\mathcal Z_1$ and $\mathcal Z_2$ be two neighbourhood bases as in the theorem. For every $V_1\in\mathcal Z_1$, there exists $V_2\in\mathcal Z_2$ such that $V_2\subset V_1$, and then $\sup_{r \in V_2} \nur(r)\leq\sup_{r \in V_1} \nur(r)$. This implies that $\inf_{V \in \mathcal{Z}_2} \sup_{r \in V}\nur(r)\leq \inf_{V \in \mathcal{Z}_1} \sup_{r \in V} \nur(r)$, and the independence of the choice obviously follows.

The constant $\Nr$ can be viewed as $\limsup_{r\to e}\nur(r)$.

\begin{theorem}\label{t:crossed_product_has_bounded_approximate_identities}
Let $\dynsys$ be a Banach algebra dynamical system, where $A$ has an $M$-bounded approximate left \textup{(}resp.\ right\textup{)} identity $(u_i)$, and let $\mathcal{R}$ be a non-empty uniformly bounded class of continuous covariant representations. Let $\eps > 0$, and choose a neighbourhood $V_0$ of $e$ with compact closure such that
\[
\Nr \leq \sup_{r \in V_0} \nur(r) \leq \Nr + \eps.
\]
Let $\mathcal{Z}$ be a neighbourhood basis of $e$ of which all elements are contained in $V_0$. For each $V \in \mathcal{Z}$, let $z_V \in C_c(G)$ be a positive function with support contained in $V$ and integral equal to one. Define $f_{(V,i)} := z_V \otimes u_i$, for each $V\in\mathcal Z$ and each index $i$.

Then the associated net $\left(\qr(f_{(V,i)})\right)$ as above is a $\Cr M(\Nr+\eps)$-bounded approximate left \textup{(}resp.\ right\textup{)} identity of $\crosprod$.

If $V_0$ satisfies $N^\mr = \sup_{r \in V_0} \nur(r)$, then $\left(\qr(f_{(V,i)})\right)$ is $\Cr M \Nr$-bounded.
\end{theorem}

\begin{proof}
We will prove the left version, the right version is similar.

If $V \in \mathcal Z$ and $i$ is an index then we find that, for $\covrep \in \mr$,
\begin{align}\label{e:norm_computation_approximate_identity}
\norm{\intform(f_{(V,i)})} &= \norm{\int_{V_0} z_V(s) \pi(u_i) U_s \,ds} \\
&\leq \norm{\pi} \norm{u_i} \sup_{r \in V} \norm{U_r} \int_{V} z_V(s)\,ds \notag\\
&\leq \Cr  M \sup_{r\in V_0}\nur(r).\notag
\end{align}
Hence $\sigmar(\qr(f_{(V,i)}))=\sigmar(f_{(V,i)}) \leq \Cr  M \sup_{r\in V_0}\nur(r)\leq \Cr M(\Nr+\eps)$, as desired. To show that $\left(\qr(f_{(V,i)})\right)$ is actually an approximate left identity, we start by noting that, according to Theorem~\ref{t:C_c(G,A)_has_approximate_identities_in_inductive_limit_topology}, $f_{(V,i)} * f \to f$ in the inductive limit topology on $C_c(G,A)$, for all $f \in C_c(G,A)$. Therefore, Lemma~\ref{l:quotient_map_continuous_in_inductive_limit_topology_on_C_c(G,A)} implies that $\qr(f_{(V,i)}) * \qr(f) \to \qr(f)$, for all $f\in C_c(G,A)$. Since we have already established that $\left(\qr(f_{(V,i)})\right)$ is uniformly bounded in $\crosprod$, an easy $3\eps$-argument shows that the net is indeed a left approximate identity of $\crosprod$.

As for the second part, if $V_0$ and $\mathcal Z$ are as indicated and $V\subset V_0$ is in $\mathcal Z$, then a computation as in \eqref{e:norm_computation_approximate_identity} yields
\[ \norm{\intform(f_{(V,i)})} \leq \Cr  M \sup_{r\in V_0}\nur(r)\leq \Cr M(\Nr+\eps), \]
hence $\sigmar(\qr(f_{(V,i)}))\leq  \Cr  M (\Nr+\eps)$. This computation with $\eps = 0$ shows the final remark of the theorem.
\end{proof}

For convenience we introduce the following notation.

\begin{defn}\label{d:infimum_of_bounds}
 Let $\mathcal{A}$ be a normed algebra with a bounded approximate left (resp.\ right) identity. Then $M_l^\mathcal{A}$ (resp.\ $M_r^\mathcal{A}$) denotes the infimum of the upper bounds of all approximate left (resp.\ right) identities. If $\dynsys$ is a Banach algebra dynamical system, with $A$ having a bounded left (resp.\ right) approximate identity, and $\mr$ is a non-empty uniformly bounded class of continuous covariant representations, then we will write $M_l^\mr$ (resp.\ $M_r^\mr$), rather than $M_l^{\crosprod}$ $\left( \textup{resp.\ }M_r^{\crosprod} \right)$.
\end{defn}

\begin{corol}\label{c:approximate_identities}
Let $\dynsys$ be a Banach algebra dynamical system, where $A$ has a bounded approximate left \textup{(}resp.\ right\textup{)} identity $(u_i)$, and let $\mathcal{R}$ be a non-empty uniformly bounded class of continuous covariant representations. Then $\crosprod$ has a bounded approximate left \textup{(}resp.\ right\textup{)} identity, and
\begin{align*}
 M_l^\mr \leq \Cr M_l^A \Nr, \\
 M_r^\mr \leq \Cr M_r^A \Nr. \\
\end{align*}
\end{corol}

\begin{proof}
 We prove the left version, the right version being similar. If $\eps > 0$, then $A$ has an $M_l^A + \eps$ approximate left identity, so by the above theorem $\crosprod$ has a $\Cr (M_l^A + \eps) (\Nr + \eps)$-bounded approximate left identity, and the result follows.
\end{proof}

\section{Representations: from $\dynsys$ to $\crosprod$}\label{sec:from_dynsys_to_crosprod}

Our principal interest lies in the relation between a non-empty uniformly bounded class $\mr$ of continuous covariant representations of a Banach algebra dynamical system $\dynsys$, and the bounded representations of the associated crossed product $\crosprod$. In this section, we study the easiest part of this relation, which is concerned with passing from suitable continuous covariant representations of $\dynsys$ to $\sigmar$-bounded representations of $C_c(G,A)$, and subsequently to bounded representations of $\crosprod$. The other way round, i.e., passing from bounded representations of $\crosprod$ to continuous covariant representations of $\dynsys$, is more involved, and will be taken up in Section~\ref{sec:from_crosprod_to_dynsys}, after the preparatory Section~\ref{sec:centralizer_algebras}. At that point, non-degeneracy of representations will become essential, but for the present section this is not necessary yet.

Above, we wrote ``suitable'' representations, because there are more continuous covariant representations yielding bounded representations of the crossed product, than just those used to construct that crossed product (which yield contractive ones). The relevant terminology is introduced in the following definition.

\begin{defn}\label{d:affiliated}
Let $\dynsys$ be a Banach algebra dynamical system, and let $\mr$ be a non-empty uniformly bounded class of continuous covariant representations. A covariant representation $\covrep$ of $\dynsys$ in a Banach space $X$ is called \emph{$\mr$-continuous}, if it is continuous, and the homomorphism
\[
\intform: C_c(G,A)\to B(X)
\]
is $\sigmar$-bounded.
\end{defn}

\begin{remark}\label{r:affiliated_is_bounded}
 It is clear that a continuous covariant representation $\covrep$ of $\dynsys$ is $\mr$-continuous if and only if $\intform$ is continuous as an operator from the space $C_c(G,A)$, equipped with the topology induced by the seminorm $\sigmar$, to $B(X)$, equipped with the norm topology.
\end{remark}

By Lemma~\ref{l:completion_associated_with_seminorm}, an $\mr$-continuous covariant representation yields a bounded representation of the Banach algebra $\crosprod$, determined by \eqref{e:Tr_definition}, which gives \eqref{e:intformr_definition} again:
\begin{equation}
(\intform)^\mr(\qr(f))=\intform(f),
\end{equation}
for $f\in C_c(G,A)$. Then $\Vert \intform \Vert=\Vert \intform\Vert^\mr$. If, in addition, $\dynsys$, $\mr$ and $\covrep$ are involutive, then $(\intform)^\mr$ is an involutive representation of the $C^*$-algebra $\crosprod$. Of course, returning to the not necessarily involutive case, the continuous covariant representations in $\mr$ are certainly $\mr$-continuous, and the corresponding representations of $\crosprod$ are contractive.

Later, in Proposition~\ref{p:from_crossed_product_to_dyn_sys}, we will be able to show that, if $\mr$ is a non-empty uniformly bounded class of continuous covariant representations and if $A$ has a bounded left approximate identity, the assignment $\covrep \mapsto (\intform)^\mr$ is injective on the non-degenerate $\mr$-continuous covariant representations. For the moment we are interested in the preservation of non-degeneracy, the set of closed invariant subspaces and the Banach space of intertwining operators under this map. As a first step, we consider these issues for the assignment $\covrep \mapsto \intform$, for a still arbitrary continuous covariant representation $\covrep$. In order to do this, we introduce four maps which will be very useful later on as well. Two of these ($j_A$ and $j_G$ below) will only be needed in the involutive case.

\begin{prop}\label{p:i_j_A_G_properties}
 Let $\dynsys$ be a Banach algebra dynamical system. Define maps $i_A, j_A: A \to \End(C_c(G,A))$ and $i_G, j_G: G \to \End(C_c(G,A))$ by
\begin{align}\label{e:i_definitions}
 [i_A(a)f](s) &:= af(s),  \quad \quad \quad \quad [j_A(a)f](s) := f(s) \alpha_s(a), \\
 [i_G(r)f](s) &:= \alpha_r(f(r^{-1}s)), \quad \; [j_G(r)f](s) := \Delta(r^{-1}) f(sr^{-1}). \notag
\end{align}
for $a \in A$, $r \in G$ and $f \in C_c(G,A)$. Then $i_A$ and $i_G$ and homomorphisms and $j_A$ and $j_G$ are anti-homomorphisms. If $\covrep$ is a continuous covariant representation of $\dynsys$, then, for $a \in A$, $r \in G$, and $f \in C_c(G,A)$,
\begin{align}\label{e:intforms_of_i_j_A_G}
 \intform(i_A(a)f) &= \pi(a) \circ \intform(f), \quad \, \intform(j_A(a)f) = \intform(f) \circ \pi(a), \\
\intform(i_G(r)f) &= U_r \circ \intform(f), \quad \quad \intform(j_G(r)f) = \intform(f) \circ U_r . \notag
\end{align}
\end{prop}

It is actually true that $i_A$ and $i_G$ map into the left centralizers of $C_c(G,A)$ and that $j_A$ and $j_G$ map into the right centralizers of $C_c(G,A)$. The former will be shown during the proof of Proposition~\ref{p:covariant_representation_on_crossed_product} and the proof of the latter is similar, cf.\ Proposition~\ref{p:non_degenerate_anti_covariant_anti_representation_on_crossed_product}.

\begin{proof}
It is easy to check that the maps are (anti-)homomorphisms. Let $a \in A$ and $f \in C_c(G,A)$ and let $\covrep$ be a continuous covariant representation, then using the covariance,
\begin{align*}
 \intform(j_A(a)f) &= \int_G \pi[(j_A(a)f)(s)] U_s \ds \\
&= \int_G \pi(f(s)) \pi(\alpha_s(a)) U_s \ds \\
&= \int_G \pi(f(s)) U_s \pi(a) \ds \\
&= \intform(f) \circ \pi(a),
\end{align*}
and for $r \in G$ we obtain
\begin{align*}
 \intform(i_G(r)f) &= \int_G \pi[(i_G(r)f)(s)] U_s \ds \\
&= \int_G \pi[ \alpha_r(f(r^{-1}s)) ] U_s \ds \\
&= \int_G \pi[ \alpha_r(f(s)) ] U_r U_s \ds \\
&= \int_G U_r \pi(f(s)) U_s \ds \\
&= U_r \circ \intform(f).
\end{align*}
The other computations are similar and will be omitted.
\end{proof}

Before we continue we need a preparatory result, in which the version for $j_A$ will not be applied immediately, but will be useful later on.

\begin{lemma}\label{l:i_j_A_continuous_in_inductive}
 Let $\dynsys$ be a Banach algebra dynamical system with a bounded approximate left \textup{(}resp.\ right\textup{)} identity $(u_i)$, and let $f \in C_c(G,A))$. Then $i_A(u_i)f$ \textup{(}resp.\ $j_A(u_i)f$\textup{)} converges to $f$ in the inductive limit topology of $C_c(G,A)$.
\end{lemma}

\begin{proof}
Starting with the left version, we note that Lemma~\ref{l:density_lemma_inductive_limit_topology} implies easily that it is sufficient to prove the statement for elementary tensors. So let $f = z \otimes a \in C_c(G) \otimes A$. Then $[i_A(u_i)f](s) = z(s) u_i a$, which converges uniformly to $z(s) a = f(s)$ on $G$.

As to the right version, again Lemma~\ref{l:density_lemma_inductive_limit_topology}, when combined with the observation that the operators $\alpha_s\in B(A)$ are uniformly bounded as $s$ ranges over a compact subset of $G$, implies that it is sufficient to prove the statement for $f = z \otimes a \in C_c(G) \otimes A$. Let $\eps > 0$. Then, for $s\in G$, $\norm{[j_A(u_i)f](s)-f(s)}=\norm{z(s)a\alpha_s(u_i)-z(s)a}\leq |z(s)|\norm{a\alpha_s(u_i)-a}$. For $s\notin\supp(z)$, the right hand side is zero, for all $i$. Lemma~\ref{l:approximation_on_compacta_involving_alpha_and_approximate_right_identity} shows that there exists an index $i_0$ such that, for all $i\geq i_0$, the right hand side is less than $\eps$, for all $s\in \supp(z)$, and $i\geq i_0$. Hence $\Vert j_A(u_i)f-f\Vert_\infty\to 0$. Therefore, $j_A(u_i)f \to f$ in the inductive limit topology of $C_c(G,A)$.
\end{proof}

\begin{prop}\label{p:integrated_form_is_non_degenerate_rep_of_C_c(G,A)}
 Let $\dynsys$ be a Banach algebra dynamical system and let $\covrep$ be a continuous covariant representation of $\dynsys$ on the Banach space $X$.
\begin{enumerate}
 \item If $\covrep$ is non-degenerate, then $\intform$ is a non-degenerate representation of $C_c(G,A)$. If $A$ has a bounded approximate left identity, the converse also holds.
 \item If $Y$ is a closed subspace of $X$ which is invariant for both $\pi$ and $U$, then $Y$ is invariant for $\intform$.
 \item If $Y$ is a Banach space, $(\rho, V)$ a continuous covariant representation on $Y$ and $\Phi: X \to Y$ a bounded intertwining operator between $\covrep$ and $(\rho,V)$, then $\Phi$ is a bounded intertwining operator between $\intform$ and $\rho \rt V$. If $\covrep$ is non-degenerate, the converse also holds.
 \item If $\dynsys$ and $\covrep$ are involutive, then so is $\intform$.
\end{enumerate}
\end{prop}

\begin{proof}
\begin{enumerate}
 \item Suppose $0 \not= x \in X$ is of the form $x = \pi(a)y$, and let $\eps > 0$. By the strong continuity of $U$ there exists a neighbourhood $V$ of $e$ such that $s \in V$ implies that $\norm{U_s y - y} < \eps / \norm{\pi(a)}$. Let $z \in C_c(G)$ be nonnegative with compact support contained in $V$ and with integral equal to 1. Define $f := z \otimes a \in C_c(G,A)$, then
\begin{align*}
 \norm{\intform(f)y - x} &= \norm{ \int_G \pi(f(s)) U_s y \ds - \int_G z(s)x \ds} \\
&\leq \int_G z(s) \norm{\pi(a) U_s y - \pi(a)y} \,ds \\
&\leq \int_G z(s) \norm{\pi(a)} \norm{U_s y - y} \,ds \\
&\leq \int_{\supp(z)} z(s) \norm{\pi(a)} \frac{\eps}{\norm{\pi(a)}} \ds = \eps.
\end{align*}
This implies that $\overline{\intform(C_c(G,A))\cdot X} \supset \overline{\pi(A)\cdot X}$. Therefore, if $\pi$ is non-degenerate, then so is $\intform$.

For the converse, let $(u_i)$ be a bounded approximate left identity of $A$. By Remark~\ref{r:strong_continuity_approx_identity} we have to show that $\pi(u_i)x \to x$, for all $x \in X$. By the boundedness of $\pi$ and $(u_i)$ and an easy $3 \eps$-argument, it is sufficient to show this for $x$ in a dense subset of $X$. For this we choose $\intform(C_c(G,A))\cdot X$, which is dense in $X$ by assumption. Let $f \in C_c(G,A)$ and $y \in X$, then using \eqref{e:intforms_of_i_j_A_G},
\[  \pi(u_i) \intform(f)y = \intform(i_A(u_i)f)y \to \intform(f)y \]
by Lemma~\ref{l:i_j_A_continuous_in_inductive} and the continuity of $\intform$ in the inductive limit topology (Lemma~\ref{l:quotient_map_continuous_in_inductive_limit_topology_on_C_c(G,A)}).
By linearity, this implies that  $\pi(u_i)x \to x$ for all $x\in\intform(C_c(G,A)) \cdot X$, as desired.

\item If $Y$ is a closed subspace invariant for both $\pi$ and $U$, then it is immediate from the properties of our vector-valued integral that $Y$ is also invariant under $\intform(C_c(G,A))$.

\item Let $\Phi: X \to Y$ be a bounded intertwining operator between $\covrep$ and $(\rho, V)$. Then for $x \in X$ and $f \in C_c(G,A)$ we have
\begin{align*}
 \Phi \circ \intform(f) &= \Phi \circ \int_G \pi(f(s)) U_s \ds \\
&= \int_G \Phi \, \pi(f(s)) U_s \ds \\
&= \int_G \rho(f(s)) \, \Phi \, U_s \ds \\
&= \int_G \rho(f(s)) V_s \, \Phi \ds \\
&= \int_G \rho(f(s)) V_s \ds \circ \Phi \\
&= \rho \rt V(f) \circ \Phi.
\end{align*}

Conversely, suppose that $\Phi: X \to Y$ is a bounded intertwining operator for $\aintform$ and $V \rt \rho$ and that $\covrep$ is non-degenerate. For elements of $X$ of the form $\intform(f)x$, where $x \in X$ and $f \in C_c(G,A)$, we obtain for $r \in G$, using \eqref{e:intforms_of_i_j_A_G},
\begin{align*}
 [\Phi \circ U_r] (\intform(f)x) &= [\Phi \circ U_r \circ (\intform)(f)] x \\
&= [\Phi \circ (\intform)(i_G(r)f)] x \\
&= [(\rho \rt V)(i_G(r)f) \circ \Phi] x \\
&= [V_r \circ (\rho \rt V)(f) \circ \Phi] x \\
&= [V_r \circ \Phi \circ (\intform)(f)] x \\
&= [V_r \circ \Phi] (\intform(f)x).
\end{align*}
By $(i)$, $\intform$ is non-degenerate, and so $\Phi \circ U_r$ and $V_r \circ \Phi$ agree on a dense subset of $X$ and hence are equal. Similarly we obtain that $\Phi \circ \pi(a) = \rho(a) \circ \Phi$ for all $a \in A$.

\item This has been shown in Section~\ref{sec:construction_and_basic_properties}, following \eqref{d:intformdef}.
\end{enumerate}
\end{proof}

Together with Lemma~\ref{l:completion_associated_with_seminorm} the above proposition immediately leads to most the following.

\begin{theorem}\label{t:summary_from_dyn_sys_to_crossed_product}
Let $\dynsys$ be a Banach algebra dynamical system, and let $\mathcal{R}$ be a non-empty uniformly bounded class of continuous covariant representations. Consider the assignment $\covrep \to (\intform)^\mr$ from the $\mr$-continuous covariant representations of $\dynsys$ to the bounded representations of $\crosprod$.
\begin{enumerate}
 \item If $\covrep$ is non-degenerate, then $\intformr$ is a non-degenerate representation of $\crosprod$. If $A$ has a bounded approximate left identity, the converse also holds.
 \item If $Y$ is a closed subspace of $X$ which is invariant for both $\pi$ and $U$, then $Y$ is invariant for $\intformr$. If $\covrep$ is non-degenerate and $A$ has a bounded approximate left identity, the converse also holds.
 \item If $Y$ is a Banach space, $(\rho, V)$ a continuous covariant representation on $Y$ and $\Phi: X \to Y$ a bounded intertwining operator between $\covrep$ and $(\rho,V)$, then $\Phi$ is a bounded intertwining operator between $\intformr$ and $(\rho \rt V)^\mr$. If $\covrep$ is non-degenerate, the converse also holds.
 \item If $\dynsys$, $\mr$, and $\covrep$ are involutive, then so is $\intformr$.
\end{enumerate}

Furthermore, for a general Banach dynamical system, $\normr{\intform}=\norm{(\intform)^\mr}$, for each $\mr$-continuous covariant representation $\covrep$ of $\dynsys$.
\end{theorem}

\begin{proof}
 All that has to be shown is that if $Y$ is invariant for $\intformr$, $\covrep$ is non-degenerate and $A$ has a bounded approximate left identity, then $Y$ is invariant for $\covrep$. Under these assumptions, the first part of the theorem shows that $\intformr$ is non-degenerate. By Theorem \ref{t:crossed_product_has_bounded_approximate_identities}, $\crosprod$ has a bounded approximate left identity, so by Lemma \ref{l:alg_restriction_also_non_degenerate}, $\intformr|_Y$ is non-degenerate. Using the density of $\qr(C_c(G,A))$ in $\crosprod$, we obtain that
\[ \intform(C_c(G,A)) \cdot Y = \intformr \left( \qr(C_c(G,A)) \right) \cdot Y \]
is dense in $Y$. Now for $a \in A$, $r \in G$, $y \in Y$ and $f \in C_c(G,A)$ by \eqref{e:intforms_of_i_j_A_G},
\begin{align*}
 \pi(a) \circ \intform(f)y &= \intform(i_A(a)f)y \in Y \\
 U_r \circ \intform(f)y &= \intform(i_G(r)f)y \in Y,
\end{align*}
 so $\pi(a)$ and $U_r$ map a dense subset of $Y$ into $Y$, which proves the claim.
\end{proof}

\section{Centralizer algebras}\label{sec:centralizer_algebras}

The passage from non-degenerate bounded representations of $\crosprod$ to continuous covariant representations of $\dynsys$ in Section~\ref{sec:from_crosprod_to_dynsys} will be obtained using the left centralizer algebra of $\crosprod$. This will be done in Proposition~\ref{p:from_crossed_product_to_dyn_sys} below, and it consists of two steps. The idea is to first construct, by general means, a bounded representation of the left centralizer algebra of $\crosprod$ from a given non-degenerate bounded representation of $\crosprod$, and next to compose this new representation with covariant homomorphisms (to be constructed below) of $A$ and $G$ into this left centralizer algebra, thus obtaining (at least algebraically) a covariant representations of the group and the algebra.

In the present section, which is a preparation for the next, we start by recalling the basic general theorem which underlies the first step in the above procedure. This will make it obvious why it is so important that $\crosprod$ has a bounded left approximate identity if $A$ has one, something which is---as observed before--automatic in the $C^*$-case, but not in the general setting.  Next  we construct the homomorphisms needed for the second step. We also include some results for the double centralizer algebra of $\crosprod$; these will be needed for the involutive case only.

Commencing with representations of a general normed algebra and its centralizer algebras, we let $\ma$ be a normed algebra: the results below will be applied with the Banach algebra $\ma = \crosprod$. We let $\LCA \subset B(\ma)$ denotes the unital normed algebra of left centralizers of $\ma$, i.e., the algebra of bounded operators $L: \ma \to \ma$ commuting with all right multiplications, or equivalently, satisfying $L(a)b = L(ab)$ for all $a,b \in \ma$. Every $a \in \ma$ determines a left centralizer by left multiplication, and we let $\lambda: \ma \to \LCA$ denotes the corresponding homomorphism. Likewise, the algebra $\RCA\subset B(\ma)$ denotes the unital normed algebra of right centralizers, i.e., the algebra of operators $R: \ma \to \ma$ commuting with all left multiplications, or equivalently, satisfying $R(ab) = aR(b)$ for all $a, b \in \ma$, and $\rho: \ma \to \RCA$ denotes the canonical anti-homomorphism. The unital normed algebra of double centralizers of $\ma$ is denoted by $\DCA$ and consists of pairs $(L,R)$, where $L$ is a left centralizer and $R$ is a right centralizer, such that $aL(b) = R(a)b$ for all $a,b \in \ma$. Multiplication in $\mathcal{M}(\ma)$ is defined by $(L_1, R_1) (L_2, R_2) = (L_1 L_2, R_2 R_1)$ and the norm by $\norm{(L,R)}_{\mathcal{M}(\ma)} = \max(\norm{L}, \norm{R})$. We let $\phi_l: \DCA \to \LCA$ denote the contractive unital homomorphism $(L, R) \mapsto L$, and $\delta: \mathcal{A} \to \DCA$ denote the homomorphism $a \mapsto (\lambda(a), \rho(a))$.

If $L$ is invertible in $\LCA$ and $R$ is invertible in $\RCA$, then $(L,R)$ is invertible in $\DCA$ with inverse $(L,R)^{-1} = (L^{-1}, R^{-1})$. If $\ma$ has a bounded involution, then for $L \in \mathcal{M}_l(\ma)$ the map $L^*: \ma \to \ma$ defined by $L^*(a) := (L(a^*))^*$ is a right centralizer, and for $R \in \mathcal{M}_r(\ma)$ the map $R^*$ defined by $R^*(a) := (R(a^*))^*$ is a left centralizer. Furthermore $(L^*)^*=L$ and $(R^*)^*=R$. As a consequence, the map $(L,R) \mapsto (R^*, L^*)$ is a bounded involution on $\mathcal{M}(\ma)$.

Obviously, if $\ma$ is a Banach algebra, then so are $\LCA$, $\RCA$, and $\DCA$.

In the following theorem we collect a few results from \cite[Remark~2.2, Theorem~4.1 and Theorem~4.5]{extendart}. The constant $M_l^{\ma}$ in it is was defined in Definition~\ref{d:infimum_of_bounds} as the infimum of the upper bounds of all approximate left identities. The theorem implies, in particular, that, given a non-degenerate bounded Banach space representation of a normed algebra with a bounded approximate left identity, there exists a unique representation (which is then automatically bounded and non-degenerate) of its left centralizer algebra which is compatible with the canonical homomorphism $\lambda: \ma \to \LCA$. This is a crucial step in our approach, and it should be thought of as the analogue of extending a representation of a $C^*$-algebra to its multiplier algebra.

\begin{theorem}\label{t:summary_for_centralizers}
Let $\mathcal{A}$ be a normed algebra with a bounded approximate left identity, and let $X$ be a Banach space.

If $T: \mathcal{A} \to \mathcal B(X)$ is a non-degenerate bounded representation, then there exists a unique homomorphism $\overline{T}: \LCA \to \mathcal B(X)$ such that the diagram
\begin{equation}\label{diag:left_centralizer_diagram}
\xymatrix{
\mathcal{A} \ar[r]^T \ar[rd]_{\lambda} \ar[d]^\delta & \mathcal B(X) \\
\DCA \ar[r]_{\phi_l} & \LCA \ar[u]_{\overline{T}}
}
\end{equation}
is commutative. All maps in the diagram are bounded homomorphisms, and $\overline{T}$ is unital. One has $\norm{\overline{T}} \leq M_l^{\ma} \norm{T}$, which implies $\norm{\overline{T} \circ \phi_l} \leq M_l^{\ma} \norm{T}$. In particular, $\overline{T}$ and $\overline{T} \circ \phi_l$ are non-degenerate bounded representations of $\LCA$ and $\DCA$ on $X$.

The image $T(\mathcal{A})$ is a left ideal in $\overline{T}(\LCA)$. In fact, if $L \in \LCA$ and $a \in \mathcal{A}$, then
\begin{equation}\label{e:compatibility}
\overline{T}(L) \circ T(a) = T(L(a)).
\end{equation}

If $(u_i)$ is any bounded approximate left identity of $\mathcal{A}$ and if $L\in\LCA$, then for $x \in X$ we have
\begin{equation}\label{e:extension_as_sot_limit}
\overline{T}(L)x = \lim_i T(L(u_i))x.
\end{equation}
In particular, the set of closed invariant subspaces of $T$ coincides with the set of closed invariant subspaces of $\overline{T}$, and if $S: \mathcal{A} \to B(X)$ is another non-degenerate bounded representation, then the set of bounded intertwining operators of $T$ and $S$ coincides with the set of bounded intertwining operators of $\overline{T}$ and $\overline{S}$.

If in addition $\mathcal{A}$ has a bounded involution, $X$ is a Hilbert space and $T$ is involutive, then $\overline{T} \circ \phi_l$ is involutive.
\end{theorem}

If, returning to our original context, $\dynsys$ is a Banach algebra dynamical system, and $\mr$ is a non-empty uniformly bounded class of continuous covariant representations, then each $\covrep$ in $\mr$, being obviously $\mr$-continuous, yields a bounded (even contractive) representation $(\intform)^\mr$ of $\crosprod$, and if $\covrep\in\mr$ is non-degenerate, then $(\intform)^\mr$ is non-degenerate as well, by Theorem~\ref{t:summary_from_dyn_sys_to_crossed_product}. If, in addition, $A$ has a bounded approximate left identity, then $\crosprod$ has a bounded approximate left identity by Corollary~\ref{c:approximate_identities}, hence Theorem~\ref{t:summary_for_centralizers} provides a bounded representation $\overline{(\intform)^\mr}$ of $\mathcal M_l(\crosprod)$. These representations $\overline{(\intform)^\mr}$, for $\covrep\in\mr$, are used in the following result, which is a parallel of the separation property in Proposition~\ref{p:norm_formula_in_crossed_product}.

\begin{prop}\label{p:representations_separate_left_centralizers}
 Let $\dynsys$ be a Banach algebra dynamical system, where $A$ has a bounded approximate left identity, and let $\mathcal{R}$ be a non-empty uniformly bounded class of non-degenerate continuous covariant representations. Then the non-degenerate bounded representations $\overline{(\intform)^\mr}$ of $\mathcal M_l(\crosprod)$, for $\covrep \in \mathcal{R}$, separate the points of $\mathcal M_l(\crosprod)$.
\end{prop}

\begin{proof}
Let $L\in\mathcal M_l(\crosprod)$ be such that $\overline{(\intform)^\mr} (L)=0$, for all $\covrep \in \mathcal{R}$. Then, for arbitrary $c\in\crosprod$, the combination of Proposition~\ref{p:norm_formula_in_crossed_product} and \eqref{e:compatibility} shows that
\begin{align*}
\normr{L(c)}&= \sup_{\covrep \in \mathcal{R}} \norm{(\intform)^\mr(L(c))}\\
&=\sup_{\covrep \in \mathcal{R}} \norm{\overline{(\intform)^\mr}(L)\circ(\intform)^\mr (c)}\\
&=0.
\end{align*}
Hence $L=0$.
\end{proof}

We continue our preparation for the representation theory in the next section by investigating a particular continuous covariant representation of $\dynsys$ in $\crosprod$, needed for the second step in the procedure outlined in the beginning of this section. An important feature, in view of Theorem~\ref{t:summary_for_centralizers}, of this particular continuous covariant representations is that the corresponding images of $A$ and $G$ are contained in the left centralizer algebra $\mathcal M_l(\crosprod)$ of $\crosprod$, so that it can be composed with representations of $M_l(\crosprod)$ resulting from the first step. We will now proceed to construct this continuous covariant representations, which is done using the actions of $A$ and $G$ on $C_c(G,A)$, as defined in \eqref{e:i_definitions}.

\begin{lemma}\label{l:i_action_sigma_bounded}
Let $\dynsys$ be a Banach algebra dynamical system, and let $\mr$ be a non-empty uniformly bounded class of continuous covariant representations. Let $a\in A$ and $r\in G$. Then the maps
\[ i_A(a), \; i_G(r): (C_c(G,A), \sigmar) \to (C_c(G,A), \sigmar) \]
are bounded. In fact,
\[ \normr{i_A(a)} \leq \sup_{\covrep\in\mr} \norm{\pi(a)}\leq \Cr\norm{a}, \]
and
\[ \normr{i_G(r)} \leq \nur(r).\]
\end{lemma}

\begin{proof}
Let $a\in A$. Then, for $f\in C_c(G,A)$ and $\covrep \in \mr_r$, using \eqref{e:intforms_of_i_j_A_G} in the first step, we find that
\begin{align*}
  \norm{\intform(i_A(a)f)} &= \norm{\pi(a) \circ \aintform(f)} \\
&\leq \left(\sup_{\covrep\in\mr} \norm{\pi(a)}\right) \left(\sup_{\covrep\in\mr} \norm{\intform(f)}\right)\\
&=\left(\sup_{\covrep\in\mr} \norm{\pi(a)}\right) \sigmar(f).
\end{align*}
Taking the supremum over $\covrep \in \mr$ implies the statement concerning $i_A(a)$. The statement concerning $i_G(r)$ follows similarly.
\end{proof}

As a consequence of the above proposition and Lemma~\ref{l:completion_associated_with_seminorm}, the operators $i_A(a)$ and $i_G(r)$ yield bounded operators from $\crosprod$ to itself with the same norm. For typographical reasons, we will denote these elements of $B(\crosprod)$ by $i_A^\mr(a)$ and $i_G^\mr(r)$ rather than $i_A(a)^\mr$ and $i_G(r)^\mr$. Hence, if $a \in A$ and $r \in G$, then $i_A^\mr(a), \; i_G^\mr(r) \in B(\crosprod)$ are determined by
\begin{align}\label{e:i_j_A_G^r_definition}
i_A^\mr(a)(\qr(f))=\qr(i_A(a)f), \quad i_G^\mr(r)(\qr(f)) = \qr(i_G(r)f)
\end{align}
for all $f\in C_c(G,A)$.

In Proposition~\ref{p:i_j_A_G_properties} we have noted that $i_A: A \to \End (C_c(G,A))$ and $i_G: G \to \End (C_c(G,A))$ are homomorphisms. As a consequence of \eqref{e:i_j_A_G^r_definition} and the density of $\qr(C_c(G,A))$ in $\crosprod$, the same is then true for $i_A^\mr : A \to B(\crosprod)$ and $i_G^\mr: G \to B(\crosprod)$. Hence we have a pair of representations $(i_A^\mr, i_G^\mr)$ on $\crosprod$.

\begin{prop}\label{p:covariant_representation_on_crossed_product}
Let $\dynsys$ be a Banach algebra dynamical system, and let $\mathcal{R}$ be a non-empty uniformly bounded class of continuous covariant representations. Then $(i_A^\mr, i_G^\mr)$, as defined by \eqref{e:i_j_A_G^r_definition}, is a continuous covariant representation of $\dynsys$ in $\crosprod$. The images $i_A^\mr(A)$ and $i_G^\mr(G)$ are contained in the left centralizer algebra $\mathcal M_l(\crosprod)$ of $\crosprod$, so we have
\begin{align*}
 i_A^\mr: A \to \mathcal M_l(\crosprod) \subset B(\crosprod), \\
 i_G^\mr: G \to \mathcal M_l(\crosprod) \subset B(\crosprod).
\end{align*}
 For the operator norm in $B(\crosprod)$ the estimates
\[ \norm{i_A^\mr(a)} \leq \sup_{\covrep\in\mr} \norm{\pi(a)}\leq \Cr\norm{a}, \]
where $a\in A$,
and
\[ \norm{i_G^\mr(r)} \leq \nur(r),\]
where $r\in G$, hold.

If, in addition, $A$ has a bounded approximate left identity, then $(i_A^\mr, i_G^\mr)$ is non-degenerate.

\end{prop}

Although it does not follow from the estimates for the operator norm in Proposition~\ref{p:covariant_representation_on_crossed_product}, if $A$ has a bounded approximate left identity and all elements of $\mr$ are non-degenerate, then it is actually true that $(i_A^\mr, i_G^\mr)$ is $\mr$-continuous, see Theorem~\ref{t:integrated_form_of_representation_as_left_centralizers_is_canonical}. Proving this will require some extra effort, and we will only be able to do so once more information has been obtained about the relation between $\mr$-continuous covariant representations of $\dynsys$ and bounded representations of $\crosprod$.

\begin{proof}
We start by proving the covariance of $(i_A^\mr, i_G^\mr)$. For this it is sufficient to show that the pair $(i_A, i_G)$ is covariant, i.e., that $[i_G(r)i_A(a)i_G(r)^{-1}f](s) = [i_A(\alpha_r(a))f](s)$ for all $f \in C_c(G,A)$, $r,s \in G$, and $a \in A$. Indeed,
\begin{align*}
 [i_G(r)i_A(a)i_G(r)^{-1}f](s) &= \alpha_r[(i_A(a)i_G(r)^{-1}f)(r^{-1}s)]\\
&=\alpha_r[ai_G(r)^{-1}f(r^{-1}s)]\\
&=\alpha_r[a\alpha_{r^{-1}}(f(rr^{-1}s))]\\
&=\alpha_r(a)(f(s))\\
&=[i_A(\alpha_r(a))f](s).
\end{align*}

We continue by showing that the bounded operators $i_A^\mr(a)$ and $i_G^\mr$ on $\crosprod$ are left centralizers of the Banach algebra $\crosprod$. To see this, let $a\in A$. Then, for $f,g \in C_c(G,A)$ and $s\in G$,
\begin{align*}
 [i_A(a)(f * g)](s) &= a \int_G f(r) \alpha_r(g(r^{-1}s)) \,dr \\
&= \int_G af(r) \alpha_r(g(r^{-1}s)) \,dr \\
&= [(i_A(a)f) * g](s).
\end{align*}
So $i_A(a)$ commutes with right multiplication in $C_c(G,A)$. Hence $i_A^\mr(a)(\qr(f)*\qr(g))=i_A^\mr(a)(\qr(f*g))=\qr(i_A(a)(f*g))=\qr([i_A(a)f]*g)=\qr(i_A(a)f)*\qr(g)=[i_A(a)^\mr\qr(f)]*\qr(g)$, for $f,g\in C_c(G,A)$.
From the density of $\qr(C_c(G,A)$ in $\crosprod$ and the boundedness of $i_A^\mr(a)$ it then follows that $i_A^\mr(a)$ is a left centralizer of $\crosprod$. As to the other case, let $r \in G$. Then, for $f,g \in C_c(G,A)$ and $s\in G$,
\begin{align*}
 [i_G(r)(f * g)](s) &= \alpha_r ([f * g](r^{-1}s)) \\
&= \alpha_r \left( \int_G f(t) \alpha_t(g(t^{-1}r^{-1}s)) \,dt \right)  \\
&= \int_G \alpha_r(f(t)) \alpha_{rt}(g((rt)^{-1}s)) \,dt \\
&= \int_G \alpha_r(f(r^{-1}t))\alpha_t(g(t^{-1}s)) \,dt \\
&= [(i_G(r)f) * g](s).
\end{align*}
So $i_G(r)$ commutes with right multiplication in $C_c(G,A)$. As for $i_A(a)$, it follows that $i_G^\mr(r)$ is a left centralizer of $\crosprod$.

Next, we will show that $i_G^\mr$ is strongly continuous. In view of the boundedness of $i_G^\mr$ on compact neighbourhoods of $e$, Corollary~\ref{c:strongly_continuous_at_e} implies that we only have to show strong continuity of $i_G^\mr$ in $e$ on a dense subset of $\crosprod$. By Corollary~\ref{c:image_of_tensor_product_dense_in_crossed_product}, $\qr(C_c(G) \otimes A)$ is dense in $\crosprod$, and so it is sufficient to show that $\sigmar(i_G(r_i)f - f) \to 0$ for all $f \in C_c(G) \otimes A$, whenever $r_i \to e$ in $G$. By linearity it is sufficient to consider only elements of the form $z \otimes a$ with $z \in C_c(G)$ and $a \in A$. Therefore, fix $z\otimes a$ and let $r_i \to e$. We may assume that the $r_i$ are all contained in a fixed compact set. It is the obvious that the net $(i_G(r_i)(z \otimes a))$ is likewise supported in a fixed compact set, so by Lemma~\ref{l:quotient_map_continuous_in_inductive_limit_topology_on_C_c(G,A)} it suffices to show that $[i_G(r_i)(z \otimes a)](s) - z(s)a \to 0$,
uniformly in $s$. Since
\begin{align*}
 \norm{[i_G(r_i)(z \otimes a)](s) - z(s)a}&=\norm{z(r_i^{-1}s)\alpha_{r_i}(a)-z(s)a}\\
&\leq \norm{z(r_i^{-1}s)\alpha_{r_i}(a)-z(r_i^{-1}s)a} + \norm{z(r_i^{-1}s)a-z(s)a}\\
&\leq \|z\|_\infty\norm{\alpha_{r_i}(a)-a} + \norm{z(r_i^{-1}s)-z(s)} \norm{a},
\end{align*}
this uniform convergence follows from the strong continuity of $\alpha$ and the uniform continuity of $z$.
Together with the discussion preceding the theorem, this concludes the proof that $(i_A^\mr, i_G^\mr)$ is a continuous covariant representation of $\dynsys$ on $\crosprod$.

If, in addition, $A$ has a bounded approximate left identity $(u_i)$, then, for $f\in C_c(G)\otimes A$, Lemma~\ref{l:i_j_A_continuous_in_inductive} shows that $i_A(u_i)f\to f$ in the inductive limit topology. As a consequence, $i_A(A)\cdot C_c(G)\otimes A$ is dense in $C_c(G)\otimes A$ in the inductive limit topology. By Lemma~\ref{l:quotient_map_continuous_in_inductive_limit_topology_on_C_c(G,A)}, $i_A^\mr(A)\cdot\qr(C_c(G)\otimes A)=\qr(i_A(A)\cdot C_c(G)\otimes A)$ is dense in $\qr(C_c(G)\otimes A)$. Since the latter is dense in $\crosprod$ by Corollary~\ref{c:image_of_tensor_product_dense_in_crossed_product}, $i_A^\mr$ is thus seen to be non-degenerate.
\end{proof}

The above Proposition~\ref{p:covariant_representation_on_crossed_product} is sufficient for the sequel in the case of general Banach algebra dynamical systems. In the involutive case, the left centralizer algebra alone is no longer sufficient, because of the lack of an involutive structure. In that case, we will use the double centralizer algebra, and in order to establish the result for the double centralizer algebra that will eventually be used, we first need the following right-sided version of part of the above theorem.

\begin{prop}\label{p:non_degenerate_anti_covariant_anti_representation_on_crossed_product}
Let $\dynsys$ be a Banach algebra dynamical system and let $\mr$ be a non-empty uniformly bounded class of continuous covariant representations. For $a \in A$ and $r \in G$, let $j_A(a)$ and $j_G(r)$ be as in \eqref{e:i_j_A_G^r_definition}. Then the maps
\[ j_A(a), \; j_G(r): (C_c(G,A), \sigmar) \to (C_c(G,A), \sigmar) \]
are bounded. Denote the corresponding bounded operators on $\crosprod$ by $j_A^\mr(a)$ and $j_G^\mr(r)$, respectively, determined by $j_A^\mr(a)(\qr(f))=\qr(j_A(a)(f))$, for all $f\in C_c(G,A)$, and by $j_G^\mr(r)(\qr(f))=\qr(j_G(r)f)$, for all $f\in C_c(G,A)$, respectively.

Then $j_A^\mr: A \to B(\crosprod)$ is a bounded anti-representation of $A$ in $\crosprod$, and $j_G^\mr: G\to B(\crosprod)$ is a strongly continuous anti-representation of $G$ in $\crosprod$.
The pair $(j_A^\mr, j_G^\mr)$ is anti-covariant in the sense that, for all $a \in A$ and all $r \in G$,
\[
j_A^\mr(\alpha_r(a)) = j_G^\mr(r)^{-1} j_A^\mr(a) j_G^\mr(r).
\]
The images $j_A^\mr(A)$ and $j_G^\mr(G)$ are contained in the right centralizer algebra $\mathcal M_r(\crosprod)$ of $\crosprod$, so we have
\begin{align*}
 j_A^\mr: A \to \mathcal M_r(\crosprod) \subset B(\crosprod), \\
 j_G^\mr: G \to \mathcal M_r(\crosprod) \subset B(\crosprod).
\end{align*}
 For the operator norm in $B(\crosprod)$ the estimates
\[ \norm{j_A^\mr(a)} \leq \sup_{\covrep \in \mr} \norm{\pi(a)} \leq \Cr\norm{a}, \]
where $a\in A$, and
\[\norm{j_G^\mr(r)} \leq \nur(r),\]
where $r\in G$, hold.

If, in addition, $A$ has a bounded approximate right identity, then $j_A^\mr$ is non-degenerate.
\end{prop}

\begin{proof}
The proof is similar to the proof of the corresponding statements in Proposition~\ref{p:covariant_representation_on_crossed_product} and the details are therefore omitted.
\end{proof}

\begin{prop}\label{p:covariant_homomorphisms_into_double_centralizer_algebra}
Let $\dynsys$ be a Banach algebra dynamical system and let $\mr$ be a non-empty uniformly bounded class of continuous covariant representations of continuous covariant representations. For $a \in A$ and $r \in G$, let $i_A^\mr(a)$ and $i_G^\mr(r)$ be as in Proposition~\ref{p:covariant_representation_on_crossed_product}, and let $j_A^\mr(a)$ and $j_G^\mr(r)$ be as in Proposition~\ref{p:non_degenerate_anti_covariant_anti_representation_on_crossed_product}. Then $((i_A^\mr(a), j_A^\mr(a))$ and $(i_G^\mr(r), j_G^\mr(r))$ are both double centralizers of $\crosprod$, and we have
\[
 \norm{(i_A^\mr(a), j_A^\mr(a))} \leq \sup_{\covrep \in \mr} \norm{\pi(a)} \leq \Cr\norm{a},
\]
and
\[
 \norm{(i_G^\mr(r), j_G^\mr(r))} \leq \nur(r).
\]
Furthermore, the maps $a \mapsto (i_A^\mr(a), j_A^\mr(a))$ and $r \mapsto (i_G^\mr(r), j_G^\mr(r))$ are homomorphisms of $A$ into $\mathcal M (\crosprod)$ and of $G$ into $\mathcal M (\crosprod)$, respectively, and
the pair $((i_A^\mr, j_A^\mr), (i_G^\mr, j_G^\mr))$ is covariant in the sense that
\[ (i_A^\mr(\alpha_r(a)), j_A^\mr(\alpha_r(a)) ) = (i_G^\mr(r), j_G^\mr(r))\cdot (i_A^\mr(a), j_A^\mr(a))\cdot (i_G^\mr(r), j_G^\mr(r))^{-1}, \]
for all $a \in A$ and all $r \in G$.

Moreover, if $\dynsys$ and $\mr$ are involutive, then
\[ (i_A^\mr, j_A^\mr): A \to \mathcal M(\crosprod) \]
is an involutive homomorphism, and $(i_G^\mr(r), j_G^\mr(r))^* = (i_G^\mr(r^{-1}), j_G^\mr(r^{-1}))$, for all $r\in G$.
\end{prop}

\begin{proof}
Let $a\in A$ and suppose $f,g \in C_c(G,A)$. Then the computation, for $s\in G$,
\begin{align*}
 f * (i_A(a)g) (s) &= \int_G f(r) \alpha_r(i_A(a)g(r^{-1}s)) \, dr \\
&= \int_G f(r) \alpha_r(a) \alpha_r(g(r^{-1}s)) \,dr \\
&= (j_A(a)f) * g (s)
\end{align*}
shows that $(i_A(a), j_A(a))$ is a double centralizer of $C_c(G,A)$. By continuity and density, the same holds for $(i_A^\mr(a), j_A^\mr(a))$ and $\crosprod$. Similarly, if $r,s\in G$ and $f,g\in C_c(G,A)$, then
\begin{align*}
 f * (i_G(r)g) (s) &= \int_G f(t) \alpha_t((i_G(r)g)(t^{-1}s)) \, dt \\
&= \int_G f(t) \alpha_t ( \alpha_r(g(r^{-1}t^{-1}s))) \,dt \\
&= \int_G f(t) \alpha_{tr}( g((tr)^{-1}s)) \, dt \\
&= \int_G \Delta(r^{-1}) f(tr^{-1}) \alpha_t ( g(t^{-1}s)) \, dt \\
&= \int_G (j_G(r)f) (t) \alpha_t(g(t^{-1}s)) \, dt\\
&= (j_G(r)f) * g (s)
\end{align*}
implies that $(i_G^\mr(r), j_G^\mr(r))$ is a double centralizer of $\crosprod$.

The fact that the maps are homomorphisms and the covariance property follow directly from the corresponding statements in Proposition~\ref{p:covariant_representation_on_crossed_product} and Proposition~\ref{p:non_degenerate_anti_covariant_anti_representation_on_crossed_product}, and the definition of the inverse and the multiplication in the double centralizer algebra.

As to the final statement, suppose that $\dynsys$ is involutive, and that $\mathcal{R}$ consists of involutive representations. To show that the homomorphism $(i_A, j_A)$ from $A$ into $\mathcal M(\crosprod)$ is involutive, we have to show that, for $a\in A$, $(i_A^\mr(a),j_A^\mr(a))^* = (i_A^\mr(a^*), j_A^\mr(a^*))$, i.e., that  $(j_A^\mr(a)^*,i_A^\mr(a)^*)=(i_A^\mr(a^*), j_A^\mr(a^*))$. Recalling the definitions \eqref{e:involution_definition} and \eqref{e:i_j_A_G^r_definition}, we find, for $f\in C_c(G,A)$ and $s\in G$, that
\begin{align*}
\left[ j_A(a)^* f\right](s) &= \left[j_A(a)f^*\right]^*(s) \\
&= \Delta(s^{-1}) \alpha_s\left[\left\{(j_A(a)f^*)(s^{-1})\right\}^*\right]\\
&= \Delta(s^{-1}) \alpha_s\left[\left\{f^*(s^{-1})\alpha_{s^{-1}}(a)\right\}^*\right]\\
&= \Delta(s^{-1}) \alpha_s\left[\left\{\Delta(s)\alpha_{s^{-1}}(f(s)^*)\alpha_{s^{-1}}(a)\right\}^*\right] \\
&= a^* f(s)\\
&= [i_A(a^*)f](s).
\end{align*}
and
\begin{align*}
\left[i_A(a)^* f\right](s) &= \left[i_A(a)f^*\right]^*(s)\\
&=\Delta(s^{-1})\alpha_s\left[\left\{(i_A(a)f^*)(s^{-1})\right\}^*\right]\\
&=\Delta(s^{-1})\alpha_s\left[\left\{af^*(s^{-1})\right\}^*\right]\\
&=\Delta(s^{-1})\alpha_s\left[\left\{a\Delta(s)\alpha_{s^{-1}}(f(s)^*)\right\}^*\right]\\
&=f(s)\alpha_s(a^*)\\
&=\left[j_A(a^*)f\right](s),
\end{align*}
By continuity and density, this implies that $(j_A^\mr(a)^*,i_A^\mr(a)^*)=(i_A^\mr(a^*), j_A^\mr(a^*))$, as desired.

A similar unwinding of the definitions establishes, by continuity and density, that $i_G^\mr(r)^* = j_G^\mr(r^{-1})$, for all $r\in G$. Taking adjoints, this implies $j_G^\mr(r)^*=i_G^\mr(r^{-1})$, hence $(i_G^\mr(r), j_G^\mr(r))^*=(j_G^\mr(r)^*, (i_G^\mr(r)^*)=(i_G^\mr(r^{-1}), j_G^\mr(r^{-1}))$, for all $r\in G$.
\end{proof}

\section{Representations: from $\crosprod$ to $\dynsys$}\label{sec:from_crosprod_to_dynsys}

As already indicated in the previous section, Theorem~\ref{t:summary_for_centralizers} and  Proposition~\ref{p:covariant_representation_on_crossed_product} provide a means to generate a covariant representation of $\dynsys$ from a non-degenerate bounded representation of $\crosprod$, as follows. If $A$ has a bounded approximate left identity, then the same holds for $\crosprod$, by Corollary~\ref{c:approximate_identities}, and hence any non-degenerate bounded representation $T$ of $\crosprod$ yields a bounded representation $\overline T$ of $\mathcal M_l(\crosprod)$, by Theorem~\ref{t:summary_for_centralizers}. Since, by Proposition~\ref{p:covariant_representation_on_crossed_product}, the images $i_A^\mr(A)$ and $i_G^\mr(G)$ are contained in $\mathcal M_l(\crosprod)$, the pair of maps $(\overline T\circ i_A^\mr, \overline T\circ i_G^\mr)$ is meaningfully defined and will then be a covariant representation of $\dynsys$, since the covariance requirement is automatically satisfied as a consequence of the covariance property of $(i_A^\mr, i_G^\mr)$, the latter being part of Proposition~\ref{p:covariant_representation_on_crossed_product}. Some natural questions that arise are, e.g., whether this covariant representation is continuous, and, if so, whether it is $\mr$-continuous. We will now investigate these and related matters, and incorporate some of the results from Section~\ref{sec:from_dynsys_to_crosprod} (the passage in the other direction, from $\mr$-continuous covariant representations of $\dynsys$ to bounded representations of $\crosprod$) in the process. After that, the proofs of our main results in Section~\ref{sec:general_correspondence} will be a mere formality.

Recall from Definition~\ref{d:infimum_of_bounds} and Corollary~\ref{c:approximate_identities} that $M_l^\mr$ denotes the infimum of the upper bounds of the approximate left identities of $\crosprod$, with estimate $M_l^\mr \leq \Cr M_l^A N^\mr$, where $M_l^A$ denotes the infimum of the upper bounds of the approximate left identities of $A$.

\begin{prop}\label{p:from_crossed_product_to_dyn_sys}
Let $\dynsys$ be a Banach algebra dynamical system, where $A$ has a bounded approximate left identity, and let $\mr$ be a non-empty uniformly bounded class of continuous covariant representations. Let $(i_A^\mr, i_G^\mr)$ be the continuous covariant representation of $\dynsys$ on $\crosprod$, as in Proposition~\ref{p:covariant_representation_on_crossed_product}.

Suppose that $T$ is a non-degenerate bounded representation of $\crosprod$ in a Banach space $X$, and let $\overline T$ be the associated bounded representation of $\mathcal M_l(\crosprod)$ in $X$, as in Theorem~\ref{t:summary_for_centralizers}. Then the pair $(\overline{T} \circ i_A^\mr, \overline{T} \circ i_G^\mr)$ is a non-degenerate continuous covariant representation of $\dynsys$ in $X$.
For the operator norm on the bounded operators on the representation space the estimates
\begin{align*}
\norm{\left(\overline{T} \circ i_A^\mr\right)(a)} \leq M_l^\mr \norm{T} \sup_{\covrep\in \mr} \norm{\pi(a)} \leq M_l^\mr \norm{T} \Cr \norm{a},
\end{align*}
where $a \in A$, and
\begin{equation*}
\norm{\left(\overline{T} \circ i_G^\mr\right)(r)} \leq M_l^\mr \norm{T}  \nur(r),
\end{equation*}
where $r\in G$, hold.

If a closed subspace of $X$ is invariant for $T$, it is invariant for $\overline{T} \circ i_A^\mr$ and $\overline{T} \circ i_G^\mr$, and if $Y$ is a Banach space, $S: \crosprod \to B(Y)$ a representation and $\Phi \in B(X,Y)$ intertwines $T$ and $S$, then $\Phi$ intertwines $(\overline{T} \circ i_A^\mr, \overline{T} \circ i_G^\mr)$ and $(\overline{S} \circ i_A^\mr, \overline{S} \circ i_G^\mr)$.

If, in addition, $\dynsys$, $\mathcal{R}$, and $T$ are involutive, then $(\overline{T} \circ i_A^\mr, \overline{T} \circ i_G^\mr)$ is involutive.

Moreover, if, in the not necessarily involutive case, $\covrep$ is an non-degenerate $\mr$-continuous covariant representation of $\dynsys$, with corresponding non-degenerate bounded representation $(\intform)^\mr$ of $\crosprod$, then
\begin{equation}\label{e:reconstruction_formula}
\left( \overline{(\intform)^\mr} \circ i_A^\mr, \overline{(\intform)^\mr} \circ i_G^\mr \right) = (\pi,U).
\end{equation}
\end{prop}

Although it does not follow from the estimates for the operator norm in the theorem, if all elements of $\mr$ are non-degenerate, then it is (in analogy with the continuous covariant representation $(i_A^\mr, i_G^\mr)$ of $\dynsys$ in $\crosprod$), actually true that $(\overline{T} \circ i_A^\mr, \overline{T} \circ i_G^\mr)$ is $\mr$-continuous, see Theorem~\ref{t:r_affiliated_from_crossed_product_to_dyn_sys}.

Note that the final statement of the theorem implies the injectivity of the assignment $\covrep\to (\intform)^\mr$ on the non-degenerate $\mr$-continuous covariant representations if $A$ has a bounded approximate left identity, as was already announced following Definition~\ref{d:affiliated}.

\begin{proof}
Let $T$ be a non-degenerate bounded representation of $\crosprod$ in the Banach space $X$. As already remarked preceding the theorem, the definitions
\[
\pi := \overline{T}\circ i_A^\mr \quad \textup{and} \quad U := \overline{T}\circ i_G^\mr
\]
are meaningful and provide a covariant representation $\covrep$ of $\dynsys$.
We show that it has the properties as claimed, and start with the bounds for $\norm{\pi}$ and $\norm{U_r}$, for $r\in G$. Let $\eps > 0$, then $\crosprod$ has an $(M_l^\mr + \eps)$-bounded approximate left identity. Since Theorem~\ref{t:summary_for_centralizers} and Proposition~\ref{p:covariant_representation_on_crossed_product} provide a bound for $\norm{\overline T}$, $\norm{i_A^\mr(a)}$, and $\norm{i_G^\mr (r)}$, we have, for $a\in A$,
\begin{align*}
\norm{\pi(a)} &\leq \norm{\overline{T}} \norm{i_A^\mr(a)} \\
&\leq (M_l^\mr + \eps) \norm{T}  \sup_{(\rho,V)\in \mr} \norm{\rho(a)}\\
&\leq (M_l^\mr + \eps) \norm{T}  \Cr \norm{a}
\end{align*}
and, for $r\in G$,
\[
\norm{U_r} \leq \norm{\overline{T}} \norm{i_G^\mr(r)} \leq (M_l^\mr + \eps)\norm{T}  \nur(r).
\]
Since $\eps>0$ was arbitrary, this establishes the estimates in the theorem.

We have to prove that $\pi$ is non-degenerate and that $U$ is strongly continuous. Starting with $\pi$, by Remark~\ref{r:strong_continuity_approx_identity} it has to be shown that $\pi(u_i)x \to x$ for all $x \in X$, where $(u_i)$ is a bounded approximate left identity of $A$. By the boundedness of $\pi$, which we already established, and the boundedness of $(u_i)$, it is sufficient to establish this for $x$ in a dense subset of $X$. Now since $T$ is non-degenerate and $\qr(C_c(G)\otimes A)$ is dense in $\crosprod$ by Corollary~\ref{c:image_of_tensor_product_dense_in_crossed_product}, $T(\qr(C_c(G)\otimes A))\cdot X$ is dense in $X$. So let $x \in X$ and $f \in C_c(G) \otimes A$, then by \eqref{e:compatibility} in Theorem~\ref{t:summary_for_centralizers},
\begin{align*}
\pi(u_i) T(\qr(f))x &= \overline{T}(i_A^\mr(u_i))T(\qr(f))x\\
 &= T[i_A^\mr(u_i)(\qr(f))]x \\
 &= T[\qr(i_A(u_i)f)]y\\
 &\to T(\qr(f))x,
\end{align*}
where the last step is by Lemma~\ref{l:i_j_A_continuous_in_inductive}, Lemma~\ref{l:quotient_map_continuous_in_inductive_limit_topology_on_C_c(G,A)} and the boundedness of $T$.

Now we turn to the strong continuity of $U$, Since we have already established that $\norm{U_r}\leq\nur(r)$, for $r\in G$, and $\nur$ is bounded on compact sets, Corollary~\ref{c:strongly_continuous_at_e} shows that it is sufficient to show strong continuity of $U$ in $e$ when acting on a dense subset of $X$. For this set we choose $T(\crosprod) \cdot X$, which is dense by the non-degeneracy of $T$, and then by linearity it is sufficient to show strong continuity in $e$ when acting on elements of the form $T(c)y$, where $c \in \crosprod$, and $y\in X$. So let $x = T(c)y \in X$, and let $r_i\to e$.  Then by \eqref{e:compatibility} in Theorem~\ref{t:summary_for_centralizers} we find
\[
U_{r_i} x = \overline{T}(i_G^\mr(r_i))T(c)y= T(i_G^\mr(r_i)(c))y.
\]
By Proposition~\ref{p:covariant_representation_on_crossed_product}, $i_G^\mr$ is strongly continuous. Hence, by the continuity of $T$,
\[
U_{r_i} x = T(i_G^\mr(r_i)(c))y \to T(c)y=x,
\]
as required.

Suppose $Y$ is a closed invariant subspace of $X$ for $T$. By Theorem~\ref{t:summary_for_centralizers} $\overline{T}(L)y \in Y$ for all $y \in Y$ and $L \in \mathcal{M}_l(\crosprod)$. Applying this with $L = i_A(a)$ and $L = i_G(r)$, for $a \in A$ and $r \in G$, shows that $Y$ is invariant for $\overline{T} \circ i_A$ and $\overline{T} \circ i_G$.

If $Y$ is a Banach space, $S$ a non-degenerate bounded representation of $\crosprod$ in $Y$ and $\Phi$ a bounded intertwining operator for $T$ and $S$, then it follows from Theorem~\ref{t:summary_for_centralizers} that $\Phi \circ \overline{T}(L) = \overline{S}(L) \circ \Phi$ for all $L \in \mathcal{M}_l(\crosprod)$. Again applying this with $L = i_A(a)$ and $L = i_G(r)$, for $a \in A$ and $r \in G$, shows that $\Phi \circ [\overline{T} \circ i_A](a) = [\overline{S} \circ i_A](a) \circ \Phi$ and $\Phi \circ [\overline{T} \circ i_G](r) = [\overline{S} \circ i_G](r) \circ \Phi$.

Considering the statement on involutions, suppose that, in addition, $\dynsys$ and $\mathcal{R}$ are both involutive. Let $T$ be an involutive representation of $\crosprod$. By Proposition~\ref{p:covariant_homomorphisms_into_double_centralizer_algebra} the homomorphism $(i_A^\mr, j_A^\mr):A\to \mathcal M(\crosprod)$ is involutive and Theorem~\ref{t:summary_for_centralizers} shows that $\overline{T} \circ \phi_l$ is involutive. Combining these, we obtain that
\[
\pi=\overline{T} \circ i_A^\mr = \overline{T} \circ \left[\phi_l \circ (i_A^\mr, j_A^\mr)\right]=\left[\overline{T} \circ \phi_l\right] \circ (i_A^\mr, j_A^\mr)
\]
is an involutive representation of $A$. Finally, if $r\in G$, then using the involutive property of $\overline{T} \circ \phi_l$ again, as well as Proposition~\ref{p:covariant_homomorphisms_into_double_centralizer_algebra}, we see that
\begin{align*}
U_r^*&=\left[\overline T(i_G^\mr(r))\right]^*\\
&=\left[\left(\overline T\circ\phi_l\right)\left((i_G^\mr(r), j_G^\mr(r))\right)\right]^* \\
&=\left(\overline T\circ\phi_l\right)\left[ (i_G^\mr(r), j_G^\mr(r))^* \right] \\
&=\left(\overline T\circ\phi_l\right)\left[ (j_G^\mr(r)^*, i_G^\mr(r)^*)\right] \\
&=\left(\overline T\circ\phi_l\right) [(i_G^\mr(r^{-1}), j_G^\mr(r^{-1}))] \\
&=\overline T(i_G^\mr(r^{-1})) = U_{r^{-1}} = U_r^{-1}.
\end{align*}
 Hence $U$ is a unitary representation of $G$, and this completes the proof that $(\overline{T} \circ i_A^\mr, \overline{T} \circ i_G^\mr)$ is involutive.

To conclude with, we consider the final statement on the recovery of an $\mr$-continuous covariant representation $\covrep$ from $(\intform)^\mr$.
Starting with $\pi$, let $a\in A$. Then the compatibility equation \eqref{e:compatibility} in Theorem~\ref{t:summary_for_centralizers}, when applied with $L$ replaced with $i_A^\mr(a)$ and $a$ replaced with $\qr(f)$, for $f\in C_c(G,A)$, yields
\begin{align}\label{e:retrieving_pi_1}
\overline{(\intform)^\mr}(i_A^\mr(a))\circ(\intform)^\mr(\qr(f)) &= (\intform )^\mr (i_A^\mr(a)\qr(f)) \notag \\
&= \intform(i_A(a)f).
\end{align}
Take an element $x \in X$ of the form $x = \intform(f) y$, with $f\in C_c(G,A)$. Using \eqref{e:intforms_of_i_j_A_G}, we find that
\[
\pi(a)x = \pi(a) \circ \intform(f)y = \intform(i_A(a)f)y.
\]
Combining this with both sides of \eqref{e:retrieving_pi_1} acting on $y$, we see that, for such $x$,
\begin{equation}\label{e:retrieving_pi_2}
\overline{(\intform)^\mr}(i_A^\mr(a))x=\pi(a)x.
\end{equation}
By Proposition~\ref{p:integrated_form_is_non_degenerate_rep_of_C_c(G,A)} the linear span of elements of the form $\intform(f)y$, with $f\in C_c(G,A)$ and $y\in X$, is dense in $X$, and therefore \eqref{e:retrieving_pi_2} implies that $\pi(a)$ and $\overline{(\intform)^\mr}(i_A^\mr(a))$ are equal.

The proof that $\overline{(\intform)^\mr}(i_G(r)) = U_r$, for $r \in G$, is similar.
\end{proof}

The reconstruction formula \eqref{e:reconstruction_formula} will enable us to complete our results on the continuous covariant representation $(i_A^\mr , i_G^\mr)$ from Proposition~\ref{p:covariant_representation_on_crossed_product}, under the extra conditions that $A$ has a bounded approximate left identity and that all elements of $\mr$ are non-degenerate.

\begin{theorem}\label{t:integrated_form_of_representation_as_left_centralizers_is_canonical}
Let $\dynsys$ be a Banach algebra dynamical system, where $A$ has a bounded approximate left identity, and let $\mathcal{R}$ be a non-empty uniformly bounded class of non-degenerate continuous covariant representations. Then the non-degenerate continuous covariant representation $(i_A^\mr, i_G^\mr)$ of $\dynsys$ on $\crosprod$ from Proposition~\ref{p:covariant_representation_on_crossed_product} is $\mr$-continuous, and the associated non-degenerate bounded representation $(i_A^\mr \rt i_G^\mr)^\mr$ of $\crosprod$ on itself coincides with the left regular representation, and is therefore contractive.
\end{theorem}

\begin{proof}
We start by showing that $(i_A^\mr, i_G^\mr)$ is $\mr$-continuous. If $f\in C_c(G,A)$, then $i_A(f(s)) i_G(s)$ is a left centralizer for all $s \in G$, and hence commutes with all right multiplications. By \eqref{d:operators_through_integral} these right multiplications can be pulled through the integral, therefore $i_A^\mr \rt i_G^\mr(f) = \int_G i_A(f(s)) i_G(s) \ds$ commutes with all right multiplications as well, hence it is a left centralizer.

Let $\lambda$ denote the left regular representation of $\crosprod$. Then using \eqref{e:reconstruction_formula} in the fourth step, we find that, for all $\covrep\in\mr$,
\begin{align*}
\overline{(\intform)^\mr}\left(i_A^\mr\rt i_G^\mr (f)\right)&=
 \overline{(\intform)^\mr}\brackets{\int_G i_A^\mr(f(s))i_G^\mr(s)\,ds}\\ &= \int_G \overline{(\intform)^\mr}(i_A^\mr(f(s))i_G^\mr(s))\,ds \\
&= \int_G \overline{(\intform)^\mr}(i_A^\mr(f(s)))\cdot\overline{(\intform)^\mr}(i_G^\mr(s)) \,ds \\
&= \int_G \pi(f(s))U_s \,ds \\
&= \intform(f)\\
&=(\intform)^\mr(\qr(f))\\
&= \overline{(\intform)^\mr}(\lambda(\qr(f))),
\end{align*}
where diagram~\eqref{diag:left_centralizer_diagram} was used in the final step. By Proposition~\ref{p:representations_separate_left_centralizers} the representations $\overline{(\intform)^\mr}$ separate the points, and it follows that
\[
i_A^\mr\rt i_G^\mr (f)=\lambda(\qr(f)).
\]
Consequently, $\norm{i_A^\mr\rt i_G^\mr (f)}\leq\norm{\lambda}\normr{\qr(f)}=\norm{\lambda}\sigmar(f)$. We conclude that $i_A^\mr\rt i_G^\mr$ is $\mr$-continuous.

Next we consider the statement that $(i_A^\mr\rt i_G^\mr)^\mr$, the representation of $\crosprod$ on itself, which we now know to be defined as a consequence of the first part of the proof, is the left regular representation. Let $f\in C_c(G,A)$. Then, for all $\covrep\in\mr$, the above computation shows that $\overline{(\intform)^\mr}\left((i_A^\mr\rt i_G^\mr)^\mr(\qr(f))\right)=\overline{(\intform)^\mr}\left(i_A^\mr\rt i_G^\mr(f)\right)=\overline{(\intform)^\mr}(\lambda(\qr(f)))$, so again by the point-separating property of the representations $\overline{(\intform)^\mr}$ it follows that $(i_A^\mr\rt i_G^\mr)^\mr(\qr(f))=\lambda(\qr(f))$. By continuity and density, the statement follows.
\end{proof}

In turn, Theorem~\ref{t:integrated_form_of_representation_as_left_centralizers_is_canonical}
enables us to understand that, as already remarked after Proposition~\ref{p:from_crossed_product_to_dyn_sys}, the non-degenerate continuous covariant representation obtained in that proposition is actually $\mr$-continuous, under the extra condition that all elements of $\mr$ are non-degenerate. In that case, there is an associated bounded representation of the crossed product again, and the following result, in which some other main results of this section have been included again for future reference, shows additionally that this two-step process is the identity.

\begin{theorem}\label{t:r_affiliated_from_crossed_product_to_dyn_sys}
Let $\dynsys$ be a Banach algebra dynamical system, where $A$ has a bounded approximate left identity, and let $\mr$ be a non-empty uniformly bounded class of non-degenerate continuous covariant representations. Let $(i_A^\mr, i_G^\mr)$ be the non-degenerate $\mr$-continuous covariant representation of $\dynsys$ on $\crosprod$, as in Proposition~\ref{p:covariant_representation_on_crossed_product} and Theorem~\ref{t:integrated_form_of_representation_as_left_centralizers_is_canonical}.

Suppose that $T$ is a non-degenerate bounded representation of $\crosprod$ in a Banach space $X$, and let $\overline T$ be the associated representation of $\mathcal M_l(\crosprod)$ in $X$, as in Theorem~\ref{t:summary_for_centralizers}. Then the pair $(\overline{T} \circ i_A^\mr, \overline{T} \circ i_G^\mr)$ is a non-degenerate $\mr$-continuous covariant representation of $\dynsys$ in $X$, and the corresponding non-degenerate bounded representation $\left((\overline{T} \circ i_A^\mr)\rt (\overline{T} \circ i_G^\mr)\right)^\mr$ of $\crosprod$ in $X$ coincides with $T$. In particular, $\normr{(\overline{T} \circ i_A^\mr)\rt (\overline{T} \circ i_G^\mr)}=\norm{T}$.

If a closed subspace of $X$ is invariant for $T$, it is invariant for $\overline{T} \circ i_A$ and $\overline{T} \circ i_G$, and if $Y$ is a Banach space, $S: \crosprod \to B(Y)$ a representation and $\Phi \in B(X,Y)$ intertwines $T$ and $S$, then $\Phi$ intertwines $(\overline{T} \circ i_A^\mr, \overline{T} \circ i_G^\mr)$ and $(\overline{S} \circ i_A^\mr, \overline{S} \circ i_G^\mr)$.

If, in addition, $\dynsys$, $\mathcal{R}$, and $T$ are involutive, then $(\overline{T} \circ i_A^\mr, \overline{T} \circ i_G^\mr)$ is involutive.
\end{theorem}

\begin{proof}
The statements concerning invariant subspaces, intertwiners and involutions have already been proven in Proposition~\ref{p:from_crossed_product_to_dyn_sys}.

Denote $\pi := \overline{T} \circ i_A^\mr$ and $U := \overline{T} \circ i_G^\mr$. Proposition~\ref{p:from_crossed_product_to_dyn_sys} asserts that $(\pi,U)$ is a continuous covariant representation, hence its integrated form $\intform:C_c(G,A)\to B(X)$ is defined.
We claim that, for all $f\in C_c(G,A)$,
\begin{equation}\label{e:claim}
\intform(f)=T(\qr(f)).
\end{equation}
Assuming this for the moment, we see that, for all $f\in C_c(G,A)$, $\norm{\intform(f)}=\norm{T(\qr(f))}\leq\norm{T}\normr{\qr(f)}=\norm{T}\sigmar(f)$. Hence $(\pi,U)$ is $\mr$-continuous, and consequently the corresponding bounded representation $(\intform)^\mr:\crosprod\to B(X)$ can indeed be defined and we conclude, using the definition and \eqref{e:claim}, that $(\intform)^\mr(\qr(f))=\intform(f)=T(\qr(f))$, for all $f\in C_c(G,A)$. By the density of $\qr(C_c(G,A))$ in $\crosprod$, this implies that $(\intform)^\mr=T$. The statement concerning the norms then follows from Theorem~\ref{t:summary_from_dyn_sys_to_crossed_product}.

Hence it remains to establish \eqref{e:claim}. For this, let $f,\,g\in C_c(G,A)$. Then \eqref{e:compatibility} implies that
\begin{align*}
U_s \circ T(\qr(g)) &= \overline T(i_G^\mr(s)) \circ T(\qr(g)) \\
&= T(i_G^\mr(s)(q^\mr(g))) \\
&= T\left(\qr(i_G(s)g)\right).
\end{align*}
Similarly, we have, for all $a \in A$, and $h\in C_c(G,A)$,
\[ \pi(a) \circ T(\qr(h)) = T \left(\qr(i_A(a)h)\right)x. \]
Combining these, we find that, for $s \in G$,
\begin{align*}
\pi(f(s)) U_s \circ T(\qr(g)) &= \pi(f(s)) \circ T\left(\qr(i_G(s)g)\right) \\
&= T \left( \qr(i_A(f(s)) i_G(s)g) \right) \\
&= T \left(i_A^\mr(f(s)) i_G^\mr(s) \qr(g) \right).
\end{align*}
We conclude that, for all $f,\,g\in C_c(G,A)$,
\begin{equation}\label{e:first_expression}
\intform(f) \circ T(\qr(g)) = \int_G T \left( i_A^\mr(f(s)) i_G^\mr(s)\qr(g) \right) \,ds.
\end{equation}
On the other hand, with $\lambda$ denoting the left regular representation of $\crosprod$, Theorem~\ref{t:integrated_form_of_representation_as_left_centralizers_is_canonical} implies that, for $f,\,g\in C_c(G,A)$,
\begin{align*}
\qr(f) * \qr(g)&=\lambda(\qr(f)) \qr(g) \\
&= (i_A^\mr \rt i_G^\mr)^\mr(\qr(f)) \qr(g) \\
&= i_A^\mr \rt i_G^\mr(f) \qr(g) \\
&=\int_G i_A^\mr(f(s))i_G^\mr(s) \qr(g) \,ds.
\end{align*}
Applying the bounded homomorphism $T$ to this relation yields
\begin{equation}\label{e:second_expression}
T(\qr(f))\circ T(\qr(g))= \int_G T\left(i_A^\mr(f(s))i_G^\mr(s)\qr(g)\right)\,ds,
\end{equation}
for all $f,\,g\in C_c(G,A)$. For $x \in X$, comparing \eqref{e:first_expression} and \eqref{e:second_expression} and applying them to $x$, we see that
\begin{equation}\label{e:equality_of_expressions}
\intform(f) \left( T(\qr(g))x \right) = T(\qr(f)) \left( T(\qr(g))x \right).
\end{equation}
Now since $T$ is a non-degenerate bounded representation of $\crosprod$, the restriction of $T$ to the dense subalgebra $\qr(C_c(G,A)$ must be non-degenerate as well, and so elements of the form $T(\qr(g))x$ are dense in $X$. Hence \eqref{e:equality_of_expressions} implies that $\intform(f) = T(\qr(f))$ holds for all $f\in C_c(G,A)$, as desired.
\end{proof}

\section{Representations: general correspondence}\label{sec:general_correspondence}

In this section, which can be viewed as the conclusion of the analysis in the preceding parts of this paper, we put the pieces together without too much extra effort. We give references to the relevant definitions, in order to enhance accessibility of the results to the reader who is not familiar with the details of the Sections~\ref{sec:preliminaries} through~\ref{sec:from_crosprod_to_dynsys}. Section~\ref{sec:correspondences} contains some applications.

As an introductory remark for the reader who is familiar with the preceding sections, we note that Theorem~\ref{t:summary_from_dyn_sys_to_crossed_product} describes the properties of the passage from $\mr$-continuous covariant representations of $\dynsys$ to bounded representations of $\crosprod$. Such a passage is always possible, without further assumptions on the Banach algebra dynamical system or the covariant representations. Proposition~\ref{p:from_crossed_product_to_dyn_sys}, valid under the condition that $A$ has a bounded approximate left identity, goes in the opposite direction, but it is only for non-degenerate bounded representations of $\crosprod$ that a (non-degenerate) continuous covariant representation of $\dynsys$ is constructed. If one starts with an non-degenerate $\mr$-continuous covariant representation of $\dynsys$, passes to the associated non-degenerated bounded representation of $\crosprod$, and then goes back to $\dynsys$ again, the same Proposition~\ref{p:from_crossed_product_to_dyn_sys} shows that one retrieves the original covariant representation of $\dynsys$. If, in addition, all elements of $\mr$ are themselves non-degenerate, then Proposition~\ref{p:from_crossed_product_to_dyn_sys} can be improved to Theorem~\ref{t:r_affiliated_from_crossed_product_to_dyn_sys}, where it is concluded that the (non-degenerate) continuous covariant representation of $\dynsys$ as constructed from a non-degenerate bounded representation of $\crosprod$ is actually $\mr$-continuous. Hence it is possible to go in the first direction again, thus obtaining a bounded representation of $\crosprod$, and, according to the same Theorem~\ref{t:r_affiliated_from_crossed_product_to_dyn_sys}, this is the representation of $\crosprod$ one started with. As it turns out, if we impose these conditions on $A$ (having a bounded approximate left identity) and $\mr$ (consisting of non-degenerate continuous covariant representations), and restrict our attention to \emph{non-degenerate} $\mr$-continuous covariant representations of $\dynsys$ and \emph{non-degenerate} bounded representations of $\crosprod$, then we obtain a bijection, according to our main general result, the general correspondence in Theorem~\ref{t:bijection} below.

We now turn to the formulation of the result, recalling the relevant notions and definitions as a preparation, and introducing two new notations for (covariant) representations of a certain type. If $\dynsys$ is a Banach algebra dynamical system (Definition~\ref{d:banach_algebra_dynamical_system}), $\mr$ is a non-empty uniformly bounded (Definition~\ref{d:uniformly_bounded_class}) class of non-degenerate continuous covariant representations (Definition~\ref{d:covariant_representation}) of $\dynsys$, and $\bclass$ is a non-empty class of Banach spaces, we let $\Covreprndc(\dynsys,\bclass)$ denote the non-degenerate $\mr$-continuous (Definitions~\ref{d:crossed_product} and~\ref{d:affiliated}) representations of $(A,G,\alpha)$ in spaces from $\bclass$, and we let $\Repndb(\crosprod,\bclass)$ denote the non-degenerate bounded representations of the crossed product $\crosprod$ (Definition~\ref{d:crossed_product}) in spaces from $\bclass$. There need not be a relation between the representation spaces corresponding to the elements of $\mr$ and the spaces in $\bclass$.

Furthermore, we let $\integrate$ denote the assignment $\covrep \to (\intform)^\mr$, sending an $\mr$-continuous covariant representation of $\dynsys$ to a bounded representation of $\crosprod$, as explained following Remark~\ref{r:affiliated_is_bounded}. If $A$ has a bounded approximate left identity, then we let $\separate$ denote the assignment $T\to (\overline{T} \circ i_A^\mr, \overline{T} \circ i_G^\mr)$, as in Proposition~\ref{p:from_crossed_product_to_dyn_sys} and Theorem~\ref{t:r_affiliated_from_crossed_product_to_dyn_sys}, sending a non-degenerate bounded representation of $\crosprod$ to a non-degenerate continuous covariant representation of $\dynsys$, obtained by first constructing a non-degenerate bounded representation $\overline T$ of the left centralizer algebra of $\crosprod$, compatible with $T$, and subsequently composing this with the canonical continuous covariant representation $(i_A^\mr,i_A^\mr)$ of $\dynsys$ in $\crosprod$, the images of which actually lie in this left centralizer algebra (see Proposition~\ref{p:covariant_representation_on_crossed_product}).

The notations $\integrate$ and $\separate$ are meant to suggest ``integration" and ``separation", respectively.

Finally, we recall the notions of an involutive Banach algebra dynamical system (Definition~\ref{d:banach_algebra_dynamical_system}), of an involutive representation of such a system (Definition~\ref{d:covariant_representation}), and of bounded intertwining operators between (covariant) representations (final part of Section~\ref{subsec:banach_algebra_dynamical_systems}).

\begin{theorem}[General correspondence theorem]\label{t:bijection}
Let $\dynsys$ be a Banach algebra dynamical system, where $A$ has a bounded approximate left identity, let $\mr$ be a non-empty uniformly bounded class of non-degenerate continuous covariant representations of $\dynsys$, and let $\bclass$ be a non-empty class of Banach spaces.
Then the restriction of $\integrate$ yields a bijection
\[
\integrate:\Covreprndc(\dynsys,\bclass)\to\Repndb(\crosprod,\bclass),
\]
and the restriction of $\separate$ yields a bijection
\[
\separate:\Repndb(\crosprod,\bclass)\to\Covreprndc(\dynsys,\bclass),
\]
In fact, these restrictions of $\integrate$ and $\separate$ are inverse to each other.

Furthermore, both these restrictions of $\integrate$ and $\separate$ preserve the set of closed invariant subspaces for an element of their domain, as well as the Banach space of bounded intertwining operators between two elements of their domain.

If $\dynsys$ and $\mr$ are involutive, then both these restrictions of $\integrate$ and $\separate$ preserve the property of being involutive.
\end{theorem}

\begin{proof}
According to Theorem~\ref{t:summary_from_dyn_sys_to_crossed_product} and Theorem~\ref{t:r_affiliated_from_crossed_product_to_dyn_sys}, and also taking into account that $\integrate$ and $\separate$ obviously preserve the representation space, these restricted maps are indeed meaningfully defined with domains and codomains as in the statement.
Theorem~\ref{t:r_affiliated_from_crossed_product_to_dyn_sys} asserts that $\integrate(\separate(T)) = T$, for each non-degenerate bounded representation $T$ of $\crosprod$, whereas Proposition~\ref{p:from_crossed_product_to_dyn_sys} shows that $\separate(\integrate(\covrep))=\covrep$, for each non-degenerate $\mr$-continuous covariant representation $\covrep$ of $\dynsys$. This settles the bijectivity statements.

Theorem~\ref{t:summary_from_dyn_sys_to_crossed_product} and Theorem~\ref{t:r_affiliated_from_crossed_product_to_dyn_sys} contain the statements about preservation of closed invariant subspaces, intertwining operators and the property of being involutive.

\end{proof}

\begin{remark}\label{r:s_lim}
 The map $\separate$, associated with a Banach algebra dynamical system $\dynsys$, where $A$ has a bounded left approximate identity, and a non-empty uniformly bounded class $\mr$ of non-degenerate continuous covariant representations of $\dynsys$, can be made explicit by recalling how Theorem~\ref{t:summary_for_centralizers} was used in its definition. Indeed, let $T$ be a non-degenerate bounded representation of $\crosprod$ in a Banach space $X$. We recall the bounded approximate left identity $(\qr(z_V \otimes u_i))$ of $\crosprod$ of Theorem~\ref{t:crossed_product_has_bounded_approximate_identities}; here $(u_i)$ is a bounded approximate left identity of $A$, $V$ runs through a neighbourhood basis $\mathcal Z$ of $e\in G$, of which all elements are contained in a fixed compact subset of $G$, and $z_V\in C_c(G)$ is positive, with total integral equal to 1, and supported in $V$. Let $x\in X$. Then by \eqref{e:extension_as_sot_limit} we find, for $a\in A$,
\begin{align}\label{e:S_explicit_A}
 (\overline{T} \circ i_A^\mr)(a)x &= \overline{T}(i_A^\mr(a))x \\
&= \lim_{(V,i)} T \left[ i_A^\mr(a)\left( \qr(z_V \otimes u_i)\right) \right] x\notag \\
&= \lim_{(V,i)} T \left[ \qr(z_V \otimes a u_i) \right] x,\notag
\end{align}
and, for $r \in G$,
\begin{align}\label{e:S_explicit_G}
 (\overline{T} \circ i_G^\mr)(r)x &= \overline{T}(i_G^\mr(r))x \\
&= \lim_{(V,i)} T \left[ i_G^\mr(r) \left(\qr(z_V \otimes u_i)\right) \right] x\notag \\
&= \lim_{(V,i)} T \left[ \qr(z_V(r^{-1} \cdot) \otimes \alpha_r(u_i)) \right] x.\notag
\end{align}
Denoting s-lim for the limit in the strong operator topology, it follows that
\[ \separate(T) = \left( a \mapsto \mbox{s-}\lim_{(V,i)} T \left[ \qr(z_V \otimes a u_i) \right], r \mapsto \mbox{s-}\lim_{(V,i)} T \left [\qr(z_V(r^{-1} \cdot) \otimes \alpha_r( u_i)) \right] \right). \]
\end{remark}

\begin{remark}
Any non-degenerate continuous covariant representation $\covrep$ of $\dynsys$ in a Banach space $X$ is an element of $\Covreprndc(\dynsys,\bclass)$ with $\mr=\{\covrep\}$ and $\bclass=\{X\}$. If $A$ has a bounded approximate left identity, then, for $a \in A$, $r \in G$ and $x \in X$, inserting \eqref{e:S_explicit_A} and $\eqref{e:S_explicit_G}$ into $\separate(\integrate(\covrep))=\covrep$, as known from Theorem~\ref{t:bijection}, yields, with the $z_V\otimes u_i$ as in Remark~\ref{r:s_lim},
\begin{align*}
 \pi(a)x &= \left( \overline{\intformr} \circ i_A^\mr \right)(a) x \\
&= \lim_{(V,i)} \intformr \left[ \qr(z_V \otimes a u_i) \right] x \\
&= \lim_{(V,i)} \int_G z_V(s) \pi(a u_i) U_s x \ds, \\
U_r x &= \left( \overline{\intformr} \circ i_G^\mr \right)(r) x \\
&= \lim_{(V,i)} \intformr \left[ \qr(z_V(r^{-1} \cdot )  \otimes \alpha_r(u_i) ) \right] x \\
&= \lim_{(V,i)} \int_G z_V(r^{-1}s) \pi(\alpha_r(u_i)) U_s x \ds.
\end{align*}

These formulas, valid for an arbitrary non-degenerate continuous covariant representation $\covrep$ of $\dynsys$, where $A$ has a bounded approximate left identity, can also be obtained more directly, by writing $x = \lim_{(V,i)} \intform(z_V \otimes u_i) x$ (using Remark~\ref{r:strong_continuity_approx_identity}) and then using \eqref{e:intforms_of_i_j_A_G} in Proposition~\ref{p:i_j_A_G_properties}.
\end{remark}

\begin{remark}\label{r:constantdefinitions_estimates}
One also has norm estimates related to the maps $\integrate$ and $\separate$.

As to $\integrate$, if we assume that $\dynsys$ is a Banach algebra dynamical system, and that $\mr$ is a non-empty uniformly bounded class of continuous covariant representations of $\dynsys$, then, if  $\covrep$ is an $\mr$-continuous covariant representation of $\dynsys$, \eqref{e:sigmabound} yields that
\begin{align}\label{e:i_estimate}
\norm{\integrate(\covrep)\qr(f)}= \norm{\intform(f)} \leq \norm{\pi} \norm{f}_{L^1(G,A)} \sup_{s \in \supp(f)} \norm{U_s},
\end{align}
for all $f \in C_c(G,A)$.

In order to give estimates for $\separate$, we assume that $\dynsys$ is a Banach algebra dynamical system, with $A$ having an approximate left identity, and that  $\mr$ is a non-empty uniformly bounded class of continuous covariant representations of $\dynsys$. We recall the relevant constants
\begin{align*}
  \Cr = \sup_{\covrep \in \mr} \norm{\pi}, \quad \nur(r) = \sup_{\covrep \in \mr} \norm{U_r}, \quad \Nr = \inf_{V \in \mathcal{Z}} \sup_{r \in V} \nur(r),
\end{align*}
where $\mathcal{Z}$ is a neighbourhood basis of $e \in G$ which is contained in a fixed compact set (see Definition~\ref{d:Nr_definition} and the subsequent paragraph, showing that $\Nr$ does not depend on the choice of such $\mathcal Z$). Furthermore, if $\mathcal A$ is a normed algebra with a bounded approximate left identity, then we recall that $M_l^{\mathcal A}$ denotes the infimum of the upper bounds of the approximate left identities of ${\mathcal A}$, and that we write $M_l^{\mr}$ for $M_l^{\crosprod}$. Then, if $T$ is a non-degenerate bounded representation of $\crosprod$, Proposition~\ref{p:from_crossed_product_to_dyn_sys} shows that, for $a\in A$,
\begin{equation*}
\norm{\left(\overline{T} \circ i_A^\mr\right)(a)} \leq M_l^\mr \norm{T} \sup_{\covrep\in \mr} \norm{\pi(a)} \leq M_l^\mr \norm{T} \Cr \norm{a}.
\end{equation*}
In particular,
\begin{equation}\label{e:pi_bound_from_crossed_product_to_dyn_sys_in_general_correspondence_section}
\norm{\overline{T} \circ i_A^\mr}\leq M_l^\mr \Cr \norm{T}.
\end{equation}
Proposition~\ref{p:from_crossed_product_to_dyn_sys} also yields that, for $r\in G$,
\begin{equation}\label{e:pointwise_U_bound_from_crossed_product_to_dyn_sys_in_general_correspondence_section}
\norm{\left(\overline{T} \circ i_G^\mr\right)(r)} \leq M_l^\mr \norm{T}  \nur(r).
\end{equation}
Furthermore, by Corollary~\ref{c:approximate_identities},
\begin{equation}\label{e:upper_bound_estimate_for_bounded_approximate_left_identity}
M_l^\mr \leq \Cr M_l^A \Nr.
\end{equation}
\end{remark}

\section{Representations: special correspondences}\label{sec:correspondences}

In this section, we discuss some special cases of the crossed product construction, based on Theorem~\ref{t:bijection}. In the first part, we are concerned with a general algebra and group, and make the correspondence between (covariant) representations more explicit in a number of cases. In the second part, we consider Banach algebra dynamical systems where the algebra is trivial. This leads, amongst others, to what could be called group Banach algebras associated with a class of Banach spaces. The third part covers the case of a trivial group. Here the machinery as developed in the previous sections is not necessary, and Theorem~\ref{t:bijection}, although applicable, does, in fact, not yield optimal results. In this case the crossed product is merely the completion of a quotient of the algebra, and the correspondence between representations is then standard, but we have nevertheless included the results for the sake of completeness of the presentation.

\subsection{General algebra and group}\label{subsec:general_algebra_and_group}

Theorem~\ref{t:bijection} gives, for each class $\bclass$ of Banach spaces, a bijection between non-degenerate $\mr$-continuous covariant representations of $\dynsys$ and non-degenerate bounded representations of $\crosprod$ in spaces from $\bclass$. By definition, a continuous covariant representation $\covrep$ is $\mr$-continuous if there exists a constant $C$ such that $\norm{\intform (f)}\leq C\sup_{(\rho,V)\in\mr}\norm{\rho\rtimes V(f)}$, for all $f\in C_c(G,A)$. One would like to make this condition more explicit in terms of $\norm{\pi}$ and $\norm{U_r}$, for $r\in G$. For certain situations, this is indeed feasible (possibly by also restricting the maps $\integrate$ and $\separate$ in Theorem~\ref{t:bijection} to suitable subsets of their domains) on basis of the estimates in Remark~\ref{r:constantdefinitions_estimates}. Our basic theorem in this vein is the following.

\begin{theorem}\label{t:general_bijection_general_system}
Let $\dynsys$ be a Banach algebra dynamical system, where, for each $\eps >0$, $A$ has a $(1+\eps)$-bounded approximate left identity. Let $\mathcal{Z}$ be a neighbourhood basis of $e \in G$ contained in a fixed compact set, let $\nu: G \to [0, \infty)$ be bounded on compact sets and satisfy $\inf_{V \in \mathcal{Z}} \sup_{r \in V} \nu(r) = 1$. Let $\mr$ be a non-empty class of non-degenerate continuous covariant representations of $\dynsys$, such that, for $\covrep\in\mr$, $\pi$ is contractive and $\norm{U_r} \leq \nu(r)$, for all $r\in G$.

Let $\bclass$ be a class of Banach spaces, and suppose that $\mr$ contains the class $\mr^\prime$, consisting of all non-degenerate continuous covariant representations $\covrep$ of $\dynsys$ in spaces from $\bclass$, where $\pi$ is contractive and $\norm{U_r} \leq \nu(r)$, for all $r\in G$.

If $\mr^\prime$ is non-empty, then the map $\covrep \mapsto (\intform)^\mr$ is a bijection between $\mr^\prime$ and the non-degenerate contractive representations of $\crosprod$ in spaces from $\bclass$. This map preserves the set of closed invariant subspaces, as well as the Banach space of bounded intertwining operators between two elements of $\mr^\prime$.

If $\dynsys$ and $\mr$ are involutive, then this bijection preserves the property of being involutive.
\end{theorem}

\begin{proof}
We use Theorem~\ref{t:bijection}. Suppose that $\covrep \in \mr^\prime\subset\mr$, then certainly $\covrep$ is $\mr$-continuous, so that $\mr^\prime\subset \Covreprndc(\dynsys,\bclass)$. Hence the results of that theorem are applicable, and we must show that $\integrate$ and $\separate$ are bijections between $\mr^\prime$ and the non-degenerate contractive representations of $\crosprod$ in the elements of $\bclass$, which form a subset of $\Repndb(\crosprod,\bclass)$. Suppose that $\covrep \in \mr^\prime\subset\mr$, then $\integrate(\covrep)=\intformr$ is obviously contractive by the very definition of $\sigmar$ and $\crosprod$ in Definition~\ref{d:crossed_product}. Conversely, suppose that $T$ is a non-degenerate contractive representation of $\crosprod$ in a space from $\bclass$. In the notation of Remark~\ref{r:constantdefinitions_estimates}, we have $\Cr\leq 1$, and $M_l^A\leq 1$. By definition of $\nur$ we have $\nur \leq \nu$, so the condition on $\nu$ implies that $\Nr\leq 1$. From \eqref{e:upper_bound_estimate_for_bounded_approximate_left_identity}, we then conclude that $M_l^\mr\leq 1$. Thus $\separate(T)=(\overline{T} \circ i_A^\mr, \overline{T} \circ i_G^\mr)$ is not only known to be a non-degenerate covariant representation of $\dynsys$ in a space from $\bclass$ by Theorem~\ref{t:bijection}, but in addition we know from \eqref{e:pi_bound_from_crossed_product_to_dyn_sys_in_general_correspondence_section} that $\norm{\overline{T} \circ i_A^\mr} \leq 1$, and from \eqref{e:pointwise_U_bound_from_crossed_product_to_dyn_sys_in_general_correspondence_section} that $\norm{(\overline{T} \circ i_G^\mr)(r)} \leq \nur(r) \leq \nu(r)$, for all $r\in G$, so $\separate(T) \in \mr^\prime$. This settles the main part of the present theorem, and the rest is immediate from Theorem~\ref{t:bijection}.
\end{proof}

As a particular case of Theorem~\ref{t:general_bijection_general_system}, we let $\mr$ and $\mr^\prime$ coincide, and we specialize to $\nu \equiv 1$. Note that this condition $\norm{U_r} \leq 1$, for all $r \in G$, is equivalent to requiring that $U$ is isometric. Thus we obtain the following.

\begin{theorem}\label{t:isometry_correspondence}
Let $\dynsys$ be a Banach algebra dynamical system, where, for each $\eps >0$, $A$ has a $(1+\eps)$-bounded approximate left identity. Let $\bclass$ be a class of Banach spaces, and let $\mr$ consist of all non-degenerate continuous covariant representations $\covrep$ of $\dynsys$ in spaces from $\bclass$, such that $\pi$ is contractive and $U_r$ is an isometry, for all $r\in G$.

If $\mr$ is non-empty, then the map $\covrep \mapsto (\intform)^\mr$ is a bijection between $\mr$ and the non-degenerate contractive representations of $\crosprod$ in spaces from $\bclass$. This map preserves the set of closed invariant subspaces, as well as the Banach space of bounded intertwining operators between two elements of $\mr$.

If $\dynsys$ and $\mr$ are involutive, then this bijection preserves the property of being involutive.
\end{theorem}

Specializing Theorem~\ref{t:isometry_correspondence} in turn to the involutive case yields the following. We recall from the third part of Remark~\ref{r:seminorm_remark} that $\crosprod$ is a $C^*$-algebra if $\dynsys$ and $\mr$ are involutive.

\begin{theorem}\label{t:involutive_correspondence}
Let $\dynsys$ be a Banach algebra dynamical system, where $A$ is a Banach algebra with bounded involution and where $G$ acts as involutive automorphisms on $A$. Assume that, for each $\eps >0$, $A$ has a $(1+\eps)$-bounded approximate left identity. Let $\mathcal H$ be a class of Hilbert spaces, and let $\mr$ consist of all non-degenerate continuous covariant representations $\covrep$ of $\dynsys$ in elements of $\mathcal H$, such that $\pi$ is contractive and involutive, and $U_r$ is unitary, for all $r\in G$.

If $\mr$ is non-empty, then the map $\covrep \mapsto (\intform)^\mr$ is a bijection between $\mr$ and the non-degenerate involutive representations of the $C^*$-algebra $\crosprod$ in spaces from $\mathcal H$. This map preserves the set of closed invariant subspaces, as well as the Banach space of bounded intertwining operators between two elements of $\mr$.
\end{theorem}

\begin{remark}\label{r:cstar_remark}
Note that Theorem~\ref{t:involutive_correspondence} applies to all $C^*$-dynamical systems, since then $A$ has a 1-bounded approximate left identity. In that case, if $\mathcal R$ is non-empty, then $\crosprod$ can be considered as the $C^*$-crossed product associated with the $C^*$-dynamical system $\dynsys$ and the Hilbert spaces from $\mathcal H$. If $\mathcal H$ consists of \emph{all} Hilbert spaces, then the associated $C^*$-algebra $\crosprod$ is commonly known as \emph{the} crossed product $A\rt_\alpha G$, as in \cite{williams}. Surely $\mr$ is then non-empty, since it contains the zero representation on the zero space. However, more is true: the Gelfand-Naimark theorem furnishes a faithful non-degenerate involutive representation of $A$ in a Hilbert space, and then \cite[Lemma~2.26]{williams} provides a covariant involutive representation of $\dynsys$ (which is non-degenerate by \cite[Lemma~2.17]{williams}), of which the integrated form is a faithful representation of $C_c(G,A)$. As a consequence, $\sigmar$ is then an algebra norm on $C_c(G,A)$, rather than a seminorm, and the quotient construction as in the present paper for the general case is then not necessary.
\end{remark}

We conclude this section with a preparation for the sequel \cite{crossedtwo}, where we will show that under certain conditions $\crosprod$ is (isometrically) isomorphic to the Banach algebra $L^1(G,A)$ with a twisted convolution product. With this in place, we will then also be able to show how well-known results about (bi)-modules for $L^1(G)$ (\cite[Assertion~VI.1.32]{helemskii}, \cite[Proposition~2.1]{johnson}) fit into the general framework of crossed products of Banach algebras.

The preparatory result we will then require is the following; the function $\nur$ figuring in it is defined in Remark~\ref{r:constantdefinitions_estimates}.

\begin{theorem}\label{t:L1_case}
 Let $\dynsys$ be a Banach algebra dynamical system, where $A$ has a bounded approximate left identity. Let $D \geq 0$, and let $\mr$ be a non-empty class of non-degenerate continuous covariant representations $\covrep$ of $\dynsys$, such that $\nur(r) \leq D$, for all $r\in G$. Assume that there exists $C_1\geq 0$ such that $\norm{f}_{L^1(G,A)} \leq C_1 \sup_{\covrep\in\mr}\norm{\intform(f)}$, for all $f\in C_c(G,A)$.

Let $\bclass$ be a class of Banach spaces. Then the map $\covrep \mapsto (\intform)^\mr$ is a bijection between the non-degenerate continuous covariant representations of $\dynsys$ in spaces from $\bclass$ for which there exists a constant $C_U$, such that $\norm{U_r}\leq C_U$, for all $r\in G$, and the non-degenerate bounded representations of $\crosprod$ in spaces from $\bclass$. This map preserves the set of closed invariant subspaces, as well as the Banach space of bounded intertwining operators between two elements of $\mr$.

If $\dynsys$ and $\mr$ are involutive, then this bijection preserves the property of being involutive.
\end{theorem}

\begin{proof}
 We apply Theorem~\ref{t:bijection} and show that $\integrate$ and $\separate$ map the sets of (covariant) representations as described in the present theorem into each other. Suppose $\covrep$ is a non-degenerate continuous covariant representations of $\dynsys$ in a space from $\bclass$ for which $U$ is uniformly bounded. Then by \eqref{e:sigmabound}, the assumption implies that, for all $f \in C_c(G,A)$,
\[ \norm{\intform(f)} \leq \norm{\pi} D \norm{f}_{L^1(G,A)} \leq \norm{\pi} D C_1 \sigmar(f), \]
and so $\intform$ is $\mr$-continuous and hence induces a non-degenerate bounded representation $\intformr$ of $\crosprod$.

Conversely, let $T$ be a non-degenerate bounded representation of $\crosprod$ in a space from $\bclass$. Then $(\overline{T} \circ i_A^\mr, \overline{T} \circ i_G^\mr)$ is not only known to be a non-degenerate continuous covariant representation $\covrep$ of $\dynsys$, by Theorem~\ref{t:bijection}, but in addition \eqref{e:pointwise_U_bound_from_crossed_product_to_dyn_sys_in_general_correspondence_section}, together with $\nur(r) \leq D$ for all $r \in G$, shows that $\overline{T} \circ i_G$ is uniformly bounded.
\end{proof}

\begin{remark}
Under the hypotheses of Theorem~\ref{t:L1_case}, \eqref{e:sigmabound} implies that we also have $\sigmar(f)\leq \Cr D \norm{f}_{L^1(G,A)}$, so that $\sigmar$ and $\norm{\,.\,}_{L^1(G,A)}$ are equivalent algebra norms on $C_c(G,A)$. As a consequence, $\crosprod$ and $L^1(G,A)$ are isomorphic Banach algebras, and the non-degenerate continuous covariant representations $\covrep$ in spaces from $\bclass$, as described in the theorem, are in bijection with the non-degenerate bounded representations of $L^1(G,A)$ in the elements of $\bclass$. The questions when the condition $\norm{f}_{L^1(G,A)} \leq C_1 \sigmar(f)$ is actually satisfied, and when $\crosprod$ and $L^1(G,A)$ are even isometrically isomorphic Banach algebras, will be tackled in the sequel (\cite{crossedtwo}), to which we also postpone further discussion.
\end{remark}

\subsection{Trivial algebra: group Banach algebras}\label{subsec:trivial_algebra}

We now specialize the results of Theorem~\ref{t:general_bijection_general_system} to the case where the algebra is equal to the field $\mathbb K$, and the group acts trivially on it. We start by making some preliminary remarks.

The general representation of $\mathbb K$ in a Banach space $X$ is given by letting $\lambda\in \mathbb K$ act as $\lambda P$, where $P\in B(X)$ is an idempotent. Therefore, the only non-degenerate representation of $\mathbb K$ in $X$ is the canonical one, $\textup{can}_X:\mathbb K\to B(X)$, obtained for $P=\id_X$. As a consequence, the non-degenerate covariant representations of $(\mathbb K,G,\triv)$ in a given Banach space $X$ are in bijection with the strongly continuous representations of $G$ in that Banach space, by letting $(\textup{can}_X,U)$ correspond to $U$. Likewise, the non-degenerate involutive continuous covariant representations of $(\mathbb K,G,\triv)$ in a given Hilbert space are in bijection with the unitary strongly continuous representations of $G$ in that Hilbert space.

The Banach algebra dynamical system $(\mathbb K,G,\triv)$ has a non-degenerate continuous covariant representation in each Banach space $X$, with $G$ acting as isometries, namely, by letting the field act as scalars and letting the group act trivially. Likewise, there is a non-degenerate involutive continuous covariant representation of $(\mathbb K,G,\triv)$ in each Hilbert space, with $G$ acting as unitaries. Therefore, the hypothesis in the theorems in Section~\ref{subsec:general_algebra_and_group} that certain classes of non-degenerate continuous covariant representations are non-empty is sometimes redundant. Furthermore, the hypothesis on the existence of a suitable bounded approximate left identity in $\mathbb K$ is obviously always satisfied.

We introduce some shorthand notation. If $\mr$ is a class of strongly continuous representations of $G$, then $\widetilde{\mr}:=\{(\textup{can}_{X_U},U) : U\in\mr\}$ is a uniformly bounded class of continuous covariant representations of $(\mathbb K,G,\triv)$ precisely if there exists a function $\nu:G\to [0,\infty)$, which is bounded on compact subsets of $G$, and such that $\norm{U_r}\leq \nu(r)$, for all $U\in\mr$, and all $r\in G$. In that case, the associated crossed product $(\mathbb K\rt_{\triv} G)^{\widetilde{\mr}}$ is defined, but we will write $(\mathbb K\rt_{\triv} G)^\mr$ for short. Thus $(\mathbb K\rt_{\triv} G)^\mr$ is obtained by starting with $C_c(G)$ in its usual convolution structure, introducing the seminorm
\[
\sigmar(f)=\sup_{U\in\mr}\norm{\int_G f(s)U_s\ds}\quad(f\in C_c(G)),
\]
and completing $C_c(G)/\textup{ker}(\sigmar)$ in the norm induced on this quotient by $\sigmar$. As before, we let $\qr$ denote the canonical map from $C_c(G)$ into $(\mathbb K\rt_{\triv} G)^\mr$. If $U$ is a strongly continuous representation of $G$, then we let $U(f)=\int_G f(s)U_s\ds$, which corresponds to $\textup{can}_{X_U}\rtimes U(f)$ in the previous sections. Then $U$ will be called $\mr$-continuous if there exists a constant $C$ such that $\norm{U(f)}\leq C\sigmar(f)$, for all $f\in C_c(G)$; this corresponds to $\textup{can}_{X_U}\rtimes U$ being $\widetilde R$-continuous. In that case, there is an associated bounded representation of $(\mathbb K\rt_{\triv} G)^\mr$, denoted by $U^\mr$ rather than $(\textup{can}_{X_U}\rtimes U)^\mr$, which is given on the dense subalgebra $\qr(C_c(G))$ of $(\mathbb K\rt_{\triv} G)^\mr$ by
\[
U^\mr(\qr(f)) = \int_G f(s) U_s \ds\quad (f \in C_c(G)).
\]

With these notations, Theorems~\ref{t:general_bijection_general_system} specializes to the following result.

\begin{theorem}\label{t:general_bijection_trivial_algebra}
Let $G$ be a locally compact group. Let $\mathcal{Z}$ be a neighbourhood basis of $e \in G$ contained in a fixed compact set, let $\nu: G \to [0, \infty)$ be bounded on compact sets and satisfy $\inf_{V \in \mathcal{Z}} \sup_{r \in V} \nu(r) = 1$. Let $\mr$ be a non-empty class of strongly continuous representations of $G$ on Banach spaces, such that, for $U\in\mr$, $\norm{U_r} \leq \nu(r)$, for all $r\in G$.

Let $\bclass$ be a class of Banach spaces, and suppose that $\mr$ contains the class $\mr^\prime$, consisting of all strongly continuous representations $U$ of $G$ in spaces from $\bclass$, such that $\norm{U_r}\leq \nu(r)$, for all $r\in G$.

Then the map which sends $U\in\mr^\prime$ to $U^\mr$ is a bijection between $\mr^\prime$ and the non-degenerate contractive representations of $(\mathbb K\rt_{\triv} G)^\mr$ in the Banach spaces from $\mathcal X$. This map preserves the set of closed invariant subspaces, as well as the Banach space of bounded intertwining operators between two elements of $\mr^\prime$.

If all elements from $\mr$ are unitary strongly continuous representations, then this bijection lets unitary strongly continuous representations of $G$ in $\mr^\prime$ correspond to involutive representations of the $C^*$-algebra $(\mathbb C\rt_{\triv} G)^\mr$.

\end{theorem}

Specializing the above result to $\mr=\mr^\prime$ and $\nu\equiv 1$ (or Theorem~\ref{t:isometry_correspondence} to the case of the trivial algebra), we obtain the following.

\begin{theorem}\label{t:group_algebra_isometric_case} Let $G$ be a locally compact group.
Let $\mathcal X$ be a non-empty class of Banach spaces, and let $\mathcal{R}$ be the class of all isometric strongly continuous representations of $G$ in spaces from $\mathcal X$. Then the map which sends $U\in\mathcal R$ to $U^\mr$ is a bijection between $\mathcal R$ and the non-degenerate contractive representations of the Banach algebra $(\mathbb K\rt_{\triv} G)^\mr$ in the Banach spaces from $\mathcal X$. This map preserves the set of closed invariant subspaces, as well as the Banach space of bounded intertwining operators.
\end{theorem}

The following is a consequence of Theorem~\ref{t:involutive_correspondence}.

\begin{theorem}\label{t:group_algebra_involutive_case}
Let $G$ be a locally compact group, let $\mathcal H$ be a non-empty class of Hilbert spaces, and let $\mathcal R$ consists of all unitary strongly continuous representations of $G$ in the Hilbert spaces from $\mathcal H$.

Then the map which sends $U\in\mathcal R$ to $U^\mr$ is a bijection between $\mathcal R$ and the non-degenerate involutive representations of the $C^*$-algebra $(\mathbb K\rt_{\triv} G)^\mr$  in the Hilbert spaces from $\mathcal H$. This map preserves the set of closed invariant subspaces, as well as the Banach space of bounded intertwining operators between two elements of $\mr$.
\end{theorem}

\begin{remark}
The Banach algebra in Theorem~\ref{t:group_algebra_isometric_case} could be called the group Banach algebra $\mathcal B_{\mathcal X}(G)$ of $G$ associated with the (isometric strongly continuous representations of $G$ in the) Banach spaces from $\mathcal X$. As explained in the Introduction, these algebras, and their possible future role in decomposition theory for group representations, were part of the motivation underlying the present paper. The $C^*$-algebra in Theorem~\ref{t:group_algebra_involutive_case} is the group Banach algebra $\mathcal B_{\mathcal H}(G)$, which in this case has additional structure as a $C^*$-algebra. If $\mathcal H$ consists of all Hilbert spaces, then the group Banach algebra $\mathcal B_{\mathcal H}(G)$ is, of course, what is commonly known as $C^*(G)$, ``the'' group $C^*$-algebra of $G$.
\end{remark}

\subsection{Trivial group: enveloping algebras}\label{subsec:trivial_group}

We conclude with a few remarks on the case where the group is equal to the trivial group, $\{ e \}$, acting trivially on $A$. In this situation $C_c(\{e\}, A) \cong A$ as abstract algebras, so, if $\mr$ is a class of representations of $A$ in Banach spaces, for which there exists a constant $C$, such that $\norm{\pi}\leq C$, for all $\pi\in\mr$, then one naturally associates a uniformly bounded class of continuous covariant representations of $(A,\{e\},\triv)$ with $\mr$, and, with obvious notational convention, constructs the crossed product $(A \rt_\triv \{e\})^\mr$. This crossed product is simply the completion of $A / \ker(\sigmar)$ in the norm corresponding to the seminorm $\sigmar$ on $A$, defined by $\sigmar(a)=\sup_{\pi \in \mr}\norm{\pi(a)}$, for $a\in A$.

In principle, one could apply Theorem~\ref{t:general_bijection_general_system} in this situation, but then one would need to require $A$ to have a bounded left approximate identity. The reason underlying this is that, in general, there are no homomorphisms of the algebra or the group into the crossed product, so that the most natural idea to obtain representations of algebra and group from a representation of the crossed product, namely, to compose a given representation of the crossed product with such homomorphisms, will not work in general. In our approach in previous sections, the left centralizer algebra of the crossed product, into which the algebra and group do map, provided an alternative, but then one needs a bounded left approximate identity of the algebra, in order to be able to construct a representation of the left centralizer algebra from a (non-degenerate) representation of the crossed product. In the present case of a trivial group, however, one needs only a homomorphism of the algebra $A$ into the crossed product, and since this is a completion of $A / \ker(\sigmar)$, this clearly exists and the machinery we had to employ in previous sections is now not required. Also, the non-degeneracy of representations (needed to construct representations of the left centralizer algebra) is no longer an issue. One simply applies Lemma~\ref{l:completion_associated_with_seminorm}, and thus obtains the following elementary and well-known theorem for the crossed product $(A \rt_\triv \{e\})^\mr$, which we include for the sake of completeness.

\begin{theorem}
Let $A$ be a Banach algebra, and let $\mr$ be a non-empty uniformly bounded class of representations of $A$ in Banach spaces. Let $\sigmar(a)=\sup_{\pi \in \mr}\norm{\pi(a)}$, for $a\in A$, denoted the associated seminorm, and let $A^\mr$ be the completion of $A/\ker(\sigmar)$ in the norm corresponding to $\sigmar$ on $A$.

Let $\bclass$ be a non-empty class of Banach spaces, and say that a bounded representation $\pi$ of $A$ in a space from $\bclass$ is $\mr$-continuous if there exists a constant $C$, such that $\norm{\pi(a)}\leq C\sigmar(a)$, for all $a\in A$. In that case, define the norm of $\pi$ as the minimal such $C$.

Then the $\mr$-continuous representations of $A$ in the spaces from $\bclass$ correspond naturally with the bounded representations of $A^\mr$ in spaces from $\bclass$. This correspondence preserves the norms of the representations, the set of closed invariant subspaces, and the Banach spaces of bounded intertwining operator.

If $A$ has a \textup{(}possibly unbounded\textup{)} involution, and if all elements of $\mr$ are involutive bounded representations, then this natural correspondence sets up a bijection between the $\mr$-continuous involutive bounded representations of $A$ in the Hilbert spaces in $\bclass$, and the involutive representations of the $C^*$-algebra $A^\mr$ in those spaces.
\end{theorem}

If $A$ is an involutive Banach algebra with an isometric involution and a bounded approximate identity, and $\mr$ consists of all involutive representations in Hilbert spaces, which is uniformly bounded since all such representations are contractive by \cite[Proposition~1.3.7]{dixmier}, then the crossed product $A^\mr$ is generally known as the enveloping $C^*$-algebra of $A$ as described in \cite[Section~2.7]{dixmier}.

\subsection*{Acknowledgements}
The authors thank Miek Messerschmidt for helpful comments and suggestions.

\bibliographystyle{amsplain}
\bibliography{crossedarticlebib}

\end{document}